\documentclass[11pt,xcolor=dvipsnames]{article}
\pdfoutput=1 
\usepackage[utf8x]{inputenc}
\usepackage[english]{babel}
\usepackage{graphicx}           
\usepackage{url}
\usepackage{eurosym}
\usepackage{makeidx}
\usepackage[plain]{algorithm}
\usepackage{algorithmic} 
\usepackage{color} 
\usepackage[margin=3.3cm]{geometry}
\usepackage{setspace}
\usepackage{datetime}
\usepackage{pifont}
\usepackage{tikz}
\usepackage{ulem}
\usepackage{hyperref}
\usepackage{subcaption, nicefrac}
\usepackage{rotating}
\usepackage{amsmath, amsthm,amsfonts,amssymb,latexsym,stmaryrd,mathabx}
\usepackage{multicol}
\usepackage[numbers]{natbib}

\newcommand{\PP}{\mathbb{P}}			
\newcommand{\RR}{\mathbb{R}}			
\newcommand{\NN}{\mathbb{N}}			

\newcommand{\umoins}{{\smash{u^{\raisebox{-1pt}{\scriptsize\scalebox{0.5}{$-$}}}}}}
\newcommand{\Sin}{{\mathbf s}_{\textrm{\tiny\rm in}}}
\newcommand{\klip}{k_{\text{lip}}}

\newcommand{\X}{\mathbb{N}^* \times (0,\bar s_1)}

\newcommand{\dif}{\mathrm{d}}

\renewcommand{\tilde}{\widetilde}
\renewcommand{\bar}{\widebar}

\newcommand{\Smin}{\underline{\mathcal{S}}}
\newcommand{\Smax}{\bar{\mathcal{S}}}
\newcommand{\SminK}{s_K} 
\newcommand{\SmaxK}{S_K} 
\newcommand{\barmu}{\mu(\bar s_1)}

\newcommand{\ens}[2]{\left[#1,#2\right]}
\newcommand{\ensco}[2]{\left[#1,#2\right)}

\newcommand{\phit}{\phi_\textsf{t}^{-1}} 
\newcommand{\phis}{\phi_{\textsf{s}_0}^{-1}} 
\newcommand{\lmax}{{L_K}}
\newcommand{\setunL}{\llbracket 1, L\rrbracket}
\newcommand{\setunlmax}{\llbracket 1, \lmax\rrbracket}

\newtheorem{theorem}      {Theorem}[section]
\newtheorem{theorem*}     {theorem}
\newtheorem{proposition}  [theorem]{Proposition}

\newtheorem{lemma}        [theorem]{Lemma}

\newtheorem{remark}       [theorem]{Remark}

\newtheorem{corollary}    [theorem]{Corollary}

\newtheorem{hypothesis}   [theorem]{Assumption}

\newcommand{\Text}{T_{\text{Ext}}}
\newcommand{\1}{\mathbf{1}} 

\begin{document}

\title{Quasi-stationary behavior for an hybrid model of chemostat: the Crump-Young model}
\author{Bertrand Cloez$^{1}$ \and Coralie Fritsch$^{2}$}
\date{}

\footnotetext[1]{MISTEA, Université Montpellier, INRAE, Institut Agro, Montpellier, France}
\footnotetext[2]{Universit\'e de Lorraine, CNRS, Inria, IECL, UMR 7502, F-54000 Nancy, France}

\maketitle

\begin{abstract}
The Crump-Young model consists of two fully coupled stochastic processes modeling the substrate and micro-organisms dynamics in a chemostat.  Substrate evolves following an ordinary differential equation whose coefficients depend of micro-organisms number. Micro-organisms are modeled though a pure jump process whose the jump rates depend on the substrate concentration.  

It goes to extinction almost-surely in the sense that micro-organism population vanishes. In this work, we show that, conditionally on the non-extinction, its distribution  converges exponentially fast to a quasi-stationary distribution.

Due to the deterministic part, the dynamics of the Crump-Young model is highly degenerated. The proof is then original and consists of technical sharp estimates and new approaches for the quasi-stationary convergence.\\ 


\tableofcontents
\end{abstract}

\section{Introduction}
\label{sec:intro}

The evolution of bacteria in a bioreactor is usually described by a set of ordinary differential equations derived from a mass balance principle, see \cite{SW95,HLRS}. However, in 1979, Kenny S. Crump and Wan-Shin C. O’Young introduced in \cite{CY79} a piecewise deterministic Markov process, as defined in \cite{D93}, to model such a population.

This model corresponds to two fully coupled processes $(X_t,S_t)_{t\geq 0}$ in which the nutrient concentration $(S_t)_{t\geq 0}$ evolves continuously, through a differential equation, while the bacteria population size $(X_t)_{t\geq 0}$ evolves as a time-continuous càdlàg jump process. More precisely, they are defined by the following mechanisms:

\begin{itemize}
\item \textit{bacterial division}: the process $(X_t)_t$ jumps from $X_t$ to $X_t +1$ at rate $ \mu(S_t) X_t$;
\item \textit{bacterial washout}: the process $(X_t)_t$ jumps from $X_t$ to $X_t -1$ at rate $ D X_t$;
\item \textit{substrate dynamics}: between the jumps of $(X_t)_t$, the continuous dynamics of $(S_t)_t$ are given by the following ordinary differential equation
\begin{align}
\label{eq:substrate}
S'_t = D\,(\Sin-S_t)-k\,\mu(S_t)\, X_t \,,
\end{align}
\end{itemize}
where $\mu:\mathbb{R}_+\to \mathbb{R}_+$ and $D,\,\Sin,\,k>0$ are the specific growth rate, the dilution rate of the chemostat, the input substrate concentration and the (inverse of) yield constant respectively. 

Formally, the generator of this Markov process is the operator $\mathcal{L}$ given by

\begin{align}
\nonumber
\mathcal{L} f(x,s) 
&= \left[D(\Sin-s) - k \mu(s)\,x\right] \partial_s f(x,s) + \mu(s)\, x\, \left(f(x+1,s)-f(x,s)\right)\\
\label{eq:generator}
&\qquad + D \,x\,\left(f(x-1,s)-f(x,s)\right),
\end{align}
for all  $x\in\NN$, $s\geq 0$ and $f\in\mathcal{C}^{0,1}(\mathbb{N}\times \mathbb{R}_+)$, with $\mathcal{C}^{0,1}(\mathbb{N}\times \mathbb{R}_+)$ the space of functions $f:\mathbb{N}\times \mathbb{R}_+ \to \mathbb{R}$ such that for $x\in\mathbb{N}$, $s\mapsto f(x,s)\in \mathcal{C}^{1}(\mathbb{R}_+)$. Since then, several versions have been introduced to complete chemostat modeling as for instance in \cite{CF15,FHC,CJL11, FRRS}. However, despite its simplicity and the number of studies on it ($e.g.$ \cite{CJL11,CMMS,CF17,W16}), the long-time behaviour of this process is not well understood. It is well known that, under good assumptions (and in particular under Assumption~\ref{hyp:mu} below), it becomes extinct almost surely in finite time (see \cite[Theorem~4 and Remark~7]{CF17} and \cite[Theorem~3.1]{CMMS}); namely 
\begin{equation}
\label{eq.def.Text}
\Text:=\inf\left\{t\geq 0 \ | \ X_t=0 \right\}<+\infty \quad \text{a.s.}
\end{equation}
In addition, in \cite{CMMS}, authors proved the existence of quasi-stationary distributions (QSD), that is stationary distributions for the process conditionally on not being extinct (see \eqref{eq:def_QSD} below), as well as some regularity properties of these QSD.
Nevertheless, the long-time behavior of the process before extinction (as defined in \cite{MV12,CMS13,VP13}) was, until now, unknown. In this work, we prove that there exists a unique QSD (existence was proved in \cite{CMMS}, but not uniqueness) as well as the exponentially fast convergence of the process $(X_t,S_t)_{t\geq 0}$ to this QSD.

Convergence to QSD is usually proved though Hilbert techniques \cite{CMS13,CCLMMS,VD91}. However, our process of interest is not reversible making these techniques difficult to deal with. To overcome this problem, we use recent results \cite{CV17,CV20,BCGM,CG20} which are a generalization of usual techniques to prove convergence to  stationary distribution \cite{MT12}. These techniques hold in non-Hilbert space or when existence of the principal eigenvector is not known. A drawback is that sharp estimates are needed on the paths such as bounds on hitting times.  These estimates are often obtained throught irreducibility properties, however proving irreducibility properties for piecewise deterministic processes is an active and difficult subject of research \cite{BLMZ15,BS19,BHS18,C16}. See for instance the surprisingly behavior of some piecewise deterministic Markov processes in \cite{LMR13,BLMZ14}. A main part of our proof is nevertheless based on such result.

However in our setting, the process $(X_t,S_t)_{t\geq 0}$ is not irreducible. Fixing the number of bacteria, the flow associated to the substrate dynamics has a unique equilibrium, which is never reached and is different from the equilibirum with another number of bacteria. This makes even more difficult the hitting time estimates which are fundamental for the QSD existence and convergence.

Finally, even if our model can be seen as very specific, our proof could be mimicked in others contexts and open then the doors for others applications where this type of processes have applications. Among many others, these include  applications in neuroscience \cite{GL15, PTW},  in genomics \cite{G12, H17} or in ecology \cite{C16}. 

\medskip

The paper is organized as follow. We establish our main results in Section~\ref{sec:main.result}: first we state the exponentially fast convergence of the process towards a unique QSD for initial distributions on a given subset of $\mathbb{N^*}\times \mathbb{R}_+$ (Theorem~\ref{th:main-qsd}) then we extend the convergence towards the QSD for any initial condition of the process in $\mathbb{N^*}\times \mathbb{R}_+$ (Corollary~\ref{co:yaglom}).
Section~\ref{se:proof} is devoted to the proof of Theorem~\ref{th:main-qsd}.  We begin by detailing the scheme of our proof establishing sufficient conditions for the convergence. These conditions, proved in Section~\ref{sec:QS.behavior}, are mainly based on hitting time estimates, established in Section~\ref{sect:irre}. 
Section~\ref{sec:proof.corollary} is devoted to the proof of Corollary~\ref{co:yaglom}. 
For a better readability of the main arguments of the proofs, we postpone technical results in two appendices.
The first one establishes bounds and monotony properties of the underlying flow associated to the substrate dynamics as well as some classical properties on the probability of jump events. The second one contains the proof of the above-mentioned hitting time estimates and some properties based on Lyapunov functions bounds.
We remind in a third appendix the useful results of  \cite{BCGM} and \cite{CV20}.

\section{Main results}
\label{sec:main.result}

In all the paper, we will make the following assumption. 

\begin{hypothesis}
\label{hyp:mu}
The specific growth rate $\mu:\mathbb{R}_+ \mapsto \mathbb{R}_+$ satisfies to following properties: $\mu\in\mathcal{C}^1(\mathbb{R_+})$ and is an increasing function such that $\mu(0)=0$ and $\mu(s)>0$ for all $s>0$. 
\end{hypothesis}

Under Assumption~\ref{hyp:mu}, we denote by $\bar s_1\in(0,\Sin)$ the unique solution (see Lemma~\ref{lem:defsbar}) of $D(\Sin-\bar s_1)-k\,\mu(\bar s_1) = 0$. It corresponds to the equilibrium substrate concentration in the chemostat when the bacterial population is constant and contains only one individual.

\medskip

For any distribution $\xi$ on the space $E$, with $E=\mathbb{N}^*\times (0, \bar s_1)$ or $E=\mathbb{N}^*\times \mathbb{R}_+$, and any function $f:E\to\mathbb{R}$, we will denote by $\xi(f)$ the integral of $f$ w.r.t to $\xi$ on $E$, that is $\xi(f):=\int_E f(x,s)\,\xi(\dif x, \dif s)$.

\medskip

Our main result states the existence, uniqueness and exponential convergence to a quasi-stationary distribution (QSD). Recall that a QSD $\pi$, for the process $(X_t,S_t)_t$, is a probability measure on $\mathbb{N}^*\times \mathbb{R}_+$ such that
\begin{equation}
\label{eq:def_QSD}
\mathbb{P}_\pi \left( (X_t,S_t) \in \cdot \ | \ \Text>t\right) = \pi,
\end{equation}
with, for all probability measure $\xi$ on $\mathbb{N}^*\times \mathbb{R}_+$, $\mathbb{P}_\xi (\cdot) = \int_{\mathbb{N}^*\times \mathbb{R}_+} \mathbb{P}_{(x,s)} (\cdot) \xi(\dif x,\dif s)$ where $\mathbb{P}_{(x,s)}$ classically designs the probability conditioned to the event $\{ (X_0,S_0) = (x,s)\}$. The associated expectations are denoted by  $\mathbb{E}_{\xi}$ and $\mathbb{E}_{(x,s)}$.

From \cite[Proposition~2]{MV12} or \cite[Theorem~2.2]{CMS13}, if $\pi$ is a QSD, there exists a positive number $\lambda\geq 0$ such that
\begin{equation}
\label{eq:lambda}
\mathbb{P}_\pi\left(\Text >t\right) = e^{-\lambda t}.
\end{equation}

Roughly, the distribution $\pi$ represents the asymptotic law of $(X_t,S_t)$ before extinction and $1/\lambda$ is the mean of the extinction time.

For $\rho>1$ and $p>0$,  let define for all $(x,s)\in \mathbb{N}^*\times (0,\bar s_1)$
\[
W_{\rho,p}:(x,s) \mapsto \rho^x + \frac{1}{s}+ \frac{1}{(\bar s_1-s)^{p}}\qquad \text{and} \qquad  \psi: (x,s) \mapsto x.
\]

\begin{theorem}
\label{th:main-qsd}
We assume that $\mu(\bar s_1)>D$.  Then there exists a unique QSD $\pi$ on $\mathbb{N}^*\times (0,\bar s_1)$ such that there exist $\rho>1$ and $p\in\left(0,\frac{\mu(\bar s_1)-D}{D+k\,\mu'(\bar s_1)}\right)$ satisfying $\pi(W_{\rho,p})<+\infty$.  Moreover, for all $\rho>1$ and $p\in\left(0,\frac{\mu(\bar s_1)-D}{D+k\,\mu'(\bar s_1)}\right)$, there exists $C,\omega>0$ (depending on $\rho$ and $p$) such that for any starting distribution $\xi$ on  $\mathbb{N}^*\times (0, \bar s_1)$ such that $\xi(W_{\rho,p})<+\infty$,  and for all $t\geq 0$, we have 

\begin{equation}
\label{eq:cv_qsd}
\sup_{\|f \|_\infty \leq 1 } \left|\mathbb{E}_\xi \left[ f(X_t,S_t)  \ | \ \Text>t\right] - \pi(f) \right| \leq C \min\left( \frac{\xi(W_{\rho,p})}{\xi(\psi)}, \frac{\xi(W_{\rho,p})}{\xi(h)} \right) e^{-\omega t}
\end{equation} 
and
\begin{equation}
\label{eq:cv_qsd_X_eq_h}
\sup_{\|f \|_\infty \leq 1 } \left|e^{\lambda \,t}\,\mathbb{E}_\xi\left[f(X_t,S_t)\,\1_{X_t\neq 0}\right]-\xi(h)\,\pi(f)\right|\leq C\,\xi(W_{\rho,p})\,e^{-\omega\,t}\,,
\end{equation}
where $h$ defined for every $(x,s) \in \NN^*\times(0, \bar s_1)$ by
\begin{align}
\label{def:h}
h(x,s) := \lim_{t\to \infty} e^{\lambda t}\mathbb{P}_{(x,s)} (\Text>t)\in (0,+\infty),
\end{align}
is such that $\sup_{\X}h/W_{\rho,p}<\infty$
and where $\lambda$, the eigenvalue associated to $\pi$ ,defined by \eqref{eq:lambda}, satisfies
\begin{align}
\label{eq:eigenvalue.bounds}
0< \lambda \leq D.
\end{align}
\end{theorem}

In addition to Theorem~\ref{th:main-qsd},  several properties that we will not detail here but which can be useful in practice (bounds on $h$, spectral properties, definition of the Q-process...) can be deduced from \cite{CV20,BCGM}.  As the main objective of our paper is to give a method to verify that results of \cite{CV20,BCGM} hold for hybrid processes with a continuous component, we do not list these consequences here, but they can be easily founded in \cite{CV20,BCGM}.

Properties developed in \cite{CMMS} hold true for $\pi$ (density of the measure $\pi(x,.)$ w.r.t. the Lebesgue measures, differentiability of the density...).

Moreover, this unique QSD verifies the so-called Yaglom limit on $\mathbb{N^*}\times\mathbb{R}_+$ as stated in the next corollary; that is the dynamics conditioned on the non-extinction still tends to $\pi$ when the starting distribution is a Dirac masses on $\mathbb{N^*}\times\mathbb{R}_+$ .

\begin{corollary}
\label{co:yaglom}
We assume that $\mu(\bar s_1)>D$. Let $\pi$ as defined in Theorem~\ref{th:main-qsd}. For every $(x,s) \in \mathbb{N}^* \times \mathbb{R}_+$ and bounded function $f:\mathbb{N^*}\times \mathbb{R_+}\to \mathbb{R}$, we have
\[
\lim_{t\to \infty} \mathbb{E}_{(x,s)} \left[ f(X_t,S_t)  \ | \ \Text>t\right] = \pi(f).
\]
\end{corollary}

\begin{remark}
\label{remark:mu.lipschitz}
Assuming that $\mu$ is locally Lipschitz instead of $\mu\in \mathcal{C}^1(\mathbb{R}_+)$ is sufficient to obtain the convergences established in Theorem~\ref{th:main-qsd} and Corollary~\ref{co:yaglom}. The condition $p\in\left(0,\frac{\mu(\bar s_1)-D}{D+k\,\mu'(\bar s_1)}\right)$ then becomes $p\in\left(0,\frac{\mu(\bar s_1)-D}{D+k\,\klip}\right)$ for any local Lipschitz constant $\klip$ in a neighborhood of $\bar s_1$. See the end of Sections~\ref{sec:proof.lem.LV} and \ref{sec:proof.corollary}.
\end{remark}

We will see that the process $(X_t,S_t)_{t\geq 0}$ is not irreducible on $ \mathbb{N}^* \times (0,+\infty)$. In general, such non-irreducible process may have several quasi-stationary distributions and the convergence to them depends on the initial condition of the process; see for instance the \textbf{Bottleneck effect and condition H4} part of \cite[Section 3.1]{BCP}. 
In our setting, we will show, thanks to bounds on Lyapunov functions method, that the convergence holds for any initial distribution on $ \mathbb{N}^* \times (0,+\infty)$ because $\mathbb{N}^* \times (0,\bar s_1)$ is attractive.

\section{Proof of Theorem~\ref{th:main-qsd}}
\label{se:proof}

Similarly to the proof of \cite[Proposition~2.1 and Corollary~3.1]{CMMS},we can show that $\mathbb{N}\times (0,\Sin)$ is an invariant set for $(X_t,S_t)_{t\geq 0}$ and that $\mathbb{N^*}\times (0,\bar s_1)$ is an invariant set for $(X_t,S_t)_{t\geq 0}$ until the extinction time $\Text$. 
Consequently, for any starting distribution $\xi$ on  $\mathbb{N}^*\times (0, \bar s_1)$, the process evolves in $(\X)\cup (\{0\}\times (0,\Sin))$, with $\{0\}\times (0,\Sin)$ the absorbing set corresponding to the extinction of the process.

\medskip

Let fix $\rho>1$ and $p\in\left(0,\frac{\mu(\bar s_1)-D}{D+k\,\mu'(\bar s_1)}\right)$. We will prove that \cite[Theorem 5.1]{BCGM} and \cite[Corollary 2.4]{CV20} (that we recall in Appendix, see Theorems~\ref{th:BCGM} and \ref{th:CV20}) apply to the continuous semigroup  $(M_t)_t$ defined by
\[
	M_t f(x,s) := \mathbb{E}_{(x,s)}\left[f(X_t, S_t)\,\1_{X_t\neq 0}\right]
\]
for $(x,s)\in \mathbb{N}^*\times (0, \bar s_1)$ and $f:\mathbb{N}^*\times (0, \bar s_1) \to \mathbb{R}$ such that $\sup_{(x,s)\in \mathbb{N}^*\times (0, \bar s_1)}\frac{|f(x,s)|}{V(x,s)}<\infty$, where $V$ defined below is such that $c_1 \,W_{\rho,p} \leq V \leq c_2\, W_{\rho,p}$ for $c_1,c_2>0$.
Theorem~\ref{th:main-qsd} is then a combination of this both results whose the former gives the bound $\xi(W_{\rho,p})/\xi(h)$ whereas the latter gives the bound $\xi(W_{\rho,p})/\xi(\psi)$ in \eqref{eq:cv_qsd}. Note that the reason for working with $V$ rather than $W_{\rho,p}$ is that the bound \eqref{tag.BLF1} below is easier to obtain.

Let us fix $\alpha$ and $\theta$ such that
\begin{align}
\label{hyp:theta_p}
\alpha \geq \frac{\rho-1}{k}, \quad
\theta>\frac{p(D+k\,\mu'(\bar s_1))+D}{\mu(\bar s_1)-(p(D+k\,\mu'(\bar s_1))+D)}>0
\end{align}
 and set, for all $(x,s)\in\mathbb{N^*}\times (0,\bar s_1)$ 
\begin{align}
\label{def:psi_V}
\psi:(x,s) \mapsto  x, \qquad V:(x,s) \mapsto \frac{\rho^x e^{\alpha s}}{\log(\rho)} + \frac{1}{s} + \frac{1+\1_{x\leq 1}\theta}{(\bar s_1-s)^{p}}.
\end{align} 
Note that $1\leq\psi \leq V$ on $\mathbb{N^*}\times (0,\bar s_1)$. 
For convenience, we extend the definition of $\psi$ on the absorbing set by $\psi(0,s)= 0$ for $s\in (0,\Sin)$ such that $\psi(X_t,S_t) \,\1_{X_t\neq 0}=\psi(X_t,S_t)$.

We will show that the following three properties are sufficient to prove Theorem~\ref{th:main-qsd} and we will then prove them.
\begin{enumerate}
\item\label{hyp:thm-gen-lyap} 
Bounds on Lyapunov functions: There exists $\eta>D$ and $\zeta>0$ such that, for all $(x,s)\in \X$ and $t\geq 0$,
\begin{align}
\label{tag.BLF1}
\mathbb{E}_{(x,s)}\left[ V(X_t,S_t) \,\1_{X_t\neq 0}\right] \leq e^{-\eta  t} V(x,s) + \zeta_t  \, \psi(x,s)\,,
\tag{BLF1}
\end{align}
\begin{align}
\label{tag.BLF2}
\mathbb{E}_{(x,s)}\left[ \psi(X_t, S_t) \right] \geq e^{-D\,t} \psi(x,s),
\tag{BLF2}
\end{align}
with $\zeta_t:=\zeta \frac{e^{(\barmu-D)t}}{\eta-D}$.
\item \label{hyp:thm-gen-smallset}
Small set assertion: for every $t>0$, for every subset $K:=\llbracket 1,N\rrbracket\times [\delta_1,\delta_2]\subset \X$, with $N\in \NN^*$ and $\delta_2>\delta_1>0$, there exists a probability measure $\nu$ such that $\nu(K)=1$, and $\epsilon>0$ satisfying
\begin{align}
\label{tag:SSA}
\forall (x,s) \in K, \qquad \mathbb{P}_{(x,s)}( (X_t, S_t) \in \cdot )\geq \epsilon \nu. 
\tag{SSA}
\end{align}

\item \label{hyp:thm-gen-ratio}  Mass ratio inequality: for every compact set $K$ of $\X$, 
 we have
\begin{align}
\label{tag:MRI}
\sup_{(x,s),(y,r) \in K} \sup_{t\geq 0} \frac{\mathbb{E}_{(x,s)}\left[ \psi(X_t, S_t) \right] }{\mathbb{E}_{(y,r)}\left[ \psi(X_t, S_t) \right] } < + \infty.
\tag{MRI}
\end{align}
\end{enumerate}

\bigskip

We first establish, in Section~\ref{sec:sufficient.conditions}, that the three properties above (Bounds on Lyapunov functions \eqref{tag.BLF1} and \eqref{tag.BLF2};  Small set assertion \eqref{tag:SSA} and Mass ratio inequality \eqref{tag:MRI})  are sufficient conditions to prove Theorem~\ref{th:main-qsd}.
This three properties are then proved in Section~\ref{sec:QS.behavior}.

We verify the bounds on Lyapunov functions through classical drift conditions on the generator (see for instance \cite[Section 2.4]{BCGM}). The originality of our approach comes from the proof of the small set assertion and the mass ratio inequality (as well as the associated consequences: QSD uniqueness and exponentially fast convergence).
Proofs of these two properties are based on irreducibility properties that we describe in Section~\ref{sect:irre}. Indeed, the small set assertion establishes that, with a positive probability $\epsilon$, every starting point leads the dynamics to the same location at the same time (ensuring then also aperiodicity). A natural way to prove such result is to prove that the measures $\delta_{(x,s)} M_t$ admit a density function  with respect to some reference measure (counting measure for fully discrete processes, Lebesgue measure for diffusion processes...) and show that theses densities possesses a common lower bound. Unfortunately, due to the deterministic part of the dynamics, the measures $\delta_{(x,s)} M_t$ keeps a Dirac mass part. Moreover, we need to show that it holds for any time $t>0$ which is difficult when the process is neither diffusive nor discrete. The mass ratio inequality means that the extinction time does not vary so much with respect to the initial condition. Again, it was shown in \cite{CG20} that these conditions can be reduced to hitting time estimates.  Once more a natural way to prove such result is to assume that $\delta_{(x,s)} M_t$ admit a density function  but with moreover a common upper bound (see for example \cite{berglund2012mixed}). Another way is to use the so-called Harnack inequalities which seem to not be established for hybrid type partial differential equations. To our knowledge there is no such result for quasi-stationary distribution for such hybrid process.

\subsection{Sufficient conditions for the proof of Theorem~\ref{th:main-qsd}}
\label{sec:sufficient.conditions}

We will show that \eqref{tag.BLF1}-\eqref{tag.BLF2};  \eqref{tag:SSA} and \eqref{tag:MRI} implies that conditions of \cite[Theorem 5.1]{BCGM} and \cite[Corollary 2.4]{CV20} hold.

\medskip

Let us first detail how these three properties imply Assumption~A by \cite{BCGM} (see Assumption~\ref{hyp:BCGM}) on $\X$.
First \eqref{tag.BLF1} implies that for all $(x,s)\in \X$ and for all $t\geq 0$, $\mathbb{E}_{(x,s)}\left[ V(X_t, S_t)\,\1_{X_t\neq 0}  \right]\leq (e^{-\eta\,t}+\zeta_t)V(x,s)$ and then $(M_t)_{t\geq 0}$ actually acts on functions $f:\mathbb{N}^*\times (0, \bar s_1) \to \mathbb{R}$ such that $\sup_{(x,s)\in \mathbb{N}^*\times (0, \bar s_1)}\frac{|f(x,s)|}{V(x,s)}<\infty$.

Let $\tau>0$ and $K_R:=\{(x,s)\in \X, \, V(x,s)\leq R \psi(x,s) \}$, with $R$ chosen sufficiently large such that $K_R$ is non empty and such that $R>\frac{\zeta_\tau}{e^{-D\,\tau}-e^{-\eta\,\tau}}$, where $\eta>D$ and $\zeta>0$ are such that \eqref{tag.BLF1} holds. By definition of $V$ and $\psi$, we can easily show that $K_R$ is a compact set of $\X$. We choose $\delta_1,\delta_2>0$ and $N \in \mathbb{N}^*$ such that $K_R\subset K:=\llbracket 1,N\rrbracket\times [\delta_1,\delta_2]\in \X$. Then using the fact that $\psi \leq V/R$ on the complementary of $K_R$, for all $(x,s)\in \X$, we obtain from \eqref{tag.BLF1},
\begin{align*}
\mathbb{E}_{(x,s)}[V(X_\tau,S_\tau)\,\1_{X_\tau\neq 0}] 
& \leq \left( e^{-\eta \tau} + \frac{1}{R} \zeta_\tau \right) V(x,s)+  \zeta_\tau \,\1_{(x,s) \in K_R} \psi(x,s)\\
& \leq \left( e^{-\eta \tau} + \frac{1}{R} \zeta_\tau \right) V(x,s)+  \zeta_\tau \, \1_{(x,s)\in K} \psi(x,s),
\end{align*}
and the bound on $R$ ensures that 
\[\left( e^{-\eta \tau} + \frac{1}{R} \zeta_\tau \right)<e^{-D\,\tau}\,.\]
Consequently \eqref{tag.BLF1} and \eqref{tag.BLF2} imply that Assumptions~(A1) and (A2) of \cite{BCGM} are satisfied.

From \eqref{tag.BLF1} and the fact that $\1_K  \leq \psi\leq V$, for any positive function $f$ and $(x,s) \in K$, we have
\[
 \frac{\mathbb{E}_{(x,s)}\left[ f\left(X_\tau, S_\tau\right) \psi\left(X_\tau, S_\tau\right) \right]}{\mathbb{E}_{(x,s)}\left[ \psi\left(X_\tau, S_\tau\right)\right]} \geq \frac{1 }{\left(e^{-\eta \tau } +\zeta_\tau\right) \sup_K V} \mathbb{E}_{(x,s)}\left[ f(X_\tau, S_\tau) \1_{(X_\tau, S_\tau)\in K} \right],
\]
and then, as $K$ was chosen of the form $\llbracket 1,N\rrbracket\times [\delta_1,\delta_2]$, by \eqref{tag:SSA}, Assumption~(A3) of \cite{BCGM} is also satisfied.

Moreover \eqref{tag:MRI} ensures the existence of some constant $C\geq 1$ such that for every $(x,s), (y,r) \in K$ and $t\geq 0$, we have 
\begin{align*}
\frac{\mathbb{E}_{(x,s)}\left[ \psi(X_t, S_t) \right]}{\psi(x,s)}
&\leq  \mathbb{E}_{(x,s)}\left[ \psi(X_t, S_t) \right]
\leq C \mathbb{E}_{(y,r)}\left[ \psi(X_t, S_t) \right] \leq CN \frac{\mathbb{E}_{(y,r)}\left[ \psi(X_t, S_t) \right]}{\psi(y,r)},
\end{align*}
then integrating the last term w.r.t. $\nu(\dif y,\dif r)$ on $K$ leads to Assumption~(A4) of \cite{BCGM}.

Therefore Theorem~5.1 of \cite{BCGM} implies that there exist a unique QSD $\pi$ on $\X$ such that $\pi(V)<+\infty$, a measurable function $h:\X \to \mathbb{R}_+$ such that $\sup_{(x,s)\in\X}h(x,s)/V(x,s)<\infty$ and constants $\lambda, \, C',\, \omega'>0$ such that for any starting distribution $\xi$ on $\X$ such that $\xi(V)<+\infty$ and for all $t\geq 0$,  
\begin{equation}
\label{eq:cv_qsd_X}
\sup_{\|f \|_\infty \leq 1 } \left|\mathbb{E}_\xi \left[ f(X_t,S_t)  \ | \ \Text>t\right] - \pi(f) \right| \leq C '\frac{\xi(V)}{\xi(h)}  e^{-\omega' t}\,
\end{equation}
and 
\begin{equation}
\label{eq:cv_h_bcgm}
\sup_{\|f \|_\infty \leq 1 } \left|e^{\lambda \,t}\,\mathbb{E}_\xi\left[f(X_t,S_t)\,\1_{X_t\neq 0}\right]-\xi(h)\,\pi(f)\right|\leq C'\,\xi(V)\,e^{-\omega'\,t}\,.
\end{equation}
Taking $f\equiv \1$ and $\xi=\delta_{(x,s)}$ with $(x,s)\in\X$ in \eqref{eq:cv_h_bcgm} leads to the expression of $h$ given by \eqref{def:h} and choosing $\xi=\pi$ ensures  that $\lambda$ satisfies \eqref{eq:lambda}.
In addition \cite[Lemma~3.4.]{BCGM} ensures that $h>0$  on $\X$.
Moreover from \eqref{eq:def_QSD} and \eqref{eq:lambda}, for all $t\geq 0$, $\mathbb{E}_\pi\left[\psi(X_t,S_t)\right]=e^{-\lambda\,t}\pi(\psi)$, then integrating \eqref{tag.BLF2} with respect to $\pi$ gives the bounds~\eqref{eq:eigenvalue.bounds}.

\bigskip

Let us now detail how the three properties imply that $(M_{n\,\tau})_{n\in\mathbb{N}}$ satisfies Assumption~G by \cite{CV20} (see Assumption~\ref{hyp:CV20}).
We consider the same compact $K=\llbracket 1,N\rrbracket\times [\delta_1,\delta_2]$ as before.
By \eqref{tag:SSA}, for all $(x,s)\in K$ and all measurable $A\subset K$,
\[
	\mathbb{E}_{(x,s)}\left[V(X_\tau,S_\tau)\,\1_{X_\tau \neq 0}\,\1_{(X_\tau,S_\tau)\in A}\right]
	\geq \epsilon \int_A V(y,r)\,\nu(\dif y, \dif r)
	\geq \tilde \epsilon\, \nu(A)\,V(x,s)
\]
with $\tilde \epsilon:=\epsilon \frac{\inf_{(y,r)\in K}V(y,r)}{\sup_{(y,r)\in K}V(y,r)}>0$, then Assumption~(G1) of \cite{CV20} is satisfied.
Assumptions~(A1) and (A2) of \cite{BCGM} imply Assumption~(G2) of \cite{CV20}, then it holds. As for all $(y,r)\in K$, $1\leq \psi(y,r)\leq N$, then \eqref{tag:MRI} directly implies  Assumption~(G3) of \cite{CV20}. Moreover, as \eqref{tag:SSA} holds for all $t>0$, then Assumption~(G4) of \cite{CV20} is also satisfied. 
Finally, by \eqref{tag.BLF1} and \eqref{tag.BLF2}, for all $(x,s)\in \X$ and all $t\in[0,\tau]$,
\[
\frac{\mathbb{E}_{(x,s)}\left[ V(X_t,S_t) \,\1_{X_t\neq 0}\right]}{V(x,s)} \leq 1  + \zeta_\tau 
\qquad
\text{and}
\qquad
\frac{\mathbb{E}_{(x,s)}\left[ \psi(X_t, S_t) \right]}{\psi(x,s)} \geq e^{-D\,\tau} \,.
\]
Therefore, from \cite[Corollary 2.4.]{CV20},  there exists $C''>0$, $\omega''>0$ and a positive measure $\nu_P$ on $\X$ satisfying $\nu_P(V)=1$ and $\nu_P(\psi)>0$ such that for any starting distribution $\xi$ on  $\X$ such that $\xi(V)<+\infty$, we have  
\begin{align}
\label{eq:borne_CV}
 \sup_{\|f \|_\infty \leq 1 }   \left| \frac{\xi M_t\,f}{\xi M_t\,V} -\nu_P(f)\right| \leq C''\,e^{-\omega''\,t} \frac{\xi(V)}{\xi(\psi)}, \quad \forall t\geq 0.
\end{align}
Following the same way as \cite[Proof of Corollary 3.7]{BCGM}, for all $f$ such that $\|f \|_\infty \leq 1$, from triangular inequality and as $\left|\frac{\nu_P(f)}{\nu_P(\1)}\right|\leq 1$, we have
\begin{align*}
\left| \frac{\xi M_t f}{\xi M_t\,\1} - \frac{\nu_P(f)}{\nu_P(\1)}\right|
	&\leq
			\frac{\xi M_t V}{\xi M_t\,\1}\,  \left(
			\left| \frac{\xi M_t f}{\xi M_t\,V} - \nu_P(f)\right|
			+
			\left|\frac{\nu_P(f)}{\nu_P(\1)}\right|\,\left| \frac{\xi M_t \1}{\xi M_t\,V} - \nu_P(\1)\right| \right)\\
	&\leq
			\frac{\xi M_t V}{\xi M_t\,\1}\,  \left(
			\left| \frac{\xi M_t f}{\xi M_t\,V} - \nu_P(f)\right|
			+
			\left| \frac{\xi M_t \1}{\xi M_t\,V} - \nu_P(\1)\right| \right)\,.
\end{align*}
Applying \eqref{eq:borne_CV} first to $f$ and second to $\1$ gives 
 \begin{align*}
\left| \frac{\xi M_t f}{\xi M_t\,\1} - \frac{\nu_P(f)}{\nu_P(\1)}\right|
	&\leq
			\frac{\xi M_t V}{\xi M_t\,\1}\,  
			2\,C''\,e^{-\omega''\,t} \frac{\xi(V)}{\xi(\psi)}\,.
\end{align*}
Moreover, \eqref{eq:borne_CV} applied to $\1$ also leads to
\[
	 \frac{\xi M_t\,\1}{\xi M_t\,V} \geq \nu_P(\1) -  C''\,e^{-\omega''\,t} \frac{\xi(V)}{\xi(\psi)}\,,
\]
then for $t\geq \frac{1}{\omega''}\,\log\left(\frac{2\,C''\,\xi(V)}{\nu_P(\1)\,\xi(\psi)}\right)$ we have $\frac{\xi M_t\,\1}{\xi M_t\,V} \geq \frac{\nu_P(\1)}{2}$  and then
 \begin{align*}
\left| \frac{\xi M_t f}{\xi M_t\,\1} - \frac{\nu_P(f)}{\nu_P(\1)}\right|
	&\leq
			\frac{4}{\nu_P(\1)}\,  
			C''\,e^{-\omega''\,t} \frac{\xi(V)}{\xi(\psi)}\,.
\end{align*}
Furthermore, for $t\leq \frac{1}{\omega''}\,\log\left(\frac{2\,C''\,\xi(V)}{\nu_P(\1)\,\xi(\psi)}\right)$, we obtain
 \begin{align*}
 \left| \frac{\xi M_t f}{\xi M_t\,\1} - \frac{\nu_P(f)}{\nu_P(\1)}\right| \leq 2\leq \frac{4}{\nu_P(\1)}\,C''\,e^{-\omega''\,t}\,\frac{\xi(V)}{\xi(\psi)}\,.
\end{align*}
Therefore,
\begin{equation}
\label{eq:cv_qsd_cv}
\sup_{\|f \|_\infty \leq 1 } \left|\mathbb{E}_\xi \left[ f(X_t,S_t)  \ | \ \Text>t\right] - \frac{\nu_P(f)}{\nu_P(\1)} \right| \leq \frac{4}{\nu_P(\1)}\,  
			C''\,e^{-\omega''\,t} \frac{\xi(V)}{\xi(\psi)}\,.
\end{equation}

\medskip

Finally, from \eqref{eq:cv_qsd_X} and \eqref{eq:cv_qsd_cv}, we have $\pi = \frac{\nu_P}{\nu_P(\1)}$. Moreover, on $\X$, we have 
\[\min\{\log(\rho)^{-1},1\}\,W_{\rho,p} \leq V\leq \max\left\{1+\theta,\frac{e^{\alpha\,\bar s_1}}{\log{(\rho)}}\right\}\, W_{\rho,p},\]
then \eqref{eq:cv_qsd} and \eqref{eq:cv_qsd_X_eq_h} hold with $\omega=\min\{\omega',\omega''\}$ and $C=\max\left\{1+\theta,\frac{e^{\alpha\,\bar s_1}}{\log(\rho)}\right\}\,\max\left\{C',\frac{4\,C''}{\nu_P(\1)}\right\}$.  

Note that \eqref{eq:cv_qsd_X_eq_h} and \eqref{def:h}, which have been proved using \cite[Theorem 5.1]{BCGM}, could also have been proved using the second part of \cite[Corollary 2.4]{CV20}, where \eqref{eq:cor2.4.CV20.eta} holds with $\lambda_0=-\lambda$ and $\eta_P= \frac{h}{\nu_P(\1)}$.

\medskip 

The previous QSD $\pi=\pi_{\rho,p}$ depends on $\rho$ and $p$.  However, for any starting distribution $\xi$ on $\mathbb{N}^*\times (0, \bar s_1)$ such that $\xi(W_{\rho,p})<+\infty$ for all $\rho>1$, $p\in\left(0,\frac{\mu(\bar s_1)-D}{D+k\,\mu'(\bar s_1)}\right)$ (dirac measures on $\mathbb{N}^*\times (0, \bar s_1)$ for example), \eqref{eq:cv_qsd} gives that
\[
	\lim_{t\to \infty}\mathbb{P}_\xi \left[ (X_t,S_t)\in .  \ | \ \Text>t\right] = \pi_{\rho,p}
\]
then by uniqueness of the limit, all QSD indexed by $\rho$ and $p$ are the same.

\subsection{Additional notation}
\label{sec.add.notations}

For convenience, we extend the notation $\ens{s_1}{s_2}$ and $\ensco{s_1}{s_2}$ of respectively the closed and semi-open sets of values between $s_1$ and $s_2$, usual for $s_1\leq s_2$, to $s_1>s_2$, in the following way 
\[
	\ens{s_1}{s_2}=
	\begin{cases}
	[s_1,s_2] & \text{if $s_1\leq s_2$},\\
 	[s_2,s_1] & \text{if $s_1>s_2$},
	\end{cases}\,
	\qquad
	\ensco{s_1}{s_2}=
	\begin{cases}
	[s_1,s_2) & \text{if $s_1\leq s_2$},\\
 	(s_2,s_1] & \text{if $s_1>s_2$}.
	\end{cases}\,
\]

Let us begin by giving additional notation relative to flow associated to the ordinary differential equation \eqref{eq:substrate}; namely this concerns the case when the number of bacteria is constant, that is the behaviour between the population jumps.

For all $(\ell,s_0)\in \NN^*\times \RR_+$, let $t\mapsto \phi(\ell,s_0,t)$ be the flow function  associated to the substrate equation \eqref{eq:substrate} with $\ell$ bacteria and initial substrate concentration $s_0$. Namely, $\phi$ is the unique solution of
\begin{align}
\label{eq:flow}
\begin{cases}
&\frac{\dif \phi(\ell,s_0,t)}{\dif t} = D(\Sin-\phi(\ell,s_0,t))-k\,\mu(\phi(\ell,s_0,t))\,\ell,\\
& \phi(\ell,s_0,0)= s_0.
\end{cases}\,
\end{align}
This flow converges when $t\to\infty$ to $\bar s_\ell$ which is the unique solution of 
\begin{equation}
\label{eq:defsbar}
	D(\Sin-\bar s_\ell)-k\,\mu(\bar s_\ell)\,\ell = 0\,.
\end{equation}
where the sequence of points $(\bar s_\ell)_{\ell \geq 1}$ is strictly decreasing (see Lemmas~\ref{lem:defsbar} and \ref{lem:rapprochement}). Due to monotony properties, (see Lemmas~\ref{prop.monotony.flow} and \ref{lem:rapprochement}) we can build inverse functions of $t \mapsto \phi(\ell, s_0, t)$ and $s_0 \mapsto \phi(\ell, s_0, t)$ (both applications are represented in Figure~\ref{fig.flow}). On the one hand, for all $\ell\in \NN^*$ and $s_0\in\RR_+$ such that $s_0\neq \bar s_\ell$, the application $t\mapsto \phi(\ell,s_0,t)$ is bijective from $\RR_+$ to $\ensco{s_0}{\bar s_\ell}$. We denote by $s \mapsto \phit(\ell,s_0,s)$  the prolongation of its inverse function, defined from $\RR_+$ to $\widebar{\RR}_+$ by 
\[
\phit(\ell,s_0,s) = 
\begin{cases}
t \text{ such that } \phi(\ell,s_0,t)=s & \text{ if } s \in \ensco{s_0}{\bar s_\ell},\\
+ \infty & \text{ if not}.
\end{cases}\,
\]
It represents the time that the substrate concentration needs to go from $s_0$ to $s$ with a fixed number $\ell$ of bacteria  (without jump event). If $s$ is not reachable from $s_0$ with $\ell$ individuals, then this time is considered as infinite. By definition, $\phit(\ell,s_0,\phi(\ell,s_0,t))=t$ and if $s \in \ensco{s_0}{\bar s_\ell}$ then $\phi(\ell,s_0,\phit(\ell,s_0,s))=s$.

\medskip

On the other hand, for all $\ell \in \NN^*$ and $t\in\RR_+$, the application $s_0 \mapsto \phi(\ell,s_0,t)$ is bijective from $\RR_+$ to $[\phi(\ell,0,t),\,+\infty)$. Let  $s\mapsto \phis(\ell,s, t)$ be the prolongation of its inverse function, which is defined from  $\RR_+$ to $\RR_+$ by
\[
\phis(\ell,s,t) = 
\begin{cases}
s_0 \text{ such that } \phi(\ell,s_0,t)=s & \text{ if } s \geq \phi(\ell,0,t), \\
0 & \text{ if not}.
\end{cases}\,
\]

For $s \geq \phi(\ell,0,t)$, it represents the needed initial substrate concentration to obtain substrate concentration $s$ at time $t$ by following the dynamics with $\ell$ individuals. By definition, $\phis(\ell,\phi(\ell,s_0,t),t)=s_0$ and if $s\geq \phi(\ell,0,t)$, then $\phi(\ell,\phis(\ell,s,t),t)=s$.

\medskip

\begin{figure}
\captionsetup[subfigure]{justification=centering}
\begin{center}
\begin{subfigure}{0.45\textwidth}
 \includegraphics[height=5.5cm]{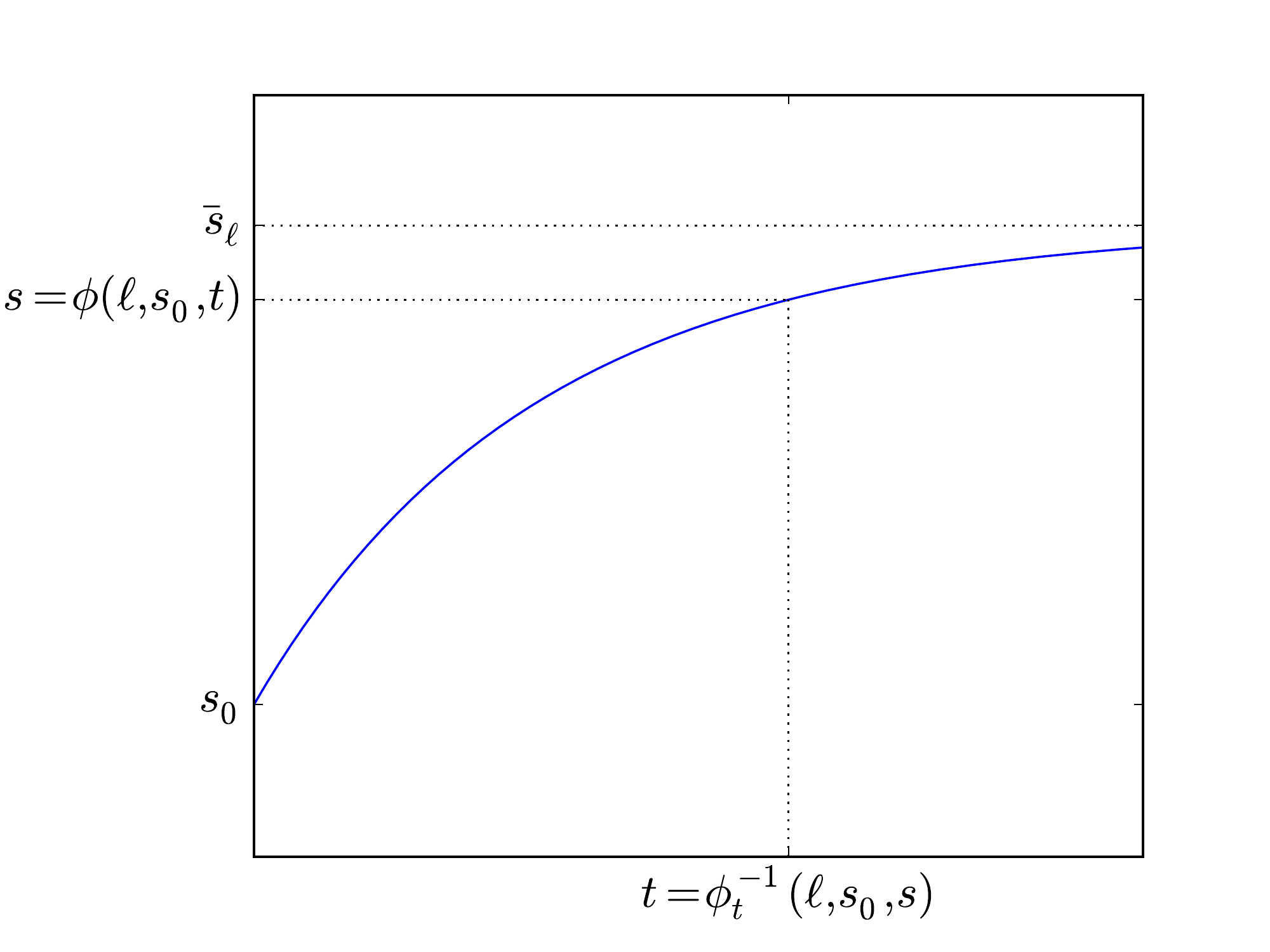}
\caption{$t\mapsto \phi(\ell,s_0,t)$}
\end{subfigure}
\begin{subfigure}{0.45\textwidth}
\includegraphics[height=5.5cm]{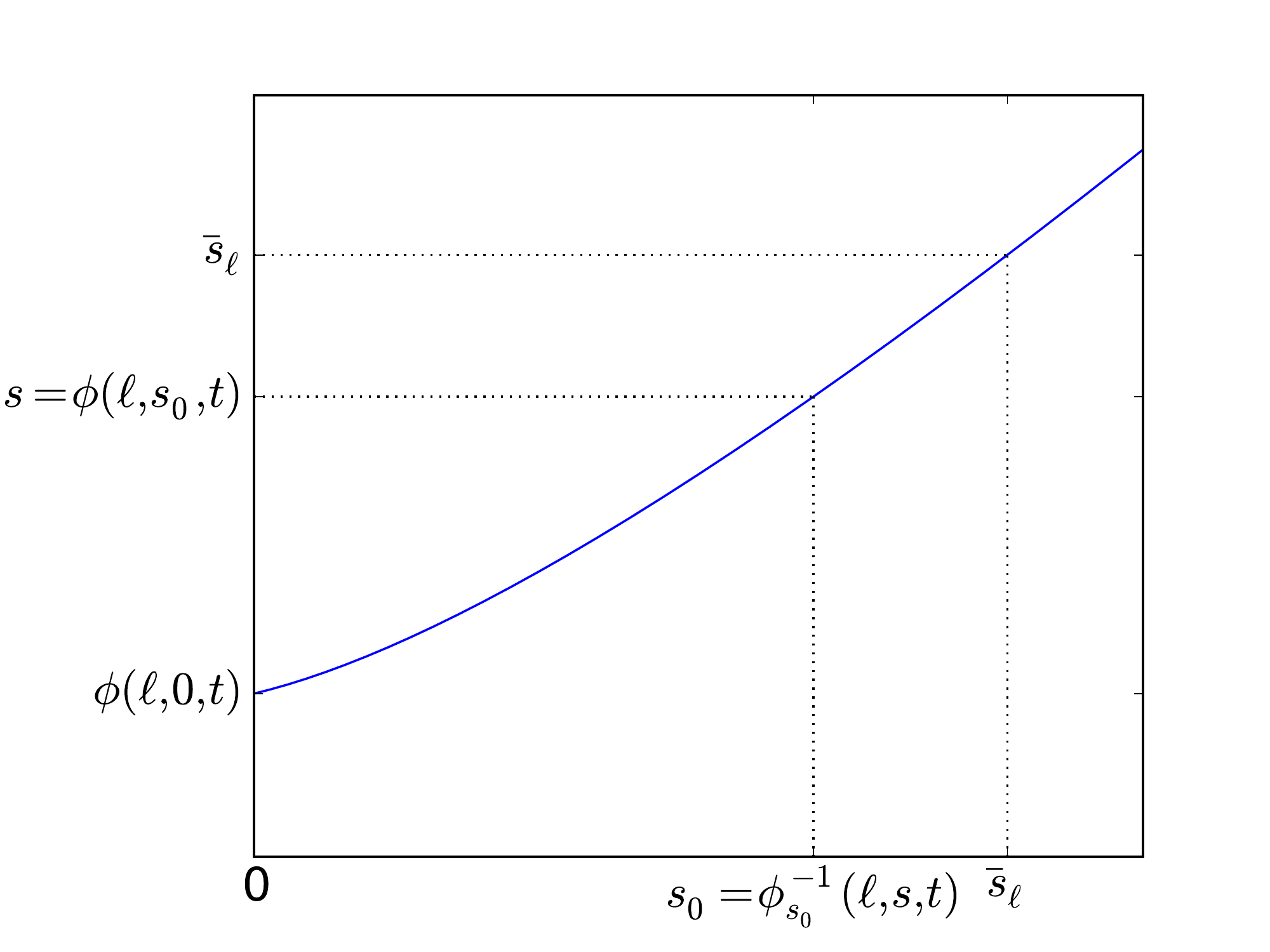}
\caption{$s_0\mapsto \phi(\ell,s_0,t)$}
\end{subfigure}
\end{center}
\caption{\label{fig.flow}Graphical representation of $t\mapsto \phi(\ell,s_0,t)$ and $s_0\mapsto \phi(\ell,s_0,t)$. }
\end{figure}

\subsection{Bounds on the hitting times of the process}
\label{sect:irre}

In this section, we will develop some irreducibility properties of the Crump-Young process through bounds on its hitting times, which will be useful to prove the mass ratio inequality in Section~\ref{sec:mass.ratio}. To that end, let $K$ be a non empty compact set of  $\NN^*\times (0,\bar s_1)$ and let $\SminK=\min_{(x, s)\in K} s $ and $\SmaxK=\max_{(x, s)\in K} s $.
We will prove that each point of $K\backslash \bigcup_{\ell\geq 1}(\ell,\bar s_\ell)$ can be reached, in a uniform way, from any point of $K$. Points $(\ell,\bar s_\ell)$ can not be reached.

 There exists $L_{\SminK}\in\NN^*$, such that $\bar s_\ell < \SminK$ for all $\ell\geq L_{\SminK}$ (see Lemma~\ref{lem:defsbar}), let then set $\lmax=\max\{\max_{(\ell, s)\in K} \ell,\, L_{\SminK}\}$. The constants $\SminK$, $\SmaxK$ and $\lmax$ satisfy $K\subset \setunlmax\times[\SminK,\SmaxK]\subset \setunlmax\times (\bar s_{\lmax},\bar s_1)$.

Let also 
\begin{align}
\label{def:tmin}
t_{\min}:= \max\left\{\phit\left(1,\bar s_{\lmax},\SmaxK\right),\, \,\phit\left(\lmax,\bar s_1,\SminK\right)\right\}
\end{align}
be the maximum between the time to go from $\bar s_{\lmax}$ to $\SmaxK$ with one individual and the time to go from $\bar s_1$ to  $\SminK$ with $\lmax$ individuals. As both times are finite then $t_{\min}<\infty$. 
Note that, from the monotony properties of the flow (see Lemma~\ref{lem:rapprochement}) for all $s_1,s_2$ such that $\bar s_{\lmax}\leq s_1\leq s_2\leq \SmaxK$, then $\phit(1,s_1,s_2)\leq \phit\left(1,\bar s_{\lmax},\SmaxK\right) \leq t_{\min}$ and for all $s_1,s_2$ such that $\SminK\leq s_2\leq s_1\leq \bar s_1$,  $\phit(\lmax,s_1,s_2)\leq \phit\left(\lmax,\bar s_1,\SminK\right)\leq t_{\min}$. Then $t_{\min}$ is the minimal quantity such that, for all $s_1, s_2$ satisfying $\bar s_{\lmax}\leq s_1\leq s_2\leq \SmaxK$ or $\SminK\leq s_2\leq s_1\leq \bar s_1$,  there exists $L\in \setunlmax$, such that $\phit(L,s_1,s_2)\leq t_{\min}$ (\textit{i.e. } the substrate concentration $s_2$ is reachable from $s_1$ in a time less than $t_{\min}$ with a constant bacterial population in $\setunlmax$).

\begin{proposition}
\label{prop.partout.a.partout}
For all $\tau_0>t_{\min}$, $\tau>\tau_0$, $\varepsilon>0$ and $\delta>0$, there exists $C>0$, such that, for all $(x,s)\in K$, for all $(y,r)\in K$ satisfying $|r-\bar s_y|>\delta$, we have
\[
\PP_{(x,s)}(\tau-\varepsilon \leq \tilde T_{y,r} \leq \tau)\geq
	C>0\,,
\]
where $\tilde T_{y,r}:= \inf\{t\geq \tau_0,\, (X_t,S_t)=(y,r)\}$ is the first hitting time of $(y,r)$ after $\tau_0$.
\end{proposition}

Proof of Proposition~\ref{prop.partout.a.partout} relies on sharp decomposition of all possibilities of combination of initial conditions. Instead of giving all details on the proof, we will expose its main steps and the technicalities are postponed in Appendix.
\begin{proof}
Let 
\[
\bar \varepsilon:= \min\left\{
	\frac{3\,\min\{\SminK-\bar s_{\lmax},\, \bar s_1-\SmaxK\}}
				{\max\{D\,\Sin,\, k\,\barmu\,\lmax\}} \, , \, 
	\frac{4\,\min\{\SminK-\bar s_{\lmax},\,\bar s_1-\SmaxK\}\,D\,(\tau_0-t_{\min})/2}
		{\max\{D\,\Sin,\, k\,\barmu\,\lmax\}\,(1+D\,(\tau_0-t_{\min})/2)}\right\}.
\]
We assume, without loss of generality, that $0<\varepsilon\leq \min\{\tau-\tau_0; \bar \varepsilon\}$ because if the result holds for all $\varepsilon>0$ sufficiently small, then it holds for all $\varepsilon>0$. 
Assuming $0<\varepsilon\leq \frac{3\,\min\{\SminK-\bar s_{\lmax},\, \bar s_1-\SmaxK\}}{\max\{D\,\Sin,\, k\,\barmu\,\lmax\}}$ ensures that, for $(y,r)\in K$, $\bar s_{\lmax}\leq \phis(y,r,\frac{\varepsilon}{3})\leq\bar s_1$ (and consequently that $\bar s_{\lmax}\leq \phis(y,r,\frac{\varepsilon}{4})\leq\bar s_1$); see Lemma~\ref{lemme.control.phi_moins_s}-\ref{lemme.control.phi_moins_s.maj} and Remark~\ref{phis0.less.sin}. Consequently $\mathcal{S}_{y,r}^\varepsilon:=\setunlmax \times \ens{\phis(y,r,\frac{\varepsilon}{3})}{\phis(y,r,\frac{\varepsilon}{4})} \subset \setunlmax \times [\bar s_{\lmax}, \bar s_1]$.

\medskip

To prove Proposition~\ref{prop.partout.a.partout}, we will prove that, with positive probability, the process: 
\begin{enumerate}
\item reaches the set
$\mathcal{S}_{y,r}^\varepsilon$ before $\tau_0$; 
\item stays in this set until the time $\tau-\varepsilon$;
\item reaches $(y,r)$ in the time interval $[\tau-\varepsilon,\,\tau]$.
\end{enumerate}
These steps are illustrated in Figure~\ref{fig.proof.prop.almost.everywhere} and  the associated probabilities are bounded from below in lemmas below. These ones are proved in Appendix~\ref{sec.proofs.lemmas}. To state them, let us introduce $\mathcal{E}_{y,r}^\varepsilon$, defined by
\[
\mathcal{E}_{y,r}^\varepsilon:=\{(\ell,r^\varepsilon_{y,r}) \, | \, \ell \in \setunlmax \text{ and } \bar s_\ell \geq r^\varepsilon_{y,r}\} \mathbin{\scalebox{1.5}{\ensuremath{\cup}}} \{(\ell,R^\varepsilon_{y,r}) \, | \, \ell \in \setunlmax \text{ and } \bar s_\ell \leq R^\varepsilon_{y,r}\} \subset\mathcal{S}_{y,r}^\varepsilon,
\]
where $r^\varepsilon_{y,r}= \min(\phis(y,r,\frac{\varepsilon}{3}),\phis(y,r,\frac{\varepsilon}{4}))$ and $R^\varepsilon_{y,r}=\max(\phis(y,r,\frac{\varepsilon}{3}),\phis(y,r,\frac{\varepsilon}{4}))$. The set $\mathcal{E}_{y,r}^\varepsilon$ represents the points $(\ell,s) \in \mathcal{S}_{y,r}^\varepsilon$ such that $s$ belongs to the bounds of the substrate part $\ens{\phis(y,r,\frac{\varepsilon}{3})}{\phis(y,r,\frac{\varepsilon}{4})}$ and $\ell$ is such that the flow $t\mapsto \phi(\ell,s,t)$ leads the dynamics to stay in $\ens{\phis(y,r,\frac{\varepsilon}{3})}{\phis(y,r,\frac{\varepsilon}{4})}$, at least for small $t$, if $\phis(y,r,\frac{\varepsilon}{3})\neq \phis(y,r,\frac{\varepsilon}{4})$ (that is if $r\neq\bar s_y$).
Note that $\mathcal{E}_{y,r}^\varepsilon$ is well defined if $r=\bar s_y$ and we obtain $\mathcal{E}_{y,\bar s_y}^\varepsilon=\setunlmax \times \{\bar s_y\}$.

\begin{figure}
\begin{center}
\includegraphics[width=10cm]{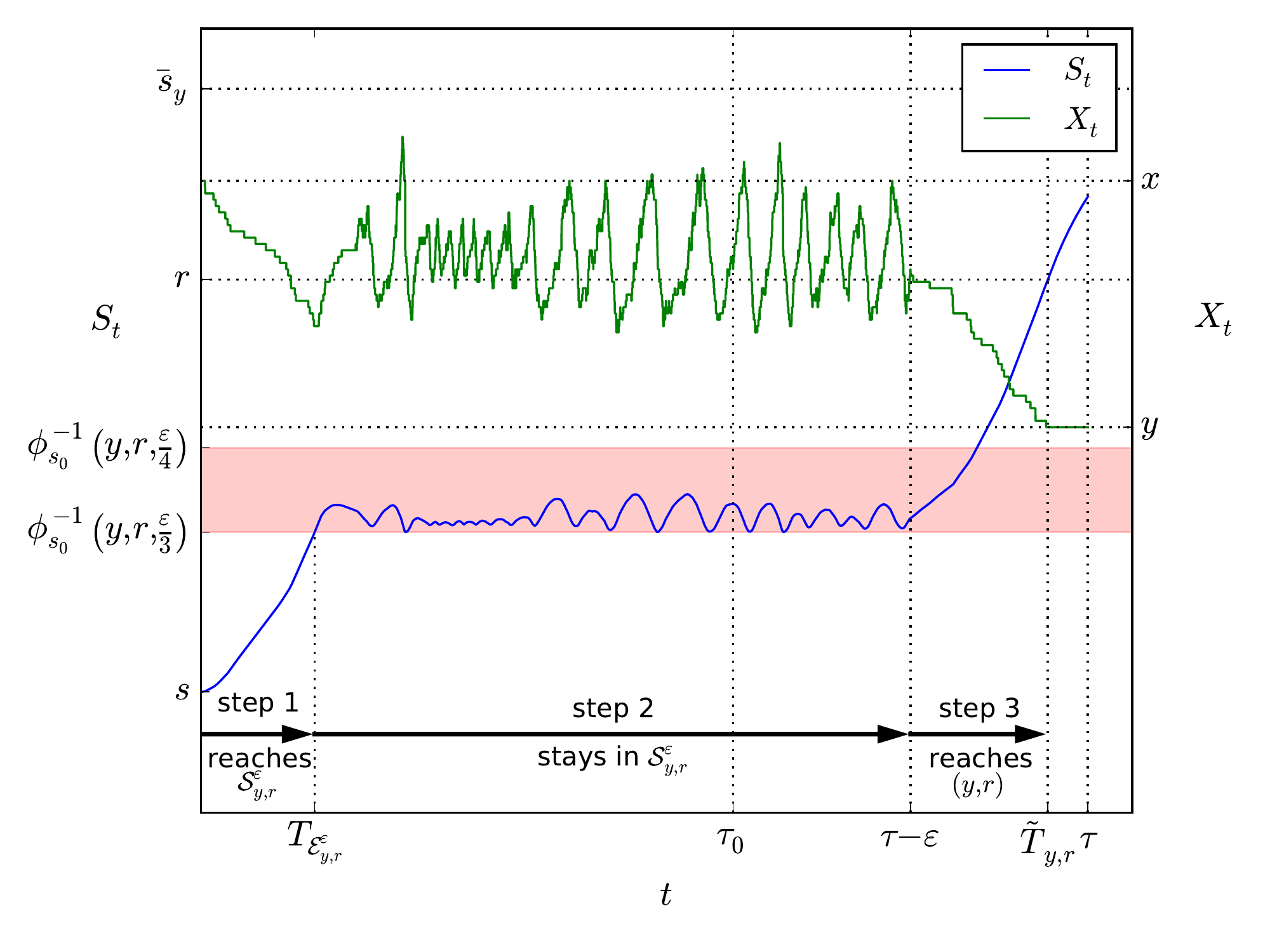}
\end{center}
\caption{\label{fig.proof.prop.almost.everywhere}Example of dynamics of the process $(X_t,S_t)_t$ from $(x,s)$ to $(y,r)$, illustrating the three steps of the proof of Proposition~\ref{prop.partout.a.partout}. Step 1: the process reaches the set $\mathcal{S}_{y,r}^\varepsilon$ before $\tau_0$. Step 2: the process stays in $\mathcal{S}_{y,r}^\varepsilon$ until $\tau-\varepsilon$. Step 3: the process reaches $(y,r)$ before $\tau$. }
\end{figure}

\medskip

\begin{lemma}
\label{de.partout.a.presque.partout}
For all $\tau_0>t_{\min}$, there exists $C_1^{\tau_0}>0$, such that, for all $(x,s)\in K$, for all $(y,r)\in K$ and for $0<\varepsilon\leq\frac{4\,\min\{\bar s_1-\SmaxK,\, \SminK-\bar s_{\lmax}\}\,D\,(\tau_0-t_{\min})/2}{\max\{D\,\Sin,\, k\,\barmu\,\lmax\}\,(1+D\,(\tau_0-t_{\min})/2)}$,
\[
\PP_{(x,s)}\big(T_{\mathcal{E}_{y,r}^\varepsilon}
	\leq \tau_0\big)\geq C_1^{\tau_0},
\]
where 
$T_{\mathcal{E}_{y,r}^\varepsilon}:=\inf\left\{t\geq 0,\, (X_t,S_t)\in \mathcal{E}_{y,r}^\varepsilon\right\}$. 
\end{lemma}

\medskip

\begin{lemma}
\label{lemma.stay.in.S}
Let $0<\varepsilon\leq \frac{3\,\min\{\SminK-\bar s_{\lmax},\, \bar s_1-\SmaxK\}}{\max\{D\,\Sin,\, k\,\barmu\,\lmax\}}$, $\delta>0$ and $T>0$. Then there exists $C^{\varepsilon,\delta,T}_2>0$, such that, for all $(y,r)\in K$ satisfying $|r-\bar s_y|>\delta$, for all $(x,s) \in \mathcal{E}_{y,r}^\varepsilon$,
\[
\PP_{(x,s)}
		\left((X_t,S_t) \in \mathcal{S}_{y,r}^\varepsilon, \, 
				\forall t\in[0,T]\right)\geq C^{\varepsilon,\delta,T}_2\,.
\]
\end{lemma}

\medskip

\begin{lemma}
\label{lemma.Tys.depuis.S}
Let $0<\varepsilon\leq \frac{3\,\min\{\SminK-\bar s_{\lmax},\, \bar s_1-\SmaxK\}}{\max\{D\,\Sin,\, k\,\barmu\,\lmax\}}$ and $\delta>0$. Then there exists $C^{\varepsilon,\delta}_3>0$, such that, for all $(y,r)\in K$ satisfying $|r-\bar s_y|>\delta$, for all $(x,s) \in \mathcal{S}_{y,r}^\varepsilon$,
\[
\PP_{(x,s)}
	\Big(T_{y,r} \leq \varepsilon \Big)\geq  
	C^{\varepsilon,\delta}_3
\]
with $T_{y,r}:=\inf\{t\geq 0, \, (X_t,S_t)=(y,r)\}$ the first hitting time of $(y,r)$.
\end{lemma}

\bigskip

Even not optimal, some explicit expressions of  $C_1^{\tau_0}$, $C^{\varepsilon,\delta,T}_2$, $C^{\varepsilon,\delta}_3$ of the previous lemmas are given in Appendix~\ref{sec.proofs.lemmas}.
Let us show below that they imply the conclusion of Proposition~\ref{prop.partout.a.partout}.
\begin{align}
\nonumber
& \PP_{(x,s)}(\tau-\varepsilon \leq \tilde T_{y,r} \leq \tau)
\\\nonumber
	& \qquad\qquad \geq \PP_{(x,s)}\Big(\{T_{\mathcal{E}_{y,r}^\varepsilon}\leq \tau_0\} 
				\cap \big\{(X_t,S_t) \in \mathcal{S}_{y,r}^\varepsilon, \, 
				\forall t\in[T_{\mathcal{E}_{y,r}^\varepsilon}, \, \tau-\varepsilon]\big\}
\\\nonumber
	&\qquad\qquad\qquad\qquad\qquad\qquad
				\cap \{\tau-\varepsilon \leq \tilde T_{y,r} \leq \tau\}\Big)
\\\nonumber
	&\qquad\qquad\geq \PP_{(x,s)}\big(T_{\mathcal{E}_{y,r}^\varepsilon}\leq \tau_0\big)\,
\\\nonumber
	&\qquad\qquad\quad
		\times
		\PP_{(x,s)}
		\Big((X_t,S_t) \in \mathcal{S}_{y,r}^\varepsilon, \, 
				\forall t\in[T_{\mathcal{E}_{y,r}^\varepsilon}, \, \tau-\varepsilon] \ | \ T_{\mathcal{E}_{y,r}^\varepsilon}\leq \tau_0 \Big)
\\\label{probability.decomposition}
	&\qquad\qquad\quad
		\times \PP_{(x,s)}
				\Big(\tau-\varepsilon \leq \tilde T_{y,r} \leq \tau \ | \ T_{\mathcal{E}_{y,r}^\varepsilon}\leq \tau_0, \, 
				(X_t,S_t) \in \mathcal{S}_{y,r}^\varepsilon, \, 
				\forall t\in[T_{\mathcal{E}_{y,r}^\varepsilon}, \, \tau-\varepsilon] \Big)\,.
\end{align}

By Lemma ~\ref{de.partout.a.presque.partout}, the first probability of the last member of \eqref{probability.decomposition} is bounded from below by a constant $C_1^{\tau_0}>0$. By Lemma~\ref{lemma.stay.in.S} and the Markov property the second probability is bounded from below by a constant $C^{\varepsilon,\delta,\tau}_2>0$. By definition,  $\tilde T_{y,r} \geq \tau_0$, moreover $(y,r) \notin \mathcal{S}_{y,r}^\varepsilon$, therefore on the event ${\{(X_0,S_0)=(x,s), \, T_{\mathcal{E}_{y,r}^\varepsilon}\leq \tau_0, \, (X_t,S_t) \in \mathcal{S}_{y,r}^\varepsilon, \,  \forall t\in[T_{\mathcal{E}_{y,r}^\varepsilon}, \, \tau-\varepsilon]\}}$ we have $\tilde T_{y,r}\geq \tau-\varepsilon$ almost surely. By Lemma~\ref{lemma.Tys.depuis.S} and Markov property, the third probability is bounded from below by a constant $C^{\varepsilon,\delta}_3>0$, which achieves the proof.
\end{proof}

\subsection{Proof of the sufficient conditions leading to Theorem~\ref{th:main-qsd}}
\label{sec:QS.behavior}

We prove in this section that the three conditions -- Bounds on Lyapunov functions \eqref{tag.BLF1} and \eqref{tag.BLF2};  Small set assertion \eqref{tag:SSA} and Mass ratio inequality \eqref{tag:MRI} -- hold. As it was proved in Section~\ref{sec:sufficient.conditions} that they imply Theorem~\ref{th:main-qsd}, it will conclude the proof of this theorem.

Bounds on Lyapunov functions \eqref{tag.BLF1} and \eqref{tag.BLF2} are given by Lemma~\ref{lem:lyapunov}; Small set assertion \eqref{tag:SSA} is given by Lemma~\ref{lem:smallset}; Mass ratio inequality\eqref{tag:MRI} is given by Lemma~\ref{lem:ratio}. 

\subsubsection{Bounds on Lyapunov functions}
\label{sec:lyapunov}

Let $\tilde V(x,s)=V(x,s)\,\1_{(x,s)\in \mathbb{N}^*\times(0,\bar s_1)}$ for all $(x,s)\in (\mathbb{N}^*\times(0,\bar s_1)) \cup (\{0\}\times (0,\Sin))$, where we recall that $V$ is defined on $\mathbb{N}^*\times(0,\bar s_1)$ by \eqref{def:psi_V}.
Assumptions \eqref{hyp:theta_p} and a simple computation lead to the following lemma, whose the proof is postponed in Appendix (see Section~\ref{sec:proof.lem.LV}). 

\begin{lemma}
\label{lem:LV}
There exists $\eta>D$ and $\zeta>0$ such that 
\[
\mathcal{L} \tilde V \leq -\eta \tilde V + \zeta \psi,
\]
on $(\mathbb{N}^*\times(0,\bar s_1)) \cup (\{0\}\times (0,\Sin))$, where $\mathcal{L}$ is the infinitesimal generator of $(X_t,S_t)_{t\geq 0}$ defined by \eqref{eq:generator}.
\end{lemma}

Using well known martingale properties associated to Crump-Young model, the drift condition exposed in Lemma~\ref{lem:LV} before extends to the following lemma.

\begin{lemma}
\label{lem:lyapunov}
There exists $\eta>D$ and $\zeta>0$, such that for all $t\geq 0$,  $x\in \mathbb{N}^*$ and $s\in (0,\bar s_1)$, we have 
\begin{equation}
\label{eq:encadrepsi}
  e^{-D t}\, x = e^{-D t} \psi(x,s) 
  \leq \mathbb{E}_{(x,s)}[\psi(X_t, S_t)] 
  \leq e^{(\mu(\bar s_1)-D) t}\, \psi(x,s)= e^{(\mu(\bar s_1)-D) t} \,x
\end{equation}
and 
\begin{equation}
\label{eq:V}
\mathbb{E}_{(x,s)}\left[ V(X_t,S_t) \, \1_{X_t\neq 0}\right] \leq e^{-\eta  t} V(x,s) + \zeta \frac{e^{(\mu(\bar s_1)-D)t}}{\eta-D} \psi(x,s)\,.
\end{equation}
\end{lemma}

\begin{proof}
It is classical (see for example Section~4 of \cite{CF15}) that, for $f\in\mathcal{C}^{0,1}(\mathbb{N}\times \RR_+)$, the process
\begin{equation}
\label{eq:martingale}
\left( f(X_t,S_t) - f(X_0,S_0) - \int_0^t \mathcal{L} f(X_u,S_u) \dif u \right)_{t\geq 0}
\end{equation}
is a  local martingale.
As $\psi \leq \tilde V$ on $(\mathbb{N}^*\times(0,\bar s_1)) \cup( \{0\}\times (0,\Sin))$, from Lemma~\ref{lem:LV}, $\tilde V$ satisfies $\mathcal{L}\tilde V\leq \zeta \tilde V $ for some $\zeta>0$.
Then using classical stopping time arguments (see \cite[Section~6.2.]{BCGM} or \cite[Theorem 2.1]{MTIII} and its proof for instance), we can show that it is a martingale when $f=\tilde V$ and then that $(\mathbb{E}_{(x,s)}[\tilde V(X_t,S_t)])_{t\in [0,T]}$ is bounded for all $T>0$. Consequently, \eqref{eq:martingale} is also a martingale for $f=\psi$, because $\psi \leq \tilde V$. Then, from the dominated convergence theorem and the fact that, from the expression of $\psi$, 
\[
-D \psi \leq \mathcal{L} \psi \leq (\mu(\bar s_1) -D)\psi,
\]
we obtain \eqref{eq:encadrepsi}.  Similarly, by the linearity of $\mathcal{L}$, from Lemma~\ref{lem:LV} and \eqref{eq:encadrepsi},
\[
	\mathcal{L}\left(\tilde V-\frac{\zeta}{\eta-D}\psi\right)
	\leq -\eta \tilde V + \zeta \psi + \frac{\zeta}{\eta-D}D\psi = -\eta\left( \tilde V-\frac{\zeta}{\eta-D}\psi\right)
\]
then, for all $(x,s)\in \mathbb{N}^*\times (0,\bar s_1)$
\begin{align*}
\mathbb{E}_{(x,s)}\left[ V(X_t,S_t) \, \1_{X_t\neq 0}\right]
	&= \mathbb{E}_{(x,s)}\left[ \tilde V(X_t,S_t)\right]
\\
		&\leq e^{-\eta\,t}\tilde V(x,s) + \frac{\zeta}{\eta-D}\left(\mathbb{E}_{(x,s)}\left[ \psi(X_t,S_t)\right]-e^{-\eta\,t}\psi(x,s)\right)
\\
	&\leq
		e^{-\eta\,t} V(x,s)+\frac{\zeta}{\eta-D}\,e^{(\mu(\bar s_1)-D) t}\psi(x,s)\,.
\end{align*}
and \eqref{eq:V} holds.
\end{proof}

Let us point out some similarities between our approach and the proof of \cite[Theorem 4.1]{CMMS}. Indeed, to prove existence of the QSD, tightness is enough and is garanted by the use of Lyapunov functions (see for instance \cite[Theorem 4.2]{CMMS2011}).  However, the interest of our work is to go behind the existence of the QSD by proving uniqueness and convergence through sharp estimate on hitting times and new Harris Theorem.

\subsubsection{Small set assertion}
\label{sec:small.set}

Let $K$ be a compact set of $\NN^*\times (0,\bar s_1)$. Consistently with notation of Section~\ref{sect:irre}, let $s_K:=\min_{(x,s)\in K}s$ and $S_K:=\min_{(x,s)\in K}s$ be respectively the minimal and maximal substrate concentration of elements of $K$.

Our aim in this subsection is to prove the small set assertion \eqref{tag:SSA} established page~\pageref{tag:SSA} by introducing the coupling measure $\nu$. The proof is based on Lemma~\ref{lem:bertrand} below.

\begin{lemma}
\label{lem:bertrand}
Let $\tau>0$, let $0<s_0<s_K$, $s_1>s_0$ and $x\in \NN^*$. Then there exists $\epsilon_0>0$ such that, for all $(y,r)\in K$,
\[
	\PP_{(y,r)}\left((X_{\tau},S_{\tau})\in \{x\}\times [s_0,s_1]\right) \geq \epsilon_0\,.
\]
\end{lemma}

Lemma~\ref{lem:bertrand} is proved in Appendix (see Section~\ref{sec:lem.bertrand}). As Proposition~\ref{prop.partout.a.partout}, its proof relies on a sharp study of the paths. From this, we deduce the next result which is one the cornerstone of the proof of Theorem~\ref{th:main-qsd}.

\begin{lemma}
\label{lem:smallset}
For every $\tau>0$, there exists $\epsilon>0$ and a probability measure $\nu$ on $\mathbb{N^*}\times (0,\bar s_1)$ such that
\begin{align}
\label{eq:smallset}
\forall (y,r) \in K, \quad \mathbb{P}_{(y,r)}\left((X_\tau, S_\tau) \in \cdot \right)\geq \epsilon \nu.
\end{align}
If moreover $K = \llbracket 1, N \rrbracket \times [\delta_1, \delta_2 ]$, for some $N \in \mathbb{N}^*$ and $\delta_2>\delta_1>0$, then we can choose $\nu$ such that $\nu(K)=1$.
\end{lemma}

\begin{proof} 
Starting from $(y,r) \in K$, the discrete component can reach any point $z$ of $\mathbb{N}^*$ in any time interval with positive probability, so we can easily use any Dirac mass $\delta_z$ (times a constant) as a lower bound for the first marginal of the law of $(X_\tau, S_\tau)$. Let us use $z=1$. For the continuous component, we can use the randomness of the last jump time to prove that its law has a lower bound with Lebesgue density. Consequently to prove \eqref{eq:smallset}, we consider the paths going to $\{2\} \times [s_0,s_1]$ for some $s_1\geq s_0$ well chosen, then being subjected to a washout, and we study the last jump to construct a lower bound with density. 

Let us consider some $\bar s_1> s_1>s_0>0$ such that $s_0<s_K$ and $0<\tau_0<\tau$ which will be fixed at the end of the proof. 
On the one hand, from Lemma~\ref{lem:bertrand}, there exists $\epsilon_0>0$ such that for all $(y,r)\in K$,
\begin{equation}
\label{eq:smallsetpresque0}
	\PP_{(y,r)}\left((X_{\tau-\tau_0},S_{\tau-\tau_0})\in \{2\}\times [s_0,s_1]\right) \geq \epsilon_0.
\end{equation}

On the other hand, let $f$ be any positive function, $s \in  [s_0,s_1]$ and $t>0$. By conditioning on the first jump time and using the Markov property, we have
\begin{align*}
\mathbb{E}_{(2,s)}\left[f(X_t,S_t)\right]
&= e^{-2\,D\, t - 2\,\int_0^t \mu (\phi(2,s,u)) \dif u} \,f\left(2, \phi(2,s,t)\right)\\
&+ \int_0^t  2\,D \,e^{-2\,D\, v - 2\,\int_0^v \mu (\phi(2,s,u)) \dif u} \,\mathbb{E}_{(1,\phi(2,s,v))}\left[f(X_{t-v},S_{t-v})\right] \dif v\\
&+ \int_0^t  2\,\mu (\phi(2,s,v)) \, e^{-2\,D\,v - 2\,\int_0^v \mu (\phi(2,s,u)) \dif u} \,
	\mathbb{E}_{(3,\phi(2,s,v))}\left[f(X_{t-v},S_{t-v})\right] \dif v\\
&\geq \int_0^t  2\,D \, e^{-2\,D \, v - 2\,\int_0^v \mu (\phi(2,s,u)) \dif u} \times e^{-\, (t-v)\,D - \int_0^{t-v} \mu (\phi(1,\phi(2,s,v),u)) \dif u} \\
&\qquad \qquad \times f\left(1, \phi(1,\phi(2,s,v),t-v)\right)  \dif v,
\end{align*}
where the last bound comes from a second use of the Markov property on the second term.  Roughly, we bounded our expectation by considering the event ``the first event is a washout and occurs during the time interval $(0,t)$ and no more jump occurs until $t$".

As $s\mapsto \phi(x,s,u)$ and $x\mapsto \phi(x,s,u)$ are respectively increasing and decreasing (see Lemma~\ref{prop.monotony.flow}) and $\mu$ is increasing, we have for all $s\in [s_0,s_1]$
\[
\mu (\phi(2,s,u))  \leq  \mu (\phi(1,s_1,u)),
\]
hence 
\begin{align*}
\mathbb{E}_{(2,s)}\left[f(X_t,S_t)\right]
&\geq 2\,D\, e^{-2\,D \,t} e^{ - 2\,\int_0^t \mu (\phi(1,s_1,u)) \dif u}\, \int_0^t  f\left(1, \phi(1,\phi(2,s,v),t-v)\right) \dif v.
\end{align*}
By the flow property and Lemma~\ref{prop.monotony.flow}, for $0<\varepsilon < t-v$, we have
\begin{align*}
\phi(1,\phi(2,s,v),t-v) & =
	\phi(1,\phi(1,\phi(2,s,v),\varepsilon),t-(v+\varepsilon))
\\
	&> 
	\phi(1,\phi(2,\phi(2,s,v),\varepsilon),t-(v+\varepsilon))
\\
	&= \phi(1,\phi(2,s,v+\varepsilon),t-(v+\varepsilon))
\end{align*}
and then  $v\mapsto \phi(1,\phi(2,s,v),t-v)$ is strictly decreasing on $[0,t]$. Moreover from \eqref{eq:flow} the derivative of $u\mapsto \phi(1,s,u)$ is bounded from above by $D\,\Sin$.  As $r\mapsto \phi(1,r,t-v)$ is increasing for $r\leq \bar s_1$ from Lemma~\ref{lem:rapprochement}, then from the expression $\phi(1,r,t-v) = r + \int_0^{t-v} (D(\Sin-\phi(1,r,u))-k\,\mu(\phi(1,r,u)))\,\dif u$, we have $0\leq\frac{\dif}{\dif r}\phi(1,r,t-v)\leq 1$. In addition, either $s\leq \bar s_2$ and $\frac{\dif}{\dif v}\phi(2,s,v)\geq 0$ or $s< \bar s_2$ and from \eqref{eq:flow} and Lemma~\ref{lem:rapprochement}, $\frac{\dif}{\dif v}\phi(2,s,v)\geq -2\,k\,\mu(\bar s_2)$. So finally, from the chain rule formula 
\begin{align*}
\frac{\dif}{\dif v} \phi(1,\phi(2,s,v),t-v)
	&= \frac{\dif}{\dif v}\phi(2,s,v)\,\frac{\dif}{\dif r} \phi(1,r,t-v)_{\vert_{r=\phi(2,s,v)}} - \frac{\dif}{\dif u} \phi(1,\phi(2,s,v),u)_{\vert_{u=t-v}},
\end{align*}
the derivative of $v\mapsto \phi(1,\phi(2,s,v),t-v)$ is then bounded from below by $-D\,\Sin-2\,k\,\mu(\bar s_2)$.  By a change of variable, for every $s_0<s_1$, we have for $c_0=\left[D\,\Sin+2\,k\,\mu(\bar s_2)\right]^{-1}$ and for $s\in (s_0,s_1)$
\begin{align*}
\mathbb{E}_{(2,s)}\left[f(X_t,S_t)\right]
&\geq c_{0} \,2\,D \, e^{-2\, D\, t} e^{ - 2\,\int_0^t \mu (\phi(1,s_1,u)) \dif u}\, \int_{\phi(2,s,t)}^{\phi(1,s,t)}  f(1,w)  \dif w \\
&\geq c_{0} \,2\, D e^{-2\, D\, t}\, e^{ - 2\,\int_0^t \mu (\phi(1,s_1,u)) \dif u} \,\int_{\phi(2,s_1,t)}^{\phi(1,s_0,t)}  f(1,w)  \dif w,
\end{align*}
where the last term is non negative as soon as $\phi(2,s_1,t)<\phi(1,s_0,t)$.
First,  we fix any $s_0<s_K$. 
As  $\phi(2,s_0,t)<\phi(1,s_0,t)$, by continuity, we can find $ s_1 > s_0$ satisfying $\phi(2,s_1,t)<\phi(1,s_0,t)$. 
Fixing such two points $s_0$ and $s_1$ for $t=\tau_0$ with $0<\tau_0<\tau$, then leads to
\begin{equation}
\label{eq:smallsetpresque}
\forall s\in [s_0,s_1], \quad \mathbb{P}_{(2,s)}\left((X_{\tau_0},S_{\tau_0})\in \cdot\right) \geq \epsilon_1 \nu,
\end{equation}
with
\[
 \nu(dy,ds) =\delta_{1}(dy) \frac{\1_{[\phi(2,s_1,\tau_0),\phi(1,s_0,\tau_0)]} (s)}{\phi(1,s_0,\tau_0) -\phi(2,s_1,\tau_0)} ds,
 \]
and
\[
\epsilon_1= c_{0} \, 2\,D \, e^{-2\,D\,\tau_0}\, e^{ - 2\,\int_0^{\tau_0}\mu (\phi(1,s_1,u)) \dif u}  (\phi(1,s_0,\tau_0) -\phi(2,s_1,\tau_0)).
\]
As a consequence,  from \eqref{eq:smallsetpresque0} and \eqref{eq:smallsetpresque}, Equation~\eqref{eq:smallset} holds with $\epsilon=\epsilon_0 \epsilon_1$, by Markov property.

\medskip
If $K = \llbracket 1, N \rrbracket \times [\delta_1, \delta_2 ]$, even if it means choosing $\tau_0$ small enough, $s_0$ and $s_1$ can be chosen such that they furthermore satisfy $\phi(1,s_0,\tau_0)> \delta_1$ and $\phi(2,s_1,\tau_0)< \delta_2$. Then $\nu(K)>0$ and \eqref{eq:smallsetpresque} holds with $\tilde \epsilon_1$ and the probability measure $\tilde \nu$, satisfying $\tilde \nu(K)=1$, defined by
\[
\tilde \nu = \frac{\nu\left(\1_K\, \cdot \right)}{\nu(K)}, \qquad \tilde \epsilon_1= \epsilon_1\, \nu(K)\,.
\]
Note that if $\tau_0<\phit(2,0,\delta_2)$ (i.e. $\phis(2,\delta_2,\tau_0)>0$) then such $s_0$ and $s_1$ exists.  In fact, we can choose 
$s_0$ and $s_1$ such that $s_0\in (\phis(1,\delta_1,\tau_0), \delta_1\wedge \phis(2,\delta_2,\tau_0))$ and $s_1\in (s_0, \phis(2,\delta_2,\tau_0)\wedge \phis(2,\phi(1,s_0,\tau_0),\tau_0))$. As $\delta_1<\bar s_1$ and as $\ell\mapsto \phis(\ell,s,\tau_0)$ and $s\mapsto \phis(\ell,s,\tau_0)$ are both increasing (by definition of $\phis$ and by Lemma~\ref{prop.monotony.flow}), we can check that $\delta_1$ and $\delta_2$ are well defined.
Moreover $s0$ and $s_1$ are such that $s_0<\delta_1=s_K $ and $s_0<s_1$, and by Lemma~\ref{prop.monotony.flow}, we have 
\[\phi(1,s_0,\tau_0)> \phi(1,\phis(1,\delta_1,\tau_0),\tau_0)\geq \delta_1;\]
\[\phi(2,s_1,\tau_0)<\phi(2, \phis(2,\delta_2,\tau_0), \tau_0)=\delta_2;\]
\[\phi(2,s_1,t)<\phi(2,\phis(2,\phi(1,s_0,\tau_0),\tau_0),\tau_0)=\phi(1,s_0,\tau_0).\qedhere\]
\end{proof}

\subsubsection{Mass ratio inequality}
\label{sec:mass.ratio}

Our aim in this subsection is to prove the mass ratio inequality \eqref{tag:MRI} given on page~\pageref{tag:MRI} by using our bounds on the hitting time given in Proposition~\ref{prop.partout.a.partout}. 

\begin{lemma}
\label{lem:ratio}
Let $K$ be a compact set of $\NN^*\times (0,\bar s_1)$, then
\[
\sup_{(x,s),(y,r)\in K} \sup_{t\geq 0} \frac{\mathbb{E}_{(y,r)} \left[ \psi (X_{t}, S_{t})\right]}{\mathbb{E}_{(x,s)} \left[ \psi (X_{t}, S_{t})\right]} <+ \infty.
\]
\end{lemma}

\begin{proof}

We set $L=\max_{(x,s)\in K}x$, $\SminK=\min_{(\ell, s)\in K} s $ and $\SmaxK=\max_{(\ell, s)\in K} s $. Let
\begin{equation}
\label{eq:deltadef}
0<\delta <  \min\left\{\frac{1}{2} \min_{y,z \in K, \, y\neq z } |\bar s_y - \bar s_z|\, ; \,  \bar s_1-\SmaxK\right\}.
\end{equation}
Note that, from Lemma~\ref{lem:defsbar}, elements of $(\bar s_\ell)_{\ell \in K}$ are all distinct and then the right member of \eqref{eq:deltadef} is then strictly positive.
Let  also $\tilde K :=   \llbracket 1, L+1 \rrbracket  \times [\min\{s_K, \bar s_L\}, \max\{S_K, \bar s_2\}]$ be a compact set of $\NN^*\times (0,\bar s_1)$ such that $K\subset \tilde K$ and let $t_{\min}$ defined by \eqref{def:tmin} for the compact set $\tilde K$. Let $\tau>t_{\min}$, from \eqref{eq:encadrepsi}, 
\[
\sup_{(x,s),(y,r)\in K} \sup_{t\leq \tau} \frac{\mathbb{E}_{(y,r)} \left[ \psi (X_{t}, S_{t})\right]}{\mathbb{E}_{(x,s)} \left[ \psi (X_{t}, S_{t})\right]} 
\leq e^{\mu(\bar s_1)\,\tau}\,L<+ \infty,
\]
then it remains to prove that there exist $C>0$, such that for all $(x,s),(y,r) \in K$ and $t\geq \tau$, we have
\begin{align}
\label{eq:ratio.t.geq.tau}
\mathbb{E}_{(y,r)} \left[ \psi (X_{t}, S_{t})\right]\leq C\,\mathbb{E}_{(x,s)} \left[ \psi (X_{t}, S_{t})\right].
\end{align}

We first show by Proposition~\ref{prop.partout.a.partout} that \eqref{eq:ratio.t.geq.tau} holds for $(y,r)\in \tilde K \setminus \left(\bigcup_{\ell= 1}^{L+1} \{\ell\} \times   \mathcal{B}(\bar s_\ell, \delta)\right)$ with $\mathcal{B}(\bar s_\ell, \delta):=\{r,\, |r-\bar s_\ell|\leq\delta\}$. Then for $(y,r) \in K \cap \bigcup_{\ell= 1}^L \{\ell\} \times   \mathcal{B}(\bar s_\ell, \delta)$,  conditioning on the first event, either no jump occurs and we make use of \eqref{eq:encadrepsi}, or a jump occurs and the process after the jump belongs to $\tilde K \setminus \left(\bigcup_{\ell= 1}^{L+1} \{\ell\} \times   \mathcal{B}(\bar s_\ell, \delta)\right)$ which then allows to use \eqref{eq:ratio.t.geq.tau}.

\medskip

By the Markov property, for all $0\leq u\leq t$, for all $(y,r)\in \tilde K$, $\mathbb{E}_{(y,r)}\left[\psi(X_t,S_t)\right] = \mathbb{E}_{(y,r)}\left[\mathbb{E}_{(X_{t-u},S_{t-u})}\left[\psi(X_u,S_u)\right]\right] $. Applying \eqref{eq:encadrepsi} to $\mathbb{E}_{(X_{t-u},S_{t-u})}\left[\psi(X_u,S_u)\right]$, we then obtain
\begin{align}
\label{eq:esppsi}
e^{-D\,u}\,\mathbb{E}_{(y,r)}\left[\psi(X_{t-u},S_{t-u})\right]
\leq \mathbb{E}_{(y,r)}\left[\psi(X_t,S_t)\right] \leq e^{(\barmu-D)u}\,\mathbb{E}_{(y,r)}\left[\psi(X_{t-u},S_{t-u})\right].
\end{align} 
 
Mimicking the arguments of \cite[Theorem 1.1]{CG20}, we then deduce that for every $(x,s) \in K \subset \tilde K$ and $(y,r)\in \tilde K \setminus \left(\bigcup_{\ell= 1}^{L+1} \{\ell\} \times   \mathcal{B}(\bar s_\ell, \delta)\right)$, for all $t\geq \tau$, 
\begin{align}
 \mathbb{E}_{(x,s)} \left[ \psi (X_t, S_t)\right]
 &\geq  \mathbb{E}_{(x,s)} \left[\1_{ T_{y,r}  \leq \tau} \times \mathbb{E}_{(y,r)}\left[ \psi\left(X_{t- u}, S_{t- u}\right)\right]_{\vert_{u=T_{y,r}}}\right]\nonumber\\
 &\geq \PP_{(x,s)}\left( T_{y,r}  \leq \tau \right) \int_0^\tau  \mathbb{E}_{(y,r)}\left[ \psi\left(X_{t- u}, S_{t-u}\right)\right] \sigma^{x,s}_{y,r}(\dif u)\nonumber\\
 &\geq \tilde C\int_0^\tau e^{-(\barmu-D) u} \sigma^{x,s}_{y,r}(\dif u)  \mathbb{E}_{(y,r)} \left[ \psi (X_t, S_t)\right]\nonumber\\
 &\geq \tilde C e^{-(\barmu-D) \tau}   \mathbb{E}_{(y,r)} \left[ \psi (X_t, S_t)\right], \label{eq:presqueA4}
\end{align}
with $T_{y,r}:=\inf\{t\geq 0, \, (X_t,S_t)=(y,r)\}$ the first hitting time of $(y,r)$ and $\tilde C>0$. In the first line, we used the strong Markov property, in the second line $\sigma^{x,s}_{y,r}$ represents the law of $T_{y,r}$ conditionally to $\{((X_0,S_0)=(x,s)) \cap (T_{y,r}\leq \tau)\}$ and the third line comes from \eqref{eq:esppsi} and Proposition~\ref{prop.partout.a.partout}.

\medskip

It remains to extend the previous inequality to $(y,r) \in K \cap \bigcup_{\ell= 1}^L \{\ell\} \times   \mathcal{B}(\bar s_\ell, \delta)$. 
By conditioning on the first jump and using Markov property, we have
\begin{align}
 \mathbb{E}_{(y,r)} \left[ \psi (X_t, S_t)\right]
&= e^{-D\,y \,t - y\,\int_0^t \mu (\phi(y,r,u)) \dif u} \,\psi\left(y, \phi(y,r,t)\right) \nonumber\\
&+ \int_0^t  y\,D \,e^{-D \,y \,v - y\,\int_0^v \mu (\phi(y,r,u)) \dif u}\, \mathbb{E}_{(y-1, \phi(y,r,v))} \left[ \psi (X_{t-v}, S_{t-v})\right] \dif v\nonumber\\
&+ \int_0^t  y\,\mu (\phi(y,r,v))  e^{-D\, y\, v - y\, \int_0^v \mu (\phi(y,r,u)) \dif u} \, \mathbb{E}_{(y+1, \phi(y,r,v))} \left[ \psi (X_{t-v}, S_{t-v})\right] \dif v .\label{eq:duhamel}
\end{align}
First notice that, as $\delta<\bar s_1- \SmaxK$,  then ($y,r)$ necessarily satisfies $y\geq 2$. Therefore $(y-1, \phi(y,r,v))\in \tilde K$ and $(y+1, \phi(y,r,v))\in \tilde K$ (see Lemma~\ref{lem:rapprochement}).

From the definition of $\psi$ and \eqref{eq:encadrepsi},  we have, for any $(x,s)\in K$
\begin{align}
\label{eq:duhamel_term1}
\psi\left(y, \phi(y,r,t)\right) 
&= \frac{y}{x} \psi\left(x, s\right)
\leq \frac{y}{x}\,e^{D \,t} \, \mathbb{E}_{(x,s)} \left[ \psi (X_t, S_t)\right].
\end{align}
From \eqref{eq:esppsi}, 
\begin{align*}
\mathbb{E}_{(y-1, \phi(y,r,v))} \left[ \psi (X_{t-v}, S_{t-v})\right]
\leq & e^{D\,v}\,\mathbb{E}_{(y-1, \phi(y,r,v))} \left[ \psi (X_{t}, S_{t})\right] .
\end{align*}
Now as  $(y, r) \in \bigcup_{\ell=1}^L \{\ell\} \times \mathcal{B}(\bar s_\ell, \delta)$, then $r \in \mathcal{B}(\bar s_y, \delta),$ and $\phi(y,r,v) \in \mathcal{B}(\bar s_y, \delta)$ for all $v\geq  0$ because of Lemma~\ref{lem:rapprochement} (\textit{i.e.} equilibrium points are attractive). Thus, by definition of $\delta$, we have
\begin{align*}
\left|\phi(y,r,v)- \bar s_{y-1} \right|
\geq  \left|\bar s_{y-1}- \bar s_y \right| - \left|\phi(y,r,v) - \bar s_{y} \right|
>  2\delta - \delta=\delta,
\end{align*}
then
\[
(y-1, \phi(y,r,v)) \notin \bigcup_{\ell=1}^L \{\ell\} \times\mathcal{B}(\bar s_\ell, \delta).
\]
We can then apply \eqref{eq:esppsi} and \eqref{eq:presqueA4} to obtain
\begin{align}
\nonumber
\mathbb{E}_{(y-1, \phi(y,r,v))} \left[ \psi (X_{t-v}, S_{t-v})\right] 
& \leq e^{D\,v}\,\mathbb{E}_{(y-1, \phi(y,r,v))} \left[ \psi (X_t, S_t)\right] \\
\label{eq:duhamel_term2}
& \leq  e^{D\,v}\,\tilde C^{-1} e^{(\barmu-D) \tau}  \mathbb{E}_{(x,s)} \left[ \psi (X_{t}, S_{t})\right] \,.
\end{align} 
Similarly, we have
\begin{align}
\label{eq:duhamel_term3}
\mathbb{E}_{(y+1, \phi(y,r,v))} \left[ \psi (X_{t-v}, S_{t-v})\right]
\leq & e^{D\,v}\, \tilde C^{-1} e^{(\barmu-D) \tau}   \mathbb{E}_{(x,s)} \left[ \psi (X_{t}, S_{t})\right] \,.
\end{align}
 From \eqref{eq:duhamel}-\eqref{eq:duhamel_term1}-\eqref{eq:duhamel_term2} and \eqref{eq:duhamel_term3}, we then obtain
\begin{align*}
 \mathbb{E}_{(y,r)} \left[ \psi (X_t, S_t)\right]
&\leq \mathbb{E}_{(x,s)} \left[ \psi (X_{t}, S_{t})\right]\,\Big(e^{-D\,(y-1)\, t - y\,\int_0^t \mu (\phi(y,r,u)) \dif u} \frac{y}{x}\\
&\quad + \tilde C^{-1}\,e^{(\barmu-D) \tau} \,\int_0^t  y\,(D+\mu (\phi(y,r,v)))\, e^{-D\,(y-1)\, v -y \,\int_0^v  \mu (\phi(y,r,u)) \dif u}\, \dif v\Big)\,.
\\
&\leq (L\wedge \tilde C^{-1}e^{(\barmu-D) \tau})\,\mathbb{E}_{(x,s)} \left[ \psi (X_{t}, S_{t})\right]
\\ & \qquad \times \Big(1+  \,\int_0^t D\, e^{-D\,(y-1)\, v -y \,\int_0^v  \mu (\phi(y,r,u)) \dif u}\, \dif v\Big)\,.
\end{align*}
As $y\geq 2$, then
\begin{align*}
 \int_0^t  D\, e^{-D\,(y-1)\, v - \int_0^v y \mu (\phi(y,r,u)) \dif u}\, \dif v
 	& \leq \int_0^t  D\, e^{-D\, v} \dif v \leq 1
\end{align*}
and \eqref{eq:ratio.t.geq.tau} holds with $C:=2\,(L\wedge \tilde C^{-1}e^{(\barmu-D) \tau})$, which achieves the proof.
\end{proof}

\section{Proof of Corollary~\ref{co:yaglom}}
\label{sec:proof.corollary}

Let us now show that the convergence towards the quasi-stationary distribution $\pi$, established in Theorem~\ref{th:main-qsd}, extends for initial measures with support larger than $\X$. The proof is in two parts, first we extend the convergence to initial conditions in $\mathbb{N}^*\times (0,+\infty)$ and then for $S_0=0$.
 
For the first part of the proof, it is sufficient to show that $h$ can be extended for all $(x,s) \in \mathbb{N}^* \times [\bar s_1,+\infty)$ such that $h(x,s)\in (0,\infty)$ and
\begin{align}
\label{eq:cv.exp.lamt.exp}
\lim_{t\to \infty} e^{\lambda t} \mathbb{E}_{(x,s)} [f(X_t,S_t)] = \pi(f) \times h(x,s),
\end{align}
for any bounded function on $\mathbb{N} \times \mathbb{R}_+$ such that $f(0,\cdot) =0$. In fact, if such function $h$ exists, then choosing $f(x,s)=\1_{x\neq 0}$ leads to $h(x,s)= \lim_{t\to \infty} e^{\lambda t}\mathbb{P}_{(x,s)} (\Text>t)$  (extending the definition of $h$ given by \eqref{def:h} on $\mathbb{N^*} \times \mathbb{R}_+$) and the result holds.

Let $\epsilon>0$ and set $T_\epsilon=T_{ \mathbb{N}^*\times (0,\bar s_1-\epsilon]}$ being the hitting time of $\mathbb{N}^*\times (0,\bar s_1-\epsilon]$. We have
\begin{align*}
\mathbb{E}_{(x,s)} [f(X_t,S_t)]
&= \mathbb{E}_{(x,s)} [f(X_t,S_t)\1_{(T_\epsilon\wedge \Text)\leq t}] + \mathbb{E}_{(x,s)} [f(X_t,S_t)\1_{(T_\epsilon\wedge \Text)>t} ]
\end{align*}

On the one hand, from Lemma~\ref{lem:lyap2} we can choose $\epsilon$ sufficiently small such that, from Markov inequality, 
\begin{align}
\label{eq:proba.T_epsilon}
\mathbb{P}_{(x,s)}(T_\epsilon\wedge \Text >t)\leq \mathbb{E}_{(x,s)}\left[e^{(D+C)(T_\epsilon\wedge \Text)} \right] e^{-(D+C)t} \leq A e^{\beta s} e^{-(D+C)t},
\end{align}
where $A$, $\beta$ and $C$ are positive constants (which depend on $\epsilon$) given by Lemma~\ref{lem:lyap2}.
As $\lambda \leq D$ by \eqref{eq:eigenvalue.bounds}, we then have 
\begin{align*}
e^{\lambda t}\mathbb{E}_{(x,s)} [f(X_t,S_t)\1_{T_\epsilon\wedge \Text>t} ] 
 &\leq \|f \|_\infty e^{\lambda t} \mathbb{P}_{(x,s)}(T_\epsilon\wedge \Text>t) \\
 &\leq \|f \|_\infty A e^{\beta s} e^{-Ct} \longrightarrow_{t\to \infty} 0\,.
\end{align*}
On the second hand, noting that $f(X_t,S_t)\1_{\Text\leq t}=0$, 
from the strong Markov Property 
\begin{align}
\nonumber
e^{\lambda t}\mathbb{E}_{(x,s)} [f(X_t,S_t)\1_{(T_\epsilon\wedge \Text)\leq t}] 
&=e^{\lambda t}\mathbb{E}_{(x,s)} [f(X_t,S_t)\1_{T_\epsilon\leq t}] \\
\label{eq:T_eps_leq_t}
&= e^{\lambda t} \mathbb{E}_{(x,s)} \left[
\mathbb{E}_{(X_{T_\epsilon},S_{T_\epsilon})}\left[f(X_{u},S_{u}) \right]_{\vert_{u=t-T_\epsilon}} \1_{T_\epsilon\leq t} \right].
\end{align}
Moreover, fixing $\rho>1$ and $p\in\left(0,\frac{\mu(\bar s_1)-D}{D+k\,\mu'(\bar s_1)}\right)$, for all $0<\tilde \omega\leq \omega$ (with $\omega$ depending on $\rho$ and $p$), as \eqref{eq:cv_qsd_X_eq_h} holds replacing  $\omega$ by $\tilde \omega$ and as, by continuity of the process $(S_t)_t$, $(X_{T_\epsilon},S_{T_\epsilon})\in\mathbb{N}^*\times \{\bar s_1-\epsilon\}\subset \mathbb{N}^*\times (0,\bar s_1)$ on the event $\{T_\epsilon\leq t\}$, we obtain
\begin{align}
\nonumber
& \left|e^{\lambda t}\,\mathbb{E}_{(x,s)} \left[
 \mathbb{E}_{(X_{T_\epsilon},S_{T_\epsilon})}\left[f(X_{u},S_{u}) \right]_{\vert_{u=t-T_\epsilon}} \1_{T_\epsilon\leq t} \right]
- \mathbb{E}_{(x,s)} \left[e^{\lambda T_\epsilon} \, h(X_{T_\epsilon},S_{T_\epsilon})\pi(f) \1_{T_\epsilon\leq t} \right] \right|\\
\nonumber 
& \qquad \leq
\mathbb{E}_{(x,s)} \left[e^{\lambda T_\epsilon} \left| e^{\lambda( t-T_\epsilon)}
 \mathbb{E}_{(X_{T_\epsilon},S_{T_\epsilon})}\left[f(X_{u},S_{u}) \right]_{\vert_{u=t-T_\epsilon}}-h(X_{T_\epsilon},S_{T_\epsilon})\pi(f)\right| \1_{T_\epsilon\leq t} \right]\\
\label{eq:cv.T_eps_leq_t}
&\qquad \leq
||f||_\infty\,C\,e^{-\tilde\omega\,t}
 \mathbb{E}_{(x,s)} \left[e^{(\lambda+\tilde\omega)\,T_\epsilon}W_{\rho,p}(X_{T_\epsilon},S_{T_\epsilon}) \1_{T_\epsilon\leq t} \right].
\end{align}
In addition,
\begin{align*}
 \mathbb{E}_{(x,s)} \left[e^{(\lambda+\tilde \omega)\, T_\epsilon}W_{\rho,p}(X_{T_\epsilon},S_{T_\epsilon}) \1_{T_\epsilon\leq t} \right]
 	&\leq	
 		\tilde C\, \mathbb{E}_{(x,s)} \left[e^{(\lambda+\tilde\omega)\, T_\epsilon} \,V_0(X_{T_\epsilon},S_{T_\epsilon}) \, \1_{T_\epsilon\leq t} \right]
\end{align*}
with $\tilde C=\log(\rho)\,e^{-\alpha(\bar s_1-\varepsilon)}+(\bar s_1-\epsilon)^{-1}+\epsilon^{-p}$ and $V_0$ defined by $V_0(x,s)=\rho^{x}\,e^{\alpha\,s}/\log(\rho)$ for all $(x,s)\in \mathbb{N}\times \mathbb{R}_+$. In the same way as in Section~\ref{sec:proof.lem.LV}, for all $\eta>0$, there exists $C_\eta>0$ such that $\mathcal{L}(V_0(x,s)-C_\eta)\leq -\eta (V_0(x,s) - C_\eta)$ for all $(x,s)\in \mathbb{N^*}\times \mathbb{R}_+$. And by the same arguments used in the proof of Lemma~\ref{lem:lyap2}, $\left(\left(V_0(X_t,S_t)-C_\eta\right)\,e^{\eta\,t}\right)_t$ is a submartingale. Then by the stopping time theorem and remarking that $\{T_\epsilon\leq t\} \subset \{T_\epsilon\leq \Text\}$, we obtain for $\eta=\lambda+\tilde \omega$
\begin{align}
\nonumber
& \mathbb{E}_{(x,s)} \left[e^{(\lambda+\tilde\omega)\, T_\epsilon}W_{\rho,p}(X_{T_\epsilon},S_{T_\epsilon}) \1_{T_\epsilon\leq t} \right]
\\ \nonumber	&\qquad\leq	
 		\tilde C\, \left|\mathbb{E}_{(x,s)} \left[e^{(\lambda+\tilde\omega)\, T_\epsilon\wedge t} \,(V_0(X_{T_\epsilon\wedge t},S_{T_\epsilon\wedge t})-C_\eta)\right]\right|
 		+\tilde C\,C_\eta\, \mathbb{E}_{(x,s)} \left[e^{(\lambda+\tilde\omega)\, T_\epsilon}\1_{T_\epsilon\leq t} \right]
\\ \label{eq:dom.exp.T_eps.V_0}
 	&\qquad \leq \tilde C\,\left|V_0(x,s)-C_\eta\right| +\tilde C\,C_\eta\, \mathbb{E}_{(x,s)} \left[e^{(\lambda+\tilde\omega)\, (T_\epsilon\wedge \Text)} \right]\,.
\end{align}
By \eqref{eq:eigenvalue.bounds}, $\lambda \leq D$. Then, for $0<\tilde \omega\leq \omega$ sufficiently small (smaller than the constant $C$ of Lemma~\ref{lem:lyap2}), Lemma~\ref{lem:lyap2} and \eqref{eq:dom.exp.T_eps.V_0} lead to
\begin{align}
\label{eq:esp.W.bounded}
\mathbb{E}_{(x,s)} \left[e^{(\lambda+\tilde\omega)\, T_\epsilon}W_{\rho,p}(X_{T_\epsilon},S_{T_\epsilon}) \1_{T_\epsilon\leq t} \right]
 \leq \tilde C\,\left|V_0(x,s)-C_\eta\right| +\tilde C\,C_\eta\, A\,e^{\beta\,s}\,.
\end{align}
Hence, \eqref{eq:T_eps_leq_t}, \eqref{eq:cv.T_eps_leq_t} and \eqref{eq:esp.W.bounded} gives
\begin{align*}
e^{\lambda t}\mathbb{E}_{(x,s)} [f(X_t,S_t)\1_{(T_\epsilon\wedge \Text)\leq t}] 
\longrightarrow_{t\to \infty}  \pi(f) \,\mathbb{E}_{(x,s)} \left[e^{\lambda T_\epsilon}  h(X_{T_\epsilon},S_{T_\epsilon}) \1_{T_\epsilon < \infty}\right],
\end{align*}
where we used that $h\leq C_h\, W_{\rho,p}$ on $\mathbb{N}^*\times (0,\bar s_1)$, \eqref{eq:esp.W.bounded} and the dominated convergence theorem. Then,  \eqref{eq:cv.exp.lamt.exp} holds with $h(x,s)=\mathbb{E}_{(x,s)} \left[e^{\lambda T_\epsilon} h(X_{T_\epsilon},S_{T_\epsilon}) \1_{T_\epsilon < \infty} \right]$ for all $(x,s) \in \mathbb{N}^* \times [\bar s_1,+\infty)$, which is finite by the previous arguments. Moreover Lemma~\ref{lem:lyap2} ensures that $h(x,s)>0$.

\bigskip 
 
It remains to show the result for $s=0$. Let $x\in \mathbb{N}^*$, the Markov property gives for $t'>t>0$,
\begin{align*}
\mathbb{E}_{(x,0)} \left[ f(X_{t'},S_{t'})  \ | \ \Text>t'\right] 
&=\frac{\mathbb{E}_{(x,0)} \left[ f(X_{t'},S_{t'})  \1_{\Text>t'} \right]}{\mathbb{E}_{(x,0)} \left[ \1_{\Text>t'}\right]}\\
& =\frac{\mathbb{E}_{(x,0)} \left[ f(X_{t'},S_{t'})  \1_{\Text>t'} \ | \ \Text> t \right]}{\mathbb{E}_{(x,0)} \left[ \1_{\Text>t'} \ | \ \Text> t\right]}\\ 
&=\frac{\mathbb{E}_{(x,0)} \left[ \mathbb{E}_{(X_{t},S_{t})}\left[ f(X_{t'-t},S_{t'-t})  \1_{\Text>(t'-t)} \right] \ | \ \Text> t \right]}{\mathbb{E}_{(x,0)} \left[ \mathbb{E}_{(X_{t},S_{t})}\left[ \1_{\Text>t'-t}  \right] \ | \ \Text> t \right]}\\ 
&=\frac{\mathbb{E}_{\xi} \left[ f(X_{t'-t},S_{t'-t})  \1_{\Text>(t'-t)} \right]}{\mathbb{E}_{\xi} \left[ \1_{\Text>t'-t}  \right] }\\ 
&=\mathbb{E}_{\xi} \left[ f(X_{t'-t},S_{t'-t}) \ | \ \Text>(t'-t) \right]\\ 
\end{align*}
where $\xi$ is the law of $(X_{t},S_{t})$ conditioned on the event $\{ \Text> t \}\cap\{(X_0,S_0)=(x,0)\}$. 
Assume that $\xi$ is a probability distribution on $\NN^*\times (0,\bar s_1)$, then, from \eqref{eq:cv_qsd},
\begin{align*}
\sup_{\|f \|_\infty \leq 1 } \left|\mathbb{E}_{\xi} \left[ f(X_{t'-t},S_{t'-t})  \ | \ \Text>t'-t\right] - \pi(f) \right| 
&\leq C \min\left(\frac{\xi(W_{\rho,p})}{\xi(h)}, \frac{\xi(W_{\rho,p})}{\xi(\psi)}\right)  e^{-\omega (t'-t)}
\end{align*}
with $\rho>1$ and $p\in\left(0,\frac{\mu(\bar s_1)-D}{D+k\,\mu'(\bar s_1)}\right)$. As $\xi(\psi)\neq 0$ (or $\xi(h)\neq 0$ as from Theorem~\ref{th:main-qsd}, $h(y,r)\in(0,\infty)$ for all $(y,r)\in\NN^*\times (0,\bar s_1)$, then Corollary~\ref{co:yaglom} holds for $s=0$ if in addition $\xi(W_{\rho,p})=\mathbb{E}_{(x,0)}[W_{\rho,p}(X_t,S_t)  \ | \ \Text> t]<+\infty$.  So let us prove that, for $\rho$ sufficiently small, $\xi$ is a probability distribution on $\NN^*\times (0,\bar s_1)$ and that $\xi(W_{\rho,p})<+\infty$, which both consist of proving that 
\[
\mathbb{E}_{(x,0)}\left[\frac{1}{S_t} \ | \ \Text >t\right] <+\infty.
\]
Indeed, note that conditionally on the non-extinction $S_t\leq \phi(1,0,t)< \bar s_1$. Moreover $(X_t)_t$ can be stochastically dominated by a pure birth process with birth rate $\mu(\bar s_1)$, whose the law at time $t$ is a negative binomial distribution with parameters $x$ and $e^{- \mu(\bar s_1) t}$. Then, for $1<\rho<(1-e^{\mu(\bar s_1)t})^{-1}$, $\mathbb{E}_{(x,0)}\left[\rho^{X_t} \ | \ \Text >t\right]\leq (e^{- \mu(\bar s_1) t}\,\rho /(1- \rho\,(1-e^{- \mu(\bar s_1) t})))^x$.

As the process $(X_t)_t$ dominates a pure death process with death rate (\textit{per capita}) $D$, we have $\mathbb{P}_{(x,0)}(\Text > t) \geq e^{-D x t}$, then it is sufficient to prove that for all (sufficiently small) $t>0$, 
\[
\mathbb{E}_{(x,0)}\left[\frac{1}{S_t} \right] <+\infty.
\]
Instead of using a Lyapunov function, we prove such inequality by coupling method. On $[0,t]$, from \eqref{eq:substrate} and given that $S_0=0$, we have the following upper-bound 
\[
\forall u\in [0,t], \quad  S_u \leq (S_0 + D\Sin t)\wedge \Sin \leq D\Sin t  ,
\]
and also the two following ones
\[
\forall s\in[0,D\Sin t], \quad \mu(s) \leq \bar\mu_t, \ \mu'(s) \leq \bar\mu'_t,
\]
for some constant $\bar\mu_t, \bar\mu'_t>0$. Consequently, we can couple $(X_u)_{u\in [0,t]}$ with a Yule process $(Z_u)_{u\in [0,t]}$ (namely a pure birth process) with jumps rate (\textit{per capita}) $\bar\mu_t$ in such  a way
\[
\forall u\leq t, \quad X_u \leq Z_u.
\]
In particular, $X_u\leq Z_t$. From this bound and the evolution equation of the substrate \eqref{eq:substrate}, we have
\begin{align}
\label{eq:min.S}
\forall u\in [0,t], \quad S_u' \geq D(\Sin -S_u) - k \bar\mu'_t S_u  Z_t,
\end{align}
and then, by a Gronwall type argument,
\[
\forall u\in [0,t], \quad S_u \geq \frac{D \Sin}{D+  k \bar\mu'_t Z_t } (1-e^{-D u-  k \bar\mu'_t Z_t u }) \geq \frac{D \Sin}{D+  k \bar\mu'_tZ_t } (1-e^{-D u}).
\]
Finally using the classical equality for pure birth processes $\mathbb{E}_{(x,0)}[Z_t] = x e^{\bar\mu_t t}$, we obtain 
\[
\mathbb{E}_{(x,0)}\left[\frac{1}{S_t} \right] \leq \frac{D+  k \bar\mu'_t x e^{\bar\mu_t t} }{D \Sin (1-e^{-D t})}  ,
\]
which ends the proof. 

Note that relaxing the assumptions as in Remark~\ref{remark:mu.lipschitz}, even if it means choosing $t$ small enough, $\bar{\mu}'_t$ can be replaced by a local Lipschitz constant in a neighborhood of $0$ in \eqref{eq:min.S}.

\appendix

\section{Classical and simple results on the Crump-Young process}

In the present section, we develop some simple properties on the Crump-Young process, under Assumption~\ref{hyp:mu}.

\subsection{Preliminary results on the flow}
\label{sec.preliminary.flow}

In this subsection, we expose simple results on the flow functions relative to the substrate dynamics with no evolution of the bacteria. We begin by results on the behavior of $\phi$, defined by \eqref{eq:flow}, and then we give bounds on $\phit$  and $\phis$. 

\begin{lemma}
\label{prop.monotony.flow}
The flow satisfies the following properties: for all $s,\, \tilde s  \in \RR_+$, $t> 0$, $\ell, \tilde \ell \in \NN^*$ such that $s< \tilde s$ and $\ell < \tilde \ell$
\begin{enumerate}
\item  $\phi(\ell,s,t)>\phi(\tilde \ell,s,t)$ ;
\item  $\phi(\ell,s,t)< \phi(\ell,\tilde s,t)$.
\end{enumerate}
\end{lemma} 
\begin{proof}
The first inequality comes from the decreasing property of  $\ell \mapsto D\,(\Sin-s)-k\,\mu(s)\,\ell$. The second point comes from the Cauchy-Lipschitz (or Picard–Lindelöf) theorem. 
\end{proof}

\begin{lemma}
\label{lem:defsbar}
For every $\ell\in \NN^*$, Equation~\eqref{eq:defsbar}, that is
\[
	D(\Sin-\bar s_\ell)-k\,\mu(\bar s_\ell)\,\ell = 0\,,
\]
admits a unique solution in $(0,\Sin)$. Furthermore the sequence $(\bar s_\ell)_{\ell\geq 1}$ is strictly decreasing and $\lim_{\ell \to \infty} \bar s_\ell = 0$.
\end{lemma}
\begin{proof}
The map $g_\ell:s\mapsto D(\Sin-s)-k\,\mu( s)\,\ell$ is strictly decreasing, $g_\ell(0)=D \Sin>0$,  $g_\ell(\Sin)=-k\,\mu(\Sin)\ell<0$ then \eqref{eq:defsbar} admits a unique solution in $(0,\Sin)$. Moreover, for every $s> 0$, the sequence $(g_\ell(s))_{\ell \geq 1}$ is strictly decreasing then $(\bar s_\ell)_{\ell\geq 1}$ is also strictly decreasing and
\[
\lim_{\ell \to \infty} \mu(\bar s_\ell) = \lim_{\ell \to \infty} \frac{D\,(\Sin-\bar s_\ell)}{k\,\ell} = 0
\]
then, by Assumption~\ref{hyp:mu}, $\lim_{\ell \to \infty} \bar s_\ell = 0$.
\end{proof}

\begin{lemma}
\label{lem:rapprochement}
For every $s \in \RR_+$, $\ell \in \NN^*$ and $t\geq 0$,
\begin{enumerate}
\item if $s< \bar s_\ell$, then $u\mapsto \phi(\ell,s,u)$ is strictly increasing from $\mathbb{R}_+$ to $[s,\bar s_\ell)$;
\item if $s> \bar s_\ell$, then $u\mapsto \phi(\ell,s,u)$ is strictly decreasing from $\mathbb{R}_+$ to $(\bar s_\ell,s]$.
\end{enumerate}
In particular
\[
|s - \bar s_\ell| \geq \left|\phi(\ell,s,t) - \bar s_\ell \right|.
\]
\end{lemma}
\begin{proof}
By Lemma~\ref{prop.monotony.flow}, if $s< \bar s_\ell$ then $\phi(\ell,s,t) < \bar s_\ell$ for every $t\geq 0$. On $[0, \bar s_\ell)$, $\partial_t \phi(\ell,\cdot,t)$ is strictly positive because, by Assumption~\ref{hyp:mu}, $g_\ell:s\mapsto D(\Sin-s)-k\,\mu( s)\,\ell$ is strictly decreasing and $g_\ell(\bar s_\ell)=0$. Finally,
\[
s \leq \phi(\ell,s,t) <  \bar s_\ell.
\]
In the same way, on $(\bar s_\ell,+\infty)$, $\partial_t \phi(\ell,\cdot,t)$ is strictly negative and $s \geq \phi(\ell,s,t) >  \bar s_\ell$ for $s\geq \bar s_\ell$ which ends the proof.
\end{proof}

\begin{corollary}
\label{cor:monotonyphit}
For every $s_0,s_1,s_2 \in \RR_+$ and $\ell \in \NN^*$ satisfying
$s_0 \geq s_1 \geq s_2 > \bar s_\ell$ or $s_0 \leq s_1 \leq s_2 < \bar s_\ell$ then
\[
\phit(\ell,s_0,s_2) = \phit(\ell,s_0,s_1)+\phit(\ell,s_1,s_2)<+\infty \,.
\]
\end{corollary}
\begin{proof}
The result directly comes from the monotony properties of the flow given by Lemma~\ref{lem:rapprochement} and the flow property.
\end{proof}

\begin{lemma}
\label{lem:phispos}
For all $\ell \in \mathbb{N}^*$, $s\geq 0$, and $t\geq \widetilde{t}\geq 0$, 
\[
\phis(\ell,s,t)>0 \Rightarrow \phis(\ell,s,\widetilde{t})>0.
\]
\end{lemma}
\begin{proof}
On the one hand, for every $u\geq 0$, by Lemma~\ref{prop.monotony.flow} and definition of $\phis$, we have 
\[
\phis(\ell,s,u)>0 \Leftrightarrow s> \phi(\ell,0,u).
\]
From Lemma~\ref{lem:rapprochement} $u \mapsto  \phi(\ell,0,u)$ is increasing. Thus 
\[
\phis(\ell,s,t)>0 \Leftrightarrow s> \phi(\ell,0,t) \Rightarrow s> \phi(\ell,0,\widetilde{t}) \Leftrightarrow \phis(\ell,s,\widetilde{t})>0.  \qedhere
\]
\end{proof}
\begin{lemma}
\label{lemme.control.phi_moins_1}
For $\ell \in \NN^*$, $(s_0,s)\in [0,\bar s_1]^2$, such that $s_0\neq \bar s_\ell$ and $\phit(\ell,s_0,s)<\infty$, 
\[
\frac{|s-s_0|}{\max\{D\,\Sin,\,k\,\barmu\,\ell\}} \leq \phit(\ell,s_0,s) \leq \frac{|s-s_0|}{D\,|\bar s_\ell-s|}\,.
\]
\end{lemma}
\begin{proof}
As $\phit(\ell,s_0,s)<\infty$, then
\begin{align*}
s &= s_0 + \int_0^{\phit(\ell,s_0,s)}[D(\Sin-\phi(\ell,s_0,u))-k\,\mu(\phi(\ell,s_0,u))\,\ell]\,\dif u
\\
&= s_0 + \int_0^{\phit(\ell,s_0,s)}[D(\bar s_\ell-\phi(\ell,s_0,u))+k\,\left(\mu(\bar s_\ell)-\mu(\phi(\ell,s_0,u))\right)\,\ell]\,\dif u\,.
\end{align*}
The first equality will allow to obtain the lower bound and the second one will lead to the upper bound of $\phit(\ell,s_0,s)$.
Either $s_0\leq s < \bar s_\ell$ then, from Lemma~\ref{lem:rapprochement}, the flow $u \mapsto \phi(\ell,s_0,u)$ is increasing and for all $u\in[0,\,\phit(\ell,s_0,s)]$, $s_0\leq \phi(\ell,s_0,u) \leq s < \bar s_\ell$ . As $\mu$ is increasing, we then obtain,
\begin{align*}
\phit(\ell,s_0,s)\,D(\bar s_\ell-s)
\leq s-s_0 \leq  \phit(\ell,s_0,s)\,D\,\Sin\,.
\end{align*}
Or $s_0 \geq s> \bar s_\ell$ then the flow $u \mapsto \phi(\ell,s_0,u)$ is decreasing and $s_0\geq \phi(\ell,s_0,u) \geq s > \bar s_\ell$ for all $u\in[0,\,\phit(\ell,s_0,s)]$. As $\phi(\ell,s_0,u)\leq \bar s_1\leq \Sin$ and $\mu$ is increasing, we then obtain
\begin{align*}
- \phit(\ell,s_0,s)\,k\,\barmu\,\ell
\leq s-s_0 \leq  \phit(\ell,s_0,s)\,D\,(\bar s_\ell-s)
\end{align*}
and the result holds.
\end{proof}

\medskip

\begin{lemma}
\label{lemme.control.phi_moins_s}
\begin{enumerate}
\item \label{lemme.control.phi_moins_s.maj} For all $(\ell,s,\varepsilon)\in \NN^* \times [0,\bar s_1]\times \RR_+$ such that $\phis(\ell,s,\varepsilon)\leq \bar s_1$,
\[
|s-\phis(\ell,s,\varepsilon)| \leq \varepsilon \max\{D\,\Sin,\,k\,\barmu\,\ell\}\,.
\]
\item \label{lemme.control.phi_moins_s.min} For all $(\ell,s,\varepsilon)\in \NN^* \times \RR_+\times \RR_+$ such that $\phis(\ell,s,\varepsilon)>0$, 
\[
D|s-\bar s_\ell|\,\varepsilon \leq |s-\phis(\ell,s,\varepsilon)|\,.
\]
\end{enumerate}
\end{lemma}

\begin{remark}
\label{phis0.less.sin}
If $s\leq \bar s_1$, then assumption $\phis(\ell,s,\varepsilon)\leq \bar s_1$ is satisfied when $\varepsilon\leq \frac{\bar s_1-s}{k\,\barmu\,\ell}$.
Indeed, from Lemmas~\ref{lem:rapprochement} and \ref{lem:defsbar}, $u\mapsto \phi(\ell,\bar s_1,u)$ is decreasing, then for all $u\geq 0$, $\phi(\ell,\bar s_1,u)\leq \bar s_1$ and
\begin{align*}
\phi(\ell,\bar s_1,\varepsilon) = \bar s_1 + \int_0^\varepsilon [D(\Sin-\phi(\ell,\bar s_1, u))-k\,\mu(\phi(\ell,\bar s_1, u))\,\ell]\,\dif u
	\geq \bar s_1-\varepsilon\,k\,\barmu\,\ell\,.
\end{align*}
Then $\varepsilon\leq \frac{\bar s_1-s}{k\,\barmu\,\ell}$ implies that $s\leq \phi(\ell,\bar s_1,\varepsilon)$. Hence, either $\phis(\ell,s,\varepsilon)=0\leq \bar s_1$, or $\phi(\ell,\phis(\ell,s,\varepsilon),\varepsilon) = s \leq \phi(\ell,\bar s_1,\varepsilon)$ and then, by Lemma~\ref{prop.monotony.flow}, $\phis(\ell,s,\varepsilon)\leq \bar s_1$.
\end{remark}

\begin{proof}[Proof of Lemma~\ref{lemme.control.phi_moins_s}]
First, we assume that $\phis(\ell,s,\varepsilon)>0$. 
By definition of $\phis$,
\begin{align*}
s & = \phis(\ell,s,\varepsilon) + \int_0^{\varepsilon}[D(\Sin-\phi(\ell,\phis(\ell,s,\varepsilon), u))-k\,\mu(\phi(\ell,\phis(\ell,s,\varepsilon), u))\,\ell]\,\dif u\\
& = \phis(\ell,s,\varepsilon) + \int_0^{\varepsilon}[D(\bar s_\ell-\phi(\ell,\phis(\ell,s,\varepsilon), u))+k\,(\mu(\bar s_\ell)-\mu(\phi(\ell,\phis(\ell,s,\varepsilon), u))\,\ell)]\,\dif u\,.
\end{align*}
On the one hand, if $s\leq \bar s_\ell$, then for all $u\in  [0,\varepsilon]$, $\phis(\ell,s,\varepsilon)\leq \phi(\ell,\phis(\ell,s,\varepsilon), u) \leq s \leq \bar s_\ell$, hence, from the second equality and as $\mu$ is increasing,
\[
s-\phis(\ell,s,\varepsilon)  \geq  D(\bar s_\ell-s)\,\varepsilon>0\,.
\]
In the same way, if $s\geq \bar s_\ell$, then for all $u\in  [0,\varepsilon]$, $\phis(\ell,s,\varepsilon)\geq \phi(\ell,\phis(\ell,s,\varepsilon), u) \geq s \geq \bar s_\ell$, hence
\[
\phis(\ell,s,\varepsilon) - s \geq  D(s-\bar s_\ell)\,\varepsilon>0
\]
and the lower bound of $|s-\phis(\ell,s,\varepsilon)|$ then holds.
 
On the other hand, if $s \in [0,\bar s_1]$ and $\phis(\ell,s,\varepsilon) \in [0,\bar s_1]$, then for all $u\in[0,\varepsilon]$, $\phi(\ell,\phis(\ell,s,\varepsilon), u)\leq \bar s_1$ and from the first equality,
\[
	 |s-\phis(\ell,s,\varepsilon)| \leq \varepsilon \max\{D\,\Sin,\,k\,\barmu\,\ell\}\,,
\]	
then the upper bound of $|s-\phis(\ell,s,\varepsilon)|$ holds for $0<\phis(\ell,s,\varepsilon)\leq \bar s_1$.

If $\phis(\ell,s,\varepsilon)=0$, then $s\leq \phi(\ell,0,\varepsilon)$ and
\begin{align*}
|s-\phis(\ell,s,\varepsilon)| =s 
& \leq  \int_0^{\varepsilon}[D(\Sin-\phi(\ell,0,u))-k\,\mu(\phi(\ell,0,u))\,\ell]\,\dif u
 \leq \varepsilon D\,\Sin
\end{align*}
and the upper bound of $|s-\phis(\ell,s,\varepsilon)| $ also holds for $\phis(\ell,s,\varepsilon)=0$.
\end{proof}

\subsection{Preliminary results on the jumps}
\label{sec:preliminary.jumps}

In contrast with the previous section, in the present one, we let the bacteria evolve. Let $(T_i)_{i \in \NN^*}$ be the sequence of the jump times of the process $(X_t)_{t\geq 0}$: 
\begin{align*}
T_i := \begin{cases}
\inf\{t> 0, \, X_{t-}\neq X_t\} & \text{if $i=1$};\\
\inf\{t>T_{ i-1}, \, X_{t-}\neq X_t\}& \text{if $i>1$.}
\end{cases}
\end{align*}

 Let us also introduce a classical notation in the study of piecewise deterministic Markov process (see \cite{BLMZ15} for instance). Let $(x_0,s_0)\in \mathbb{N}^*\times \mathbb{R}^+$, $0\leq t_1\leq \dots \leq t_{N+1}$ and let $\Psi\left(x_0,s_0,(t_j,x_j)_{1\leq j\leq N},t_{N+1}\right)$ be the iterative solution of
\begin{align}
\label{eq:defPsi}
\begin{cases}
\Psi(x_0,s_0,t_1) = \phi(x_0,s_0,t_1), \\
\Psi\left(x_0,s_0,(t_j,x_j)_{1\leq j\leq i},t_{i+1}\right)
	= \phi\Big(x_i,\Psi\left(x_0,s_0,(t_j,x_j)_{1\leq j\leq i-1},t_i\right),\,t_{i+1}-t_i\Big).
 \end{cases}
\end{align}
Then $\Psi\left(x_0,s_0,(t_j,x_j)_{1\leq j\leq N},t\right)$ represents the substrate concentration at time $t$, given the initial condition is $(x_0,s_0)$ and that the bacterial population jumps from $x_{i-1}$ to $x_i$ at time $t_i$ for $i=1,\cdots,N$. 

\medskip

For all $n\in\NN^*$, $u_1,\dots u_n>0$, let set
\[
	\mathcal{E}_D(u_1,\,\dots,u_n):=\bigcap_{i=1}^{n}\{X_{u_i}=X_0 -i\} \cap \{T_i=u_i\}
\]
and
\[
	\mathcal{E}_B(u_1,\,\dots,u_n):=\bigcap_{i=1}^{n}\{X_{u_i}=X_0 +i\} \cap \{T_i=u_i\}
\]
the event ``the first $n$ events are washouts (respectively divisions) and occur at time $u_1,\dots u_n$''.

\medskip

In Lemma~\ref{prop.minore.proba.evt} below, we use Poisson random measures to bound the probability of one event by the probability of this event conditionally to having followed a certain path (no jump, successive washouts or successive divisions). 

\begin{lemma}
\label{prop.minore.proba.evt}
Let $A$ be a measurable set (of the underlying probability space). 
We have the following inequalities. 
\begin{enumerate}
\item For all $\delta \geq 0$ and $(x,s) \in \NN^*\times \mathbb{R}_+$ 
\[
	\PP_{(x,s)}(A )\geq \PP_{(x,s)}(A \cap \{T_1 > \delta\}) \geq e^{-(D+\mu(\bar s_1 \vee s))\,x\,\delta}\,\PP_{(x,s)}(A \ | \ T_1 > \delta)\,.
\]
\item For all $\delta \geq 0$, $(x,s) \in \NN^*\times  \mathbb{R}_+$ and $1\leq n\leq x$,
\begin{align*}
\PP_{(x,s)}(A)
& \geq	
	\PP_{(x,s)}
	\Big(A \cap 
		\bigcap_{i=1}^{n}\Big\{\{T_i\ \leq\delta\} \cap \{ X_{T_i} = x-i \}\Big\}\Big)
\\
	&\geq \int_0^{\delta} \int_{u_1}^{\delta}\dots \int_{u_{n-1}}^{\delta}
	\left(\prod_{k=x-n+1}^{x}D\,k\right)\,
	e^{-(D+\mu(\bar s_1 \vee s))\,\left(x\,u_1+
	\sum_{i=1}^{n-1}(x-i)\,(u_{i+1}-u_{i})\right)}\,	
\\
	 &   \hspace{1cm} \times
	\PP_{(x,s)}
	\Big(A \ \Big| \mathcal{E}_D(u_1,\dots,u_n) \Big)\,
	\dif u_{n}\dots\,\dif u_1\,.
\end{align*}

\item For all $\delta \geq 0$, $(x,s) \in \NN^*\times  \mathbb{R}_+$ and all $n\geq 1$
\begin{align*}
\PP_{(x,s)}(A)
& 
	\geq \PP_{(x,s)}
	\Big(A \cap 
		\bigcap_{i=1}^{n}\Big\{\{T_i\ \leq\delta\} \cap \{ X_{T_i} = x+i \}\Big\}\Big)
\\
	&\geq \int_0^{\delta} \int_{u_1}^{\delta}\dots \int_{u_{n-1}}^{\delta}
	\left(\prod_{k=1}^{n}\mu\left(\Psi(x,s,(u_i,x+i)_{1\leq i \leq k-1},u_k)\right)\,(x+k-1)\right)\,
\\
	 &   \hspace{1cm} \times
	e^{-(D+\mu(\bar s_1 \vee s))\,\left(x\,u_1+
	\sum_{i=1}^{n-1}(x+i)\,(u_{i+1}-u_{i})\right)}\,	
\\
	 &   \hspace{1cm} \times
	\PP_{(x,s)}
	\Big(A \ \Big| \mathcal{E}_B(u_1,\dots, u_n) \Big)\,
	\dif u_{n}\dots\,\dif u_1\,.
\end{align*}
\end{enumerate}
\end{lemma}

\begin{proof}
Under the event $\{X_t\geq 1\}$ (or equivalently under the event $\{X_u\geq 1 \text{ for } u \in [0, t]\}$ as $\{0\}$ is an absorbing state for the process $(X_t)_t$), from the comparison theorem and Lemma~\ref{lem:rapprochement}, for all $0\leq u \leq t$ we have  $S_u \leq \phi(1,S_0,u)\leq S_0\vee \bar s_1$.
Let $(x,s) \in \NN^*\times  \mathbb{R}_+$, the individual jump rate $\mu(S_t)$ of the process $(X_t,S_t)$ starting from $(x,s)$ is then bounded by $\mu(\bar s_1\vee s)$.

The bounds established in the lemma are classical and based on the construction of the process $(X_t,S_t)$ from Poisson random measures: 
we consider two independent Poisson random measures $\mathcal{N}_d(\dif u, \dif j, \dif \theta)$ and $\mathcal{N}_w(\dif u, \dif j)$ defined on $\RR_+ \times \NN^* \times [0,1]$ and $\RR_+ \times \NN^*$ respectively, corresponding to the division and washout mechanisms respectively, with respective intensity measures
\begin{align*}
n_d(\dif u, \dif j, \dif \theta)
	 =
		\mu(\bar s_1 \vee s) \, \dif u \, \Big(\sum_{\ell \geq 1} \delta_{\ell}(\dif j) \Big)
		 \, \dif \theta
\qquad \text{and} \qquad
n_w(\dif u, \dif j)
	 =
		D \, \dif u \, \Big(\sum_{\ell \geq 1} \delta_{\ell}(\dif j) \Big)\,.
\end{align*}
Then the process $(X_t,S_t)$ starting from $(X_0,S_0)=(x,s)$ can be defined by 
\begin{align*}
(X_t,S_t) & =
	(x, \phi(x,s,t))\\
	&\quad
	+ \int_0^t \int_\mathbb{N^*} \int_0^1 \1_{\{j\leq X_{\umoins}\}}\,\1_{\{0\leq \theta \leq \mu(S_u)/\mu(\bar s_1 \vee s)\}}\,\\
	& \qquad \qquad \qquad \quad
		\left[ (1, \phi(X_\umoins +1, S_u, t-u)- \phi(X_\umoins, S_u, t-u)\right]\,\mathcal{N}_d(\dif u, \dif j, \dif \theta)
\\
&\quad
	+ \int_0^t \int_\mathbb{N^*} \1_{\{j\leq X_{\umoins}\}}\,
		\left[ (-1, \phi(X_\umoins -1, S_u, t-u)- \phi(X_\umoins, S_u, t-u)\right]\,\mathcal{N}_w(\dif u, \dif j)\,.
\end{align*}
We refer to \cite{CF15} for more details on this construction. 

1. By construction of the process, if $(X_0,S_0)=(x,s)$, we get $T_1 = T_{d}\wedge T_{w}\,$ where, $T_{d}$ is the time of the first jump of the process 
\[
  t\mapsto \mathcal{N}_{d}\left([0,t]\times \{x\}\times \left[0,\frac{\mu(\phi(x,s,u))}{\mu(\bar s_1 \vee s)}\right] \right)
\]
and $T_{w}$ is the time of the first jump of the process $t\mapsto \mathcal{N}_{w}([0,t]\times \{x\})$.

The distribution of $T_{d}$ is a non-homogeneous exponential distribution with parameter $\mu(\phi(x,s,u))\,x$,  \textit{i.e.} with the probability density function 
\[
t\mapsto \mu(\phi(x,s,t))\,x\,\exp\left({-\int_0^t \mu(\phi(x,s,u))\,x \dif u}\right).
\]
 The distribution of $T_{w}$ is a (homogeneous) exponential distribution with parameter $D\,x$. $T_{d}$ and $T_{w}$ are independent, then
\[
	\PP_{(x,s)}(T_1 > \delta) = e^{-\int_0^\delta (\mu(\phi(x,s,u))+D)\,x \, \dif u}
	\geq e^{-(D+\mu(\bar s_1 \vee s))\,x\,\delta}
\]
and the first result holds.

2. On the event $\bigcap_{i=1}^{k}\Big\{\{T_i\ =u_i\} \cap \{ X_{T_i} = x-i \}\Big\}$, the distribution of $T_{k+1}-u_k$ is a non-homogeneous exponential distribution with parameter $(\mu(\phi(x-k,S_{T_k},t))+D)\,(x-k)$
with $S_{T_k} = \Psi(x,s,(u_i,x-i)_{1\leq i \leq k-1}, u_k) \in (0,\bar s_1 \vee s)$, \textit{i.e.} with the probability density function (evaluated in $t$)
\begin{align*}
(\mu(\phi(x-k,S_{T_k},t))+D)\,(x-k)\,e^{-\int_0^t (\mu(\phi(x-k,S_{T_k},u))+D)\,(x-k) \dif u}
\\ \qquad \geq 
(\mu(\phi(x-k,S_{T_k},t))+D)\,(x-k)\,e^{-(\mu(\bar s_1 \vee s)+D)\,(x-k)\,t}
\end{align*}
 and on the event $\{T_{k+1}=u\}$, the event is a bacterial washout with probability $D/(\mu(\phi(x-k,S_{T_k},u))+D)$. We then obtain the second assertion.

3. The third assertion is obtained in the same way as the second one.
\end{proof}


\section{Proof of technical Lemmas}
\label{sec.proofs.lemmas} 

\subsection{Additional notation}
\label{sec.add.notation.app}

For all $n\geq \ell\geq 1$ and all $t\geq0$ , let $P_d(n,\ell,t)$ defined by
\begin{align*}
P_d(n,\ell,t)
	&= \int_0^{t} \int_{u_1}^{t}\dots \int_{u_{\ell-1}}^{t}
	\left(\prod_{k=n-\ell+1}^{n}D\,k\right)\,
	e^{-(D+\barmu)\,\left(n\,u_1+
	\sum_{i=1}^{\ell-1}(n-i)\,(u_{i+1}-u_{i})\right)}\,	
	\dif u_\ell \dots\,\dif u_1
\end{align*}
be the probability that the $\ell$ first events are deaths and occur in the time interval $[0,t]$ for a birth-death process, with \textit{per capita} birth rate $\barmu$ and death rate $D$, starting from $n$ individuals.

For all $n\geq \ell\geq 1$ and all $t\geq0$, let $P_b(n,\ell,t)$ defined by
\begin{align*}
P_b(n,\ell,t)
	&= \int_0^{t} \int_{u_1}^{t}\dots \int_{u_{\ell-1}}^{t}
	\left(\prod_{k=n}^{n+\ell-1}\barmu\,k\right)\,
	e^{-(D+\barmu)\,\left(n\,u_1+
	\sum_{i=1}^{\ell-1}(n+i)\,(u_{i+1}-u_{i})\right)}\,	
	\dif u_\ell \dots\,\dif u_1
\end{align*}
be the probability that the $\ell$ first events are births and occur in the time interval $[0,t]$ for a birth-death process, with \textit{per capita} birth rate $\barmu$ and death rate $D$, starting from $n$ individuals.

\begin{remark}
\label{rk.Pd.Pb.inc}
Both maps $t\mapsto P_d(n,\ell,t)$ and $t\mapsto P_b(n,\ell,t)$ are increasing.
\end{remark}

For all $L\in\NN^*$, $\Smin,\, \Smax$ such that $\bar s_L<\Smin \leq \Smax<\bar s_1$, we define the hitting time $T_{L,[\Smin,\Smax]}$ by
\[
T_{L,[\Smin,\Smax]}:=\inf\left\{t\geq 0,\, (X_t,S_t)\in B(L,[\Smin,\Smax])\right\}\,,
\]
where
\[
	B(L,[\Smin,\Smax]):=\{(\ell,\Smin) \, | \, \ell \in  \llbracket 1,L \rrbracket \text{ and } \bar s_\ell \geq \Smin\} \cup \{(\ell,\Smax) \, | \, \ell \in \llbracket 1,L \rrbracket \text{ and } \bar s_\ell \leq \Smax\}\,.
\]
In addition of being an hitting time of $\llbracket 1 , L \rrbracket \times [\Smin,\Smax]$, the boundary $B(L,[\Smin,\Smax])$ is choosen such that the process remains in this set during some positive time after $T_{L,[\Smin,\Smax]}$ if $\Smin<\Smax$. If $\Smin=\Smax$ then $B(L,[\Smin,\Smax])=\llbracket 1,L \rrbracket \times \{\Smin\}$.

\subsection{Proof of Lemma~\ref{de.partout.a.presque.partout}}
\label{sec:partout.a.partout}

Lemma~\ref{de.partout.a.presque.partout} is a consequence of Lemma~\ref{lemma.reach.T.Smin.Smax} below.

\medskip

\begin{lemma}
\label{lemma.reach.T.Smin.Smax}
Let $L \in \NN^*$, $\Smin,\Smax$ such that $\bar s_L<\Smin\leq\Smax<\bar s_1$, and let $(x,s)\in \llbracket 1, L\rrbracket \times [\bar s_L,\bar s_1]$,
\begin{enumerate}
\item \label{lemma.reach.T.Smin.Smax.s_inf_Smin} if $s\leq \Smin$, then for $\tau_0>\phit\left(1, s,\Smin\right)$,
\begin{align*}
\PP_{(x,s)}\left(T_{L,[\Smin,\Smax]}\leq \tau_0\right)	
	&\geq 
	e^{-(D+\barmu)\,(\tau_0-\delta)}\,P_d(x,x-1,\delta)\\
	&\geq 
	e^{-(D+\barmu)\,(\tau_0-\delta)}\,P_d(L,L-1,\delta)
\end{align*}
with $\delta:=(\tau_0-\phit\left(1, s,\Smin\right))\,\frac{D\,|\bar s_1-s|}{D\,|\bar s_1-s|+\max\{D\,\Sin,\,k\,\barmu\,L\}}$;

\item \label{lemma.reach.T.Smin.Smax.s_sup_Smax} if $s\geq\Smax$, then for $\tau_0>\phit\left(L, s,\Smax\right)$,
\begin{align*}
\PP_{(x,s)}\left(T_{L,[\Smin,\Smax]}\leq \tau_0\right)
	&\geq
	e^{-(D+\barmu)\,(\tau_0-\delta)\,L}\,\left(\mu(\bar s_L)/\barmu\right)^{L-x}\,
		P_b(x,L-x, \delta)
\\
	&\geq
	e^{-(D+\barmu)\,(\tau_0-\delta)\,L}\,
	\left(\mu(\bar s_L)/\barmu\right)^{L-1}\,P_b(1,L-1, \delta)
\end{align*}
with $\delta:=(\tau_0-\phit\left(L, s,\Smax\right))\,\frac{D\,|\bar s_L-s|}{D\,|\bar s_L-s|+\max\{D\,\Sin,\,k\,\barmu\,L\}}$;

\item  \label{lemma.reach.T.Smin.Smax.s_in_Smin_Smax} if $s\in (\Smin,\Smax)$, then for  $\tau_0>\phit\left(1, s,\Smax\right) \wedge \phit\left(L, s,\Smin\right)=:t^\star$
\begin{align*}
\PP_{(x,s)}\left(T_{L,[\Smin,\Smax]}\leq \tau_0\right)
&\geq
 e^{-(D+\barmu)\,\left(\tau_0-\delta_1-\delta_2\right)\,L}\\
& \quad		\times(\mu(\bar s_L)/\barmu)^{L-1}\,P_d(L,L-1, \delta_1)\,
		P_b(1,L-1, \delta_2)
\end{align*}
with $\delta_1:=\frac{\tau_0-t^\star}{2}\,\frac{D\,|\bar s_1-s|}{D\,|\bar s_1-s|+\max\{D\,\Sin,\,k\,\barmu\,L\}}$ and $\delta_2:=\frac{\tau_0-t^\star}{2}\,\frac{D\,|\bar s_L-s|}{D\,|\bar s_L-s|+\max\{D\,\Sin,\,k\,\barmu\,L\}}$.
\end{enumerate}
\end{lemma}

\medskip

\begin{proof}[Proof of Lemma~\ref{de.partout.a.presque.partout}]
Let $(y,r)\in K$ and let us define  
$\Smin:=\phis(y,r,\frac{\varepsilon}{3})$, $\Smax:=\phis(y,r,\frac{\varepsilon}{4})$ if $r\leq \bar s_y$ and 
$\Smax:= \phis(y,r,\frac{\varepsilon}{3})$, $\Smin:= \phis(y,r,\frac{\varepsilon}{4})$ if $r\geq \bar s_y$. 
Then $T_{\mathcal{E}_{y,r}^\varepsilon}=T_{\lmax,[\Smin,\Smax]}$.
From Lemma~\ref{lemme.control.phi_moins_s}-\ref{lemme.control.phi_moins_s.maj} and Remark~\ref{phis0.less.sin}, we have  $|r-\phis(y,r,\frac{\varepsilon}{4})|\leq \frac{\varepsilon}{4}\,\max\{D\,\Sin,k\,\barmu y\}$, then the condition 
\[
\varepsilon\leq\frac{4\,\min\{\bar s_1-\SmaxK,\, \SminK-\bar s_{\lmax}\}\,D\,(\tau_0-t_{\min})/2}{\max\{D\,\Sin,\, k\,\barmu\,\lmax\}\,(1+D\,(\tau_0-t_{\min})/2)}
\]
implies that, for $r\in [\SminK,\SmaxK]\subset [\bar s_\lmax,\bar s_1]$ and $y\in \setunlmax$,
\begin{align}
\label{encadre.eps4}
\texttt{s}
&\leq \phis\left(y,r,\frac{\varepsilon}{4}\right)
\leq \mathcal{S}
\end{align}
with $\texttt{s}:=\SminK-\frac{(\SminK-\bar s_{\lmax})\,D\,(\tau_0-t_{\min})/2}{1+D\,(\tau_0-t_{\min})/2}$ and
$\mathcal{S}:= \SmaxK+\frac{(\bar s_1-\SmaxK)\,D\,(\tau_0-t_{\min})/2}{1+D\,(\tau_0-t_{\min})/2}$.\\

In addition, as $\bar s_{\lmax}<\texttt{s}\leq \SminK\leq \SmaxK \leq \mathcal{S} < \bar s_1$,  from Corollary~\ref{cor:monotonyphit} and Lemma~\ref{lemme.control.phi_moins_1},
\begin{align}
\label{time.to.go.up}
\phit\left(1,\SminK,\mathcal{S}\right)
&= 
\phit\left(1,\SminK,\SmaxK\right)
+
\phit\left(1,\SmaxK,\mathcal{S}\right)\\ \nonumber
	&\leq
	t_{\min}
+
\frac{\mathcal{S}-\SmaxK}{D\,|\bar s_1-\mathcal{S}|}
= \tau_0
-
\frac{\tau_0-t_{\min}}{2}
\end{align}
and 
\begin{align}
\label{time.to.go.down}
\phit\left(\lmax,\SmaxK,\texttt{s}\right)
	&= 
\phit\left(\lmax,\SmaxK,\SminK\right)
+
\phit\left(\lmax,\SminK,\texttt{s}\right)\\ \nonumber
	&\leq
	t_{\min}
+
\frac{\SminK-\texttt{s}}{D\,|\texttt{s}-\bar s_\lmax|}
= \tau_0
-
\frac{\tau_0-t_{\min}}{2}\,.
\end{align}
Let set  $\delta_1:=\frac{\tau_0-t_{\min}}{2}\,\frac{D\,|\bar s_1-s|}{D\,|\bar s_1-s|+\max\{D\,\Sin,\,k\,\barmu\,\lmax\}}$ and $\delta_2:=\frac{\tau_0-t_{\min}}{2}\,\,\frac{D\,|\bar s_\lmax-s|}{D\,|\bar s_\lmax-s|+\max\{D\,\Sin,\,k\,\barmu\,\lmax\}}$.
From \eqref{encadre.eps4} $\Smin$ or $\Smax$ (or both) belongs to $[\texttt{s}, \mathcal{S}]$, hence for $(x,s)\in K$, we have three cases. 
\begin{enumerate}
\item If $s\leq \Smin$, then $\Smin \leq \mathcal{S}$, from Corollary~\ref{cor:monotonyphit} and from \eqref{time.to.go.up}
\[
\tau_0-\phit\left(1,s,\Smin\right)
\geq \tau_0 - \phit\left(1,\SminK,\mathcal{S}\right)
\geq \frac{\tau_0-t_{\min}}{2} >0 
\]
then from Lemma~\ref{lemma.reach.T.Smin.Smax}-\ref{lemma.reach.T.Smin.Smax.s_inf_Smin} and Remark~\ref{rk.Pd.Pb.inc},
\[
	\PP_{(x,s)}\big(T_{\mathcal{E}_{y,r}^\varepsilon}
	\leq \tau_0\big)\geq e^{-(D+\barmu)\,(\tau_0-\delta_1)}\,P_d(\lmax,\lmax-1,\delta_1)\,;
\]

\item If $s\geq \Smax$, then $\Smax \geq \texttt{s}$, from Corollary~\ref{cor:monotonyphit} and \eqref{time.to.go.down},
\[
\tau_0-\phit\left(\lmax,s,\Smax\right)
\geq \tau_0-\phit\left(\lmax,\SmaxK,\texttt{s}\right)
\geq\frac{\tau_0-t_{\min}}{2}>0
\]
then from Lemma~\ref{lemma.reach.T.Smin.Smax}-\ref{lemma.reach.T.Smin.Smax.s_sup_Smax} and Remark~\ref{rk.Pd.Pb.inc},
\[
	\PP_{(x,s)}\big(T_{\mathcal{E}_{y,r}^\varepsilon}
	\leq \tau_0\big)\geq e^{-(D+\barmu)\,(\tau_0-\delta_2)\,\lmax}\,
	\left(\frac{\mu(\bar s_{\lmax})}{\barmu}\right)^{\lmax-1}\,P_b(1,\lmax-1, \delta_2)\,;
\]

\item If $s\in (\Smin, \Smax)$, then $\Smin$ or $\Smax$ belongs to $[\texttt{s},\mathcal{S}]$, and at least one of both conditions $\tau_0-\phit\left(1,s,\Smin\right)
\geq \frac{\tau_0-t_{\min}}{2} >0 
$
or
$
\tau_0-\phit\left(\lmax,s,\Smax\right)
\geq\frac{\tau_0-t_{\min}}{2}>0
$ is satisfied. We then deduce from Lemma~\ref{lemma.reach.T.Smin.Smax}-\ref{lemma.reach.T.Smin.Smax.s_in_Smin_Smax} and Remark~\ref{rk.Pd.Pb.inc} that 
\begin{align*}
\PP_{(x,s)}\big(T_{\mathcal{E}_{y,r}^\varepsilon}\leq \tau_0\big)
	&\geq e^{-(D+\barmu)\,\left(\tau_0-\frac{\delta_1+\delta_2}{2}\right)\,\lmax}\,
		P_d\left(\lmax,\lmax-1, \frac{\delta_1}{2}\right)\,
\\
	&\qquad \times
	\left(\frac{\mu(\bar s_{\lmax})}{\barmu}\right)^{\lmax-1}\,P_b\left(1,\lmax-1, \frac{\delta_2}{2}\right)\,.
\end{align*}
\end{enumerate}
Finally Lemma~\ref{de.partout.a.presque.partout} holds with
\begin{align*}
C_1^{\tau_0}&:= 
	e^{-(D+\barmu)\,(\tau_0-\min\{\delta_1,\delta_2\})\,\lmax}\\
& \qquad
	\times P_d\left(\lmax,\lmax-1, \frac{\delta_1}{2}\right)\,
	\left(\frac{\mu(\bar s_{\lmax})}{\barmu}\right)^{\lmax-1}\,P_b\left(1,\lmax-1, \frac{\delta_2}{2}\right)\,.
	\qedhere
\end{align*}
\end{proof}

\medskip

\begin{proof}[Proof of Lemma~\ref{lemma.reach.T.Smin.Smax}]
Proof of Item~\ref{lemma.reach.T.Smin.Smax.s_inf_Smin}.  If $s \leq \Smin$,  we will prove that one way for the process to reach $B(L,[\Smin,\Smax])$ before $\tau_0$ is if the population jumps from $x$ to 1 by $x-1$ successive washout events during the time duration $\delta:=(\tau_0-\phit\left(1, s,\Smin\right))\,\frac{D\,|\bar s_1-s|}{D\,|\bar s_1-s|+\max\{D\,\Sin,\,k\,\barmu\,L\}}$ and if then no event occurs during the time duration $\tau_0-\delta$. The main arguments of the proof are the following: we will see that during the time duration $\delta$, the substrate concentration remains greater than or equal to $s-\delta\,\max\{D\,\Sin,\,k\,\barmu\,L\}$ and that $\delta$ is build such that $\phit\left(1, s-\delta\,\max\{D\,\Sin,\,k\,\barmu\,L\},\Smin\right)\leq \tau_0-\delta$ that is the remaining time after the successive washout events is enough for the substrate process to reach $\Smin$.

\begin{itemize}
\item if $x=1$ and $s_0 \in [s-\delta\,\max\{D\,\Sin,\,k\,\barmu\,L\}, \Smin]$, from Lemma~\ref{prop.minore.proba.evt} we have
\begin{align*}
\PP_{(1,s_0)}\left(T_{L,[\Smin,\Smax]}\leq \tau_0-\delta\right)
	& \geq
	\PP_{(1,s_0)}\left(\{T_{L,[\Smin,\Smax]}\leq  \tau_0-\delta)\}
						\cap \{T_1\ >  \tau_0-\delta\}\right)
\\
	&\geq
	e^{-(D+\barmu)\, (\tau_0-\delta)}\,
	\PP_{(1,s_0)}		
	\left(T_{L,[\Smin,\Smax]}\leq \tau_0-\delta \, | \, T_1>  \tau_0-\delta\right).
\end{align*}
Moreover, from Lemma~\ref{lemme.control.phi_moins_1}
\begin{align*}
\phit \Big(1,s-\delta\,\max\{D\,\Sin,\,k\,\barmu\,L\},s\Big)
	&\leq \frac{\delta\,\max\{D\,\Sin,\,k\,\barmu\,L\}}{D\,|\bar s_1-s|}
	=T
\end{align*}
with $T=(\tau_0-\phit\left(1, s,\Smin\right))\,\frac{\max\{D\,\Sin,\,k\,\barmu\,L\}}{D\,|\bar s_1-s|+\max\{D\,\Sin,\,k\,\barmu\,L\}}$.
Then from Corollary~\ref{cor:monotonyphit},
\begin{align*}
\phit\left(1, s_0,\Smin \right) &= \phit \Big(1,s-\delta\,\max\{D\,\Sin,\,k\,\barmu\,L\},\Smin\Big) 
\\ &\qquad- \phit \Big(1,s-\delta\,\max\{D\,\Sin,\,k\,\barmu\,L\},s_0\Big)
\\
	&\leq \phit \Big(1,s-\delta\,\max\{D\,\Sin,\,k\,\barmu\,L\},s\Big)+\phit \Big(1,s,\Smin\Big)
\\
	&\leq T+\phit \Big(1,s,\Smin\Big) = \tau_0 - \delta\,.
\end{align*}
Then, as from Lemma~\ref{lem:rapprochement} $t\mapsto \phi(1, s_0,t)$ is increasing,
\begin{align*}
\phi(1,s_0,\tau_0-\delta)
	&\geq \phi(1,s_0,\phit\left(1, s_0,\Smin \right))=\Smin\,.
\end{align*}
On the event $\{(X_0,S_0)=(1,s_0), \, T_1>  \tau_0-\delta \}$, we then have $S_{\tau_0-\delta} \geq \Smin$ a.s.
As $(S_t)_{t\geq 0}$ is a continuous process, from the intermediate value theorem, $(S_t)_{t\geq 0}$ reaches $\Smin$ in the time interval $[0,\tau_0-\delta]$. Moreover, as $\Smin<\bar s_1$ then  
\[
\PP_{(1,s_0)}		
	\left(T_{L,[\Smin,\Smax]}\leq \tau_0-\delta \, | \, T_1>  \tau_0-\delta\right)=1
\]
and therefore
\begin{align}
\label{eq.reach.Smin.Smax.befor.tmax.x1}
\PP_{(1,s_0)}\left(T_{L,[\Smin,\Smax]}\leq \tau_0-\delta \right)	
	&\geq
	e^{-(D+\barmu)\,(\tau_0-\delta)}\,.
\end{align}
As $\PP_{(1,s)}\left(T_{L,[\Smin,\Smax]}\leq \tau_0\right) \geq \PP_{(1,s)}\left(T_{L,[\Smin,\Smax]}\leq \tau_0-\delta\right)$, taking $s_0=s$ leads to the result.

\item if $x>1$, from Lemma~\ref{prop.minore.proba.evt},
\begin{align*}
\PP_{(x,s)}\left(T_{L,[\Smin,\Smax]}\leq \tau_0\right)
	&\geq
	\int_0^{\delta} \int_{u_1}^{\delta}\cdots \int_{u_{x-2}}^{\delta} 
		\left(\prod_{k=2}^{x}D\,k\right)\,
		e^{-(D+\barmu)\,\left(x\,u_1+\sum_{i=1}^{x-2}(x-i)\,(u_{i+1}-u_{i})\right)}
\\
	 &  \hspace{1.5cm}
	\PP_{(x,s)}
	\left(T_{L,[\Smin,\Smax]}\leq\tau_0 \ | \  
	\mathcal{E}_D(u_1,\dots, u_{x-1}) \right)\,
		\dif u_{x-1}\,\cdots \dif u_1\,.
\end{align*}

On the first hand, on the event $\mathcal{E}_D(u_1,\dots, u_{x-1})\cap \left\{(X_0,S_0)=(x,s)\right\}$,  with $u_{x-1} \leq \delta$, the substrate concentration at time $u_{x-1}$ verifies
 \[
 S_{u_{x-1}}=\Psi(x,s,(u_i,x-i)_{1\leq i \leq x-2},u_{x-1}))\geq  s-\delta\,\max\{D\,\Sin,\,k\,\barmu\,L\},
 \]
 where we recall that $\Psi$ was defined by \eqref{eq:defPsi}. Indeed, we more generally have that, for all $t \in [0,\delta]$,
\[
	S_t = s + \int_0^t (D\,(\Sin-S_u)-k\,\mu(S_u)\,X_u)\,\dif u
		\geq s-\delta\,k\,\barmu\,L\,.
\]	

On the second hand, at the end of the washout phase, either $S_{u_{x-1}}\geq \Smin$ and then $T_{L,[\Smin,\Smax]}< u_{x-1}\leq\tau_0$ or $S_{u_{x-1}}< \Smin$ and then $T_{L,[\Smin,\Smax]}\geq u_{x-1}$. Applying the Markov Property as well as \eqref{eq.reach.Smin.Smax.befor.tmax.x1} in the last case, we obtain
\begin{align*}
& \PP_{(x,s)}
	\left(T_{L,[\Smin,\Smax]}\leq\tau_0 \ | \ 	\mathcal{E}_D(u_1,\dots, u_{x-1}) \right)
\\
	&\qquad= \1_{\{\Psi(x,s,(u_i,x-i)_{1\leq i \leq x-2},u_{x-1}))\geq \Smin\}}
\\
	&\qquad \quad
			+ \1_{\{\Psi(x,s,(u_i,x-i)_{1\leq i \leq x-2},u_{x-1}))< \Smin\}}\,
\\
	&\qquad \quad
		\times \PP_{(1,\Psi(x,s,(u_i,x-i)_{1\leq i \leq x-2},u_{x-1}))}
	\left(T_{L,[\Smin,\Smax]}\leq\tau_0-u_{x-1}\right)
\\
& \qquad \geq e^{-(D+\barmu)\,(\tau_0-\delta)}
\end{align*}
and then
\begin{align}
\label{eq.reach.Smin.Smax.befor.tmax.xlarger1}
\PP_{(x,s)}\left(T_{L,[\Smin,\Smax]}\leq \tau_0\right)
	&\geq e^{-(D+\barmu)\,(\tau_0-\delta)}\,P_d(x,x-1,\delta)\,.
\end{align}
\end{itemize}

Proof of Item~\ref{lemma.reach.T.Smin.Smax.s_sup_Smax}. If $s \geq \Smax$, one way for the process to reach $B(L,[\Smin,\Smax])$ before $\tau_0$ is if the population jumps from $x$ to $L$ by $L-x$ successive division events during the time duration $\delta:=(\tau_0-\phit\left(L, s,\Smax\right))\,\frac{D\,|\bar s_L-s|}{D\,|\bar s_L-s|+\max\{D\,\Sin,\,k\,\barmu\,L\}}$ and if then no event occurs during the time duration $\tau_0-\delta$. 
We omit the details of the proof whose the sketch is exactly the same as for $s\leq \Smin$ and leads to
\begin{itemize}
\item if $x=L$, for all $s_0 \in [\Smax,s+\delta\,\max\{D\,\Sin,\,k\,\barmu\,L\}]$
\begin{align*}
\PP_{(L,s_0)}\left(T_{L,[\Smin,\Smax]}\leq \tau_0 \right)
	&\geq
	\PP_{(L,s_0)}\left(T_{L,[\Smin,\Smax]}\leq \tau_0-\delta \right)
	\geq
	e^{-(D+\barmu)\,(\tau_0-\delta)\,L}\,;
\end{align*}

\item if $x<L$, remarking that $\Psi(x,s,(u_i,x+i)_{1\leq i \leq k-1},u_k)\geq \bar s_L$ for all $1\leq k\leq L-x$ in the term below, as $\mu$ is increasing
\begin{align*}
\PP_{(x,s)}\left(T_{L, [\Smin,\Smax]}\leq \tau_0\right)
	&\geq e^{-(D+\barmu)\,(\tau_0-\delta)\,L}\,
\\
&   \quad \times
 	\int_0^{\delta} \int_{u_1}^{\delta}\dots \int_{u_{L-x-1}}^{\delta}
	e^{-(D+\barmu)\,\left(x\,u_1+
	\sum_{i=1}^{L-x-1}(x+i)\,(u_{i+1}-u_{i})\right)}\,
\\
	 &   \hspace{1cm} \times
	\left(\prod_{k=1}^{L-x}\mu\left(\Psi(x,s,(u_i,x+i)_{1\leq i \leq k-1},u_k)\right)\,(x+k-1)\right)
\,	
\\
	 &   \hspace{1cm} \times
	\dif u_{L-x}\dots\,\dif u_1
\\
	&\geq
	e^{-(D+\barmu)\,(\tau_0-\delta)\,L}\,\left(\frac{\mu(\bar s_L)}{\barmu}\right)^{L-x}\,
		P_b(x,L-x, \delta)\,.
\end{align*}
\end{itemize}

Proof of Item~\ref{lemma.reach.T.Smin.Smax.s_in_Smin_Smax}. If $s\in (\Smin,\Smax)$, in order that the process reaches $B(L,[\Smin,\Smax])$, it is necessary for the process $(S_t)_{t\geq 0}$ to exit $[\Smin,\Smax]$ and come back to this set.

If $\tau_0>\phit\left(1, s,\Smax\right)$, we will bound from below the probability that the process exits $(\Smin,\Smax)$ by the bound $\Smax$, at time $T_{L,[\Smax,\Smax]}$ (that is we also impose that the bacterial population is in $\llbracket 1, L\rrbracket$ at this exit time) before the time $\tau_0-\frac{\tau_0-\phit\left(1, s,\Smax\right)}{2}$ and then comes back to $[\Smin,\Smax]$ during the time interval $(T_{L,[\Smax,\Smax]},\tau_0]$.
We obtain
\begin{align}
\nonumber
&\PP_{(x,s)}\left(T_{L,[\Smin,\Smax]}\leq \tau_0\right)
\\
\nonumber
&\qquad \geq
\PP_{(x,s)}\left(\{T_{L,[\Smin,\Smax]}\leq\tau_0\}
	\cap \left\{T_{L,[\Smax,\Smax]}\leq \tau_0-\frac{\tau_0-\phit\left(1, s,\Smax\right)}{2}\right\}\right)
\\
\nonumber
& \qquad\geq
\PP_{(x,s)}\left(T_{L,[\Smax,\Smax]}\leq \tau_0-\frac{\tau_0-\phit\left(1, s,\Smax\right)}{2}\right)
\\
\label{lemma.reachS.proba.decompo}
& \qquad\qquad \times
\PP_{(x,s)}\left(T_{L,[\Smin,\Smax]}\leq\tau_0\ | \  T_{L,[\Smax,\Smax]}\leq \tau_0-\frac{\tau_0-\phit\left(1, s,\Smax\right)}{2} \right)\,.
\end{align}

On the one hand, as $\tau_0>\tau_0-\frac{\tau_0-\phit\left(1, s,\Smax\right)}{2}>\phit\left(1, s,\Smax\right)$, from Lemma~\ref{lemma.reach.T.Smin.Smax}-\ref{lemma.reach.T.Smin.Smax.s_inf_Smin} we have
\begin{align}
\nonumber
\PP_{(x,s)}&\left(T_{L,[\Smax,\Smax]}\leq\tau_0-\frac{\tau_0-\phit\left(1, s,\Smax\right)}{2}\right)
\\ \label{lemma.reachS.exit}
	&\geq e^{-(D+\barmu)\,\left(\tau_0-\frac{\tau_0-\phit\left(1, s,\Smax\right)}{2}-\delta_1\right)}\,
		P_d(x,x-1, \delta_1)
\end{align}
with $\delta_1:=\frac{\tau_0-\phit\left(1, s,\Smax\right)}{2}\,\frac{D\,|\bar s_1-s|}{D\,|\bar s_1-s|+\max\{D\,\Sin,\,k\,\barmu\,L\}}$.

\medskip

On the other hand, from the definition of $T_{L,[\Smax, \Smax]}$, $(X_{T_{L,[\Smax,\Smax]}},S_{T_{L,[\Smax,\Smax]}})\in \llbracket 1, L \rrbracket \times \{\Smax\}$, then by the law of total probability
\begin{align*}
&\PP_{(x,s)}\left(T_{L,[\Smin,\Smax]}\leq\tau_0\ | \  T_{L,[\Smax,\Smax]}\leq \tau_0-\frac{\tau_0-\phit\left(1, s,\Smax\right)}{2} \right)
\\
&\quad=
	\sum_{i=1}^L 
	\PP_{(x,s)}
	\left(T_{L,[\Smin,\Smax]}\leq\tau_0 \ | \ T_{L,[\Smax,\Smax]}\leq \tau_0-\frac{\tau_0-\phit\left(1, s,\Smax\right)}{2},\,
		X_{T_{L,[\Smax,\Smax]}}=i\right)
\\
	&\qquad\qquad  \times
	\PP_{(x,s)}
	\left(X_{T_{L,[\Smax,\Smax]}}=i \ | \  T_{L,[\Smax,\Smax]}\leq \tau_0-\frac{\tau_0-\phit\left(1, s,\Smax\right)}{2} \right).
\end{align*}
Set $A_i:=\left\{ T_{L,[\Smax,\Smax]}\leq \tau_0-\frac{\tau_0-\phit\left(1, s,\Smax\right)}{2},\,X_{T_{L,[\Smax,\Smax]}}=i\right\}$, Markov property entails now
\begin{align*}
&\PP_{(x,s)}
	\left(T_{L,[\Smin,\Smax]}\leq\tau_0 \ | \ A_i \right)
\\
&\quad\geq
	\PP_{(x,s)}
	\left(T_{L,[\Smin,\Smax]}\leq\tau_0 \ | \ A_i,  T_{L,[\Smax,\Smax]}\leq T_{L,[\Smin,\Smax]}\right)\,
	\PP_{(x,s)}\left(T_{L,[\Smax,\Smax]}\leq T_{L,[\Smin,\Smax]} \ | \ A_i \right)
\\
&\qquad +
	\PP_{(x,s)}
	\left(T_{L,[\Smin,\Smax]}\leq\tau_0 \ | \ A_i, T_{L,[\Smax,\Smax]}> T_{L,[\Smin,\Smax]} \right)\,
	\PP_{(x,s)}\left(T_{L,[\Smax,\Smax]}> T_{L,[\Smin,\Smax]} \ | \ A_i \right)
\\
&\quad\geq
	\PP_{(i,\Smax)}
	\left(T_{L,[\Smin,\Smax]}\leq \frac{\tau_0-\phit\left(1, s,\Smax\right)}{2}\right)\,
	\PP_{(x,s)}\left(T_{L,[\Smax,\Smax]}\leq T_{L,[\Smin,\Smax]} \ | \ A_i \right)
\\
&\qquad +
	1\times\,
	\PP_{(x,s)}\left(T_{L,[\Smax,\Smax]}> T_{L,[\Smin,\Smax]} \ | \ A_i \right)
\\
&\quad\geq
	\PP_{(i,\Smax)}
	\left(T_{L,[\Smin,\Smax]}\leq \frac{\tau_0-\phit\left(1, s,\Smax\right)}{2}\right).
\end{align*}
In addition, for all $i \in \llbracket 1, L \rrbracket$, from Lemma~\ref{lemma.reach.T.Smin.Smax}-\ref{lemma.reach.T.Smin.Smax.s_sup_Smax} applied to $\frac{\tau_0-\phit\left(1, s,\Smax\right)}{2}>0=\phit(L,\Smax,\Smax)$,  
\begin{multline*}
\PP_{(i,\Smax)}\left(T_{L,[\Smin,\Smax]}\leq \frac{\tau_0-\phit\left(1, s,\Smax\right)}{2}\right)\\
	\geq
	e^{-(D+\barmu)\,\left(\frac{\tau_0-\phit\left(1, s,\Smax\right)}{2}-\delta_2\right)\,L}\,
	\left(\frac{\mu(\bar s_L)}{\barmu}\right)^{L-1}\,P_b(1,L-1, \delta_2)
\end{multline*}
with $\delta_2:=\frac{\tau_0-\phit\left(1, s,\Smax\right)}{2}\,\frac{D\,|\bar s_L-s|}{D\,|\bar s_L-s|+\max\{D\,\Sin,\,k\,\barmu\,L\}}$. Therefore
\begin{multline}
\PP_{(x,s)}\left(T_{L,[\Smin,\Smax]}\leq\tau_0\ | \  T_{L,[\Smax,\Smax]}\leq \tau_0-\frac{\tau_0-\phit\left(1, s,\Smax\right)}{2} \right)
\\
\label{lemma.reachS.integrand}
\geq 
	e^{-(D+\barmu)\,\left(\frac{\tau_0-\phit\left(1, s,\Smax\right)}{2}-\delta_2\right)\,L}\,
	\left(\frac{\mu(\bar s_L)}{\barmu}\right)^{L-1}\,P_b(1,L-1, \delta_2)\,.
\end{multline}

Finally, from \eqref{lemma.reachS.proba.decompo}, \eqref{lemma.reachS.exit} and \eqref{lemma.reachS.integrand}
\begin{align*}
\PP_{(x,s)}\left(T_{L,[\Smin,\Smax]}\leq \tau_0\right)
& \geq
e^{-(D+\barmu)\,\left(\tau_0-\frac{\tau_0-\phit\left(1, s,\Smax\right)}{2}-\delta_1\right)}\,
		P_d(x,x-1, \delta_1)
\\
& \qquad \times
e^{-(D+\barmu)\,\left(\frac{\tau_0-\phit\left(1, s,\Smax\right)}{2}-\delta_2\right)\,L}\,
	\left(\frac{\mu(\bar s_L)}{\barmu}\right)^{L-1}\,P_b(1,L-1, \delta_2)
\\
& \geq
	e^{-(D+\barmu)\,\left(\tau_0-\delta_1-\delta_2\right)\,L}\,
		P_d(L,L-1, \delta_1)\,
	\left(\frac{\mu(\bar s_L)}{\barmu}\right)^{L-1}\,P_b(1,L-1, \delta_2)\,.
\end{align*}

\bigskip

 If $\tau_0>\phit\left(L, s,\Smin\right)$, we can bound from below the probability that the substrate process exits $(\Smin,\Smax)$ by the bound $\Smin$, at time $T_{L,[\Smin,\Smin]}$ before the time $\tau_0-\frac{\tau_0-\phit\left(L, s,\Smin\right)}{2}$ and then comes back to $[\Smin,\Smax]$ during the time intervalle $(T_{L,[\Smin,\Smin]},\tau_0]$.
In the same way as for $\tau_0>\phit\left(1, s,\Smax\right)$, we obtain
\begin{align*}
& \PP_{(x,s)}\left(T_{L,[\Smin,\Smax]}\leq \tau_0\right)\\
& \qquad \geq
\PP_{(x,s)}\left(T_{L,[\Smin,\Smin]}\leq \tau_0-\frac{\tau_0-\phit\left(L, s,\Smin\right)}{2}\right)
\\
& \qquad \quad \times
\PP_{(x,s)}\left(T_{L,[\Smin,\Smax]}\leq\tau_0 \ | \   T_{L,[\Smin,\Smin]}\leq \tau_0-\frac{\tau_0-\phit\left(L, s,\Smin\right)}{2} \right)
\\
&  \qquad \geq
	e^{-(D+\barmu)\,\left(\tau_0-\frac{\tau_0-\phit\left(L, s,\Smin\right)}{2}-\delta_2\right)\,L}\,
		\left(\frac{\mu(\bar s_L)}{\barmu}\right)^{L-x}\,P_b(x,L-x, \delta_2)
\\
& \qquad \quad 
	\times e^{-(D+\barmu)\,\left(\frac{\tau_0-\phit\left(L, s,\Smin\right)}{2}-\delta_1\right)}\,
	P_d(L,L-1, \delta_1)
\\
& \qquad \geq
	e^{-(D+\barmu)\,\left(\tau_0-\delta_1-\delta_2\right)\,L}\,
		P_d(L,L-1, \delta_1)\,
	\left(\frac{\mu(\bar s_L)}{\barmu}\right)^{L-1}\,P_b(1,L-1, \delta_2)
\end{align*}
with $\delta_1:=\frac{\tau_0-\phit\left(L, s,\Smin\right)}{2}\,\frac{D\,|\bar s_1-s|}{D\,|\bar s_1-s|+\max\{D\,\Sin,\,k\,\barmu\,L\}}$ and $\delta_2:=\frac{\tau_0-\phit\left(L, s,\Smin\right)}{2}\,\frac{D\,|\bar s_L-s|}{D\,|\bar s_L-s|+\max\{D\,\Sin,\,k\,\barmu\,L\}}$.
\end{proof}

\subsection{Proof of Lemma~\ref{lemma.stay.in.S}}
\label{sec:stay.in.S}
 
Assuming $0<\varepsilon\leq \frac{3\,\min\{\SminK-\bar s_{\lmax},\, \bar s_1-\SmaxK\}}{\max\{D\,\Sin,\, k\,\barmu\,\lmax\}}$ ensures that $[\phis(y,r,\frac{\varepsilon}{3}),\phis(y,r,\frac{\varepsilon}{4})]\subset [\bar s_{\lmax}, \bar s_1]$ from Lemma~\ref{lemme.control.phi_moins_s}-\ref{lemme.control.phi_moins_s.maj} and Remark~\ref{phis0.less.sin}. 
Moreover, remarking that,
\[
\phis\left(y,r , \frac{\varepsilon}{3}\right) 
	= \phis\left(y,\phis\left(y,r , \frac{\varepsilon}{4}\right)  , \frac{\varepsilon}{3}-\frac{\varepsilon}{4}\right)\,,
\]
from Lemma~\ref{lemme.control.phi_moins_s}-\ref{lemme.control.phi_moins_s.min} we have
\begin{align*}
\left|\phis\left(y, r, \frac{\varepsilon}{3}\right) - \phis\left(y, r, \frac{\varepsilon}{4}\right)\right| 
	& \geq 
	D\,\left|\phis\left(y,r, \frac{\varepsilon}{4}\right) - \bar s_y\right|\,\frac{\varepsilon}{12}
\\
	& = 
	D\,\left(\left|\phis\left(y,r, \frac{\varepsilon}{4}\right) - r\right| + |r- \bar s_y|\right)\,\frac{\varepsilon}{12}
\\
	& \geq
	D\,\left(D\,\frac{\varepsilon}{4}+1\right)|r- \bar s_y|\,\frac{\varepsilon}{12}\,.
\end{align*}
Lemma~\ref{lemma.stay.in.S} is then a consequence of Lemma~\ref{lemma.stay.in.S.general} below with $\beta= D\,(D\,\frac{\varepsilon}{4}+1)\,\delta\,\frac{\varepsilon}{12}$. Lemma~\ref{lemma.stay.in.S.general} states that the probability that the process stays in an interval can be bounded from below by a constant which only depends on the interval length.

\begin{lemma}
\label{lemma.stay.in.S.general}
Let $\beta>0$, $L\in\NN^*$ and $T>0$. Then there exists $C_{\ref{lemma.stay.in.S.general}}>0$ such that for all $\Smin$ and $\Smax$ such that $\bar s_L\leq \Smin < \Smax\leq \bar s_1$ and $\Smax-\Smin=\beta$, for all $(x,s) \in B(L,[\Smin,\Smax])$,
\[
\PP_{(x,s)}
		\left((X_t,S_t) \in \setunL\times [\Smin, \, \Smax] ,\,
				\forall t\in[0,T]\right)\geq C_{\ref{lemma.stay.in.S.general}}\,.
\]
\end{lemma}
\begin{proof}
Let $\ell:= \max\{l \in \NN^* \text{ such that } \bar s_l \geq\Smin\}$,  and let $s_1$ and $s_2$ such that $\Smin<s_1<s_2<\Smax$. Note that, from Lemma~\ref{lem:defsbar}, $1\leq \ell \leq L-1$.
We aim to show that Inequalities~\eqref{minoration.sl.in.S} and \eqref{minoration.sl.notin.S} below  hold. Namely, if $\bar s_\ell \in [\Smin, \, \Smax]$, then
\begin{align}
\nonumber
& \PP_{(x,s)}
		\left((X_t,S_t) \in \setunL\times [\Smin, \, \Smax],\,
				\forall t\in[0,T]\right)
\\
\label{minoration.sl.in.S}
	&\qquad \geq
	e^{-(D+\barmu)\,L\,T}\, \min\left\{
	P_d(L,L-1, t_{\Smax-\Smin});\,
	\left(\frac{\mu(\bar s_L)}{\barmu}\right)^{L-1}\,
P_b(1,L-1, t_{\Smax-\Smin})
	\right\}\,,
\end{align}
with $t_{\Smax-\Smin}=|\Smax-\Smin|/\max\{D\,\Sin,\, k\,\barmu\,L\}$
and if $\bar s_\ell \notin [\Smin, \, \Smax]$, then
\begin{align}
\nonumber
& \PP_{(x,s)}
		\left((X_t,S_t) \in \setunL\times [\Smin, \, \Smax],\,
				\forall t\in[0,T]\right)
\\
\label{minoration.sl.notin.S}
	&\qquad \geq
	C^{\left\lfloor\frac{T}{\gamma}\right\rfloor+1}\, 
	\min\left\{P_d(L,L-1, t_1);\,
	\left(\frac{\mu(\bar s_L)}{\barmu}\right)^{L-1}\,P_b(1,L-1, t_2)\right\}
\end{align}
where the preceding constants are defined by
\begin{align*}
C &=
	  \frac{\mu(\bar s_L)\,D}{(D+\barmu)^2}\, e^{-(D+\barmu)\,\ell\,\phit(\ell,\Smin,s_2)} \,
	e^{-(D+\barmu)\,(\ell+1)\,\phit(\ell+1,\Smax,s_1)}\,
\\
& \quad
	\times \left[1-e^{-(D+\barmu)\,\ell\,\phit(\ell,s_2,\Smax)}\right]\,
	\left[1-
	e^{-(D+\barmu)\,(\ell+1)\,\phit(\ell+1,s_1,\Smin)}\right]\;
\end{align*}
and
\[
\gamma = \phit(\ell,s_1,s_2)+\phit(\ell+1,s_2,s_1)\,;
\]
\[
t_1 = \frac{|\Smax-s_2|}{\max\{D\,\Sin,\, k\,\barmu\,L\}}\,;\qquad
t_2 = \frac{|s_1-\Smin|}{\max\{D\,\Sin,\, k\,\barmu\,L\}}\,.
\]
Remarking that, if $\bar s_\ell \notin [\Smin, \, \Smax]$, then $|\bar s_\ell-s_2|\geq |\Smax-s_2|$ and $|s_1- \bar s_{\ell+1}|\geq |s_1-\Smin|$, we obtain from  Lemma~\ref{lemme.control.phi_moins_1}, remarking in addition that in this case $1\leq \ell \leq L-1$,
\[
C \geq \frac{\mu(\bar s_L)\,D}{(D+\barmu)^2}\, 
	e^{-\frac{D+\barmu}{D}\,L\,\left(\frac{s_2-\Smin}{\Smax-s_2}
	+\frac{\Smax-s_1}{s_1-\Smin}\right)} \,
	\left[1-e^{-\frac{(D+\barmu)\,(\Smax-s_2)}{\max\{D\,\Sin,\,k\,\barmu\,L\}}}\right]\,
	\left[1-e^{-\frac{(D+\barmu)\,(s_1-\Smin)}{\max\{D\,\Sin,\,k\,\barmu\,L\}}}\right]
\]
and
\[
\gamma \geq \frac{2\,|s_2-s_1|}{\max\{D\,\Sin, \, k\,\barmu\,L\}}\,.
\]
In particular, choosing $s_1 = \Smin+(\Smax-\Smin)/4$ and $s_2 = \Smin+3\,(\Smax-\Smin)/4$, Lemma~\ref{lemma.stay.in.S.general} holds with
\begin{align*}
C_{\ref{lemma.stay.in.S.general}} &= \min\Bigg\{
	\left(\frac{\mu(\bar s_L)\,D}{(D+\barmu)^2}\, 
	e^{-6\,\frac{D+\barmu}{D}\,L} \,
	\left[1-e^{-\frac{(D+\barmu)\,\beta}{4\,\max\{D\,\Sin,\,k\,\barmu\,L\}}}\right]^2\right)
	^{\frac{\max\{D\,\Sin, \, k\,\barmu\,L\}\,T}{\beta}+1}\,;\,
\\
	&\qquad\qquad\qquad
	e^{-(D+\barmu)\,L\,T} \Bigg\}
\\
	&\quad
	\times \min\left\{
	P_d(L,L-1, t_{\ref{lemma.stay.in.S.general}});\,
	\left(\frac{\mu(\bar s_L)}{\barmu}\right)^{L-1}\,
P_b(1,L-1, t_{\ref{lemma.stay.in.S.general}})
	\right\}
\\
	&>0
\end{align*}
with $t_{\ref{lemma.stay.in.S.general}} = \beta/(4\,\max\{D\,\Sin,\, k\,\barmu\,L\})$.

\medskip

So let us prove, first, that if $\bar s_\ell \in [\Smin, \, \Smax]$ then \eqref{minoration.sl.in.S} holds and, second, that if $\bar s_\ell \notin [\Smin, \, \Smax]$ then \eqref{minoration.sl.notin.S} holds.
 To prove \eqref{minoration.sl.notin.S}, we first show that $C^{\left\lfloor\frac{T}{\gamma}\right\rfloor+1}$ is a lower bound for $x=\ell$ and $s<s_1$ (including $(x,s)=(\ell,\Smin)$) and for $x=\ell+1$ and $s>s_2$ (including $(x,s)=(\ell+1,\Smax)$); we then deduce the result for $x \neq \ell$ and $s=\Smin$ and for $x\neq \ell+1$ and $s=\Smax$, with $(x,s)\in B(L,[\Smin,\Smax])$ reaching one of both previous cases by successive washout or division events; then leading to \eqref{minoration.sl.notin.S} for any possible initial condition in $B(L,[\Smin,\Smax])$.

\paragraph{If $\bar s_\ell \in [\Smin, \, \Smax]$:}
\begin{itemize}
\item If $x=\ell$: If no event occurs during $[0,T]$, then by Lemma~\ref{lem:rapprochement}, for all $s_0 \in [\Smin, \, \Smax]$, the process starting from $(\ell, s_0)$ stays in $\{\ell\}\times [s_0, \, \bar s_\ell]\subset \{\ell\}\times [\Smin, \, \Smax]$. Hence,
\begin{align}
\nonumber
\PP_{(\ell,s_0)}
		\left((X_t,S_t) \in \setunL\times [\Smin, \, \Smax], \, 
				\forall t\in[0,T]\right)
	&\geq \PP_{(\ell,s_0)}(T_1 \geq T)
\\
\nonumber
&\geq
e^{-(D+\barmu)\,\ell\,T}
\\
&\geq
e^{-(D+\barmu)\,L\,T}\,.
\label{eq.stayinS.noevent}
\end{align}

\item If $x > \ell$: From Lemma~\ref{prop.minore.proba.evt},
\begin{align*}
\nonumber
& \PP_{(x,s)}
	\left((X_t,S_t) \in\setunL \times [\Smin, \, \Smax], \, 
				\forall t\in[0,T]\right)
\\
& \qquad
	\geq \int_0^{t_{\Smax-\Smin}} \int_{u_1}^{t_{\Smax-\Smin}}\dots \int_{u_{x-\ell-1}}^{t_{\Smax-\Smin}}
	\left(\prod_{k=\ell+1}^{x}D\,k\right)\,
	e^{-(D+\barmu)\,\left(x\,u_1+
	\sum_{i=1}^{x-\ell-1}(x-i)\,(u_{i+1}-u_{i})\right)}\,	
\\
	 &  \hspace{1.5cm}
	\PP_{(x,s)}
	\Big((X_t,S_t) \in \setunL\times [\Smin, \, \Smax], \, 
				\forall t\in[0,T] \ \Big| \ \mathcal{E}_D(u_1,\dots,u_{x-\ell}) \Big)\,
\\
	 &  \hspace{1.5cm}
	\dif u_{x-\ell}\dots\,\dif u_2\,\dif u_1\,.
\end{align*}
As $(x,s)\in B(L,[\Smin,\Smax])$, we easily check from Lemma~\ref{lemme.control.phi_moins_1} that, on the event $\{(X_0,S_0)=(x,s)\}\cap\mathcal{E}_D(u_1,\dots,u_{x-\ell})$, the process $(X_t,S_t)_{0\leq t\leq u_{x-\ell}}$ stays in $\setunL \times [\Smin, \, \Smax]$ for $u_{x-\ell}\leq t_{\Smax-\Smin}$.
By the Markov property and \eqref{eq.stayinS.noevent} we then obtain, for $s_0=\Psi(x,s,(u_i,x-i)_{1\leq i \leq x-\ell-1},u_{x-\ell})\in [\Smin, \, \Smax]$:
\begin{align*}
&\PP_{(x,s)}
	\Big((X_t,S_t) \in \setunL\times [\Smin, \, \Smax], \, 
				\forall t\in[0,T] \ \Big| \ \mathcal{E}_D(u_1,\dots,u_{x-\ell}) \Big)
\\
&\qquad
	= \PP_{(\ell,s_0)}
	\Big((X_t,S_t) \in \setunL \times [\Smin, \, \Smax], \, 
				\forall t\in[0,(T-u_{x-\ell})\vee 0]  \Big)
\\
&\qquad \geq
	e^{-(D+\barmu)\,L\,T}\,,
\end{align*}
and therefore 
\begin{align*}
\PP_{(x,s)}
	\left((X_t,S_t) \in \setunL \times [\Smin, \, \Smax], \, 
				\forall t\in[0,T]\right)
& \geq
	e^{-(D+\barmu)\,L\,T}\, P_d(x,x-\ell, t_{\Smax-\Smin})
\\ & \geq
	e^{-(D+\barmu)\,L\,T}\, P_d(L,L-1, t_{\Smax-\Smin})\,.
\end{align*}

\item If $x < \ell$: in the same way, replacing the washouts event condition $\mathcal{E}_D(u_1,\dots,u_{x-\ell})$ by the divisions event condition $\mathcal{E}_B(u_1,\dots,u_{\ell-x})$ in the previous case, we obtain
\begin{align*}
\nonumber
& \PP_{(x,s)}
	\left((X_t,S_t) \in \setunL \times [\Smin, \, \Smax], \, 
				\forall t\in[0,T]\right)
\\
 & \qquad \geq
	e^{-(D+\barmu)\,L\,T}\, \left(\frac{\mu(\bar s_L)}{\barmu}\right)^{\ell-x}
	\,P_b(x,\ell-x, t_{\Smax-\Smin})
\\
&
	\qquad\geq
	e^{-(D+\barmu)\,L\,T}\, \left(\frac{\mu(\bar s_L)}{\barmu}\right)^{L-1}\,
P_b(1,L-1, t_{\Smax-\Smin})
\end{align*}
and then \eqref{minoration.sl.in.S} holds.
\end{itemize}

\bigskip

\paragraph{If $\bar s_\ell \notin [\Smin, \, \Smax]$:} By definition,  $\ell$ is such that $\bar s_\ell > \Smax$ and  $\bar s_{\ell+1} < \Smin$. Note that throughout this part of the proof, we will use the following properties (see Corollary~\ref{cor:monotonyphit}): for all $\Smin \leq r_0 \leq r_1 \leq r_2 \leq \Smax$,
\[
	\phit(\ell, r_0, r_1) \leq \phit(\ell, r_0, r_2)<+\infty, \qquad 
	\phit(\ell+1, r_1,r_0) \leq \phit(\ell+1,r_2,r_0)<+\infty\,.
\]
\begin{itemize}
\item If $x=\ell$ and $\Smin\leq s\leq s_1$: We prove that 
\begin{align}
\label{eq.p.fct.s1s2}
\PP_{(\ell,s)}
		\left((X_t,S_t) \in \setunL \times [\Smin, \, \Smax], \, 
				\forall t\in[0,T]\right) \geq C^{\left\lfloor\frac{T}{\gamma}\right\rfloor+1}\,.
\end{align}

One way for the substrate concentration process $(S_t)_{t\in[0,T]}$ to stay in $[\Smin, \, \Smax]$ is if the first event is a division and occurs at time $T_1 \in [\phit(\ell,s,s_2),\,  \phit(\ell,s,\Smax))$,  the second event is a washout and occurs at time $T_2 \in [T_1+\phit(\ell+1,S_{T_1},s_1),\,  T_1+\phit(\ell+1,S_{T_1},\Smin))$ and if the process $(S_t)_{T_2 \leq t \leq T\vee T_2}$ stays in $[\Smin, \, \Smax]$. In fact, we easily check that on this event
\[
(X_t,S_t) \in
\begin{cases}
	    \{\ell\}\times[s,\Smax) & \text{if } 0 \leq t < T_1\\
	    \{\ell+1\}\times[s_2,\Smax) & \text{if }  t= T_1\\
	 \{\ell+1\}\times(\Smin,\, \Smax) & \text{if } T_1 \leq t \leq T_2
\end{cases}\,.
\]

Therefore, from Lemma~\ref{prop.minore.proba.evt} and the Markov Property
\begin{align}
\nonumber
& \PP_{(\ell,s)}
		\left((X_t,S_t) \in \setunL \times [\Smin, \, \Smax], \, 
				\forall t\in[0,T]\right)
\\
\nonumber
& \quad \geq
	\PP_{(\ell,s)}
		\Big(\big\{X_{T_{1}}=\ell+1\big\}\cap \big\{\phit(\ell,s,s_2)\leq T_1\leq  \phit(\ell,s,\Smax)  \big\}
\\
\nonumber
& \quad \qquad \qquad
	\cap \big\{X_{T_{2}}=\ell\big\}\cap \big\{\phit(\ell+1,S_{T_1},s_1)\leq T_2-T_1\leq  \phit(\ell+1,S_{T_1},\Smin)\big\} 
\\
\nonumber
& \quad \qquad \qquad
	\cap \left\{(X_t,S_t) \in \setunL \times [\Smin, \, \Smax], \, 
				\forall t\in[0,T]\right\}\Big)
\\
\nonumber
&\quad \geq
	\mu(\bar s_L) \,\ell\,
	\int_{\phit(\ell,s,s_2)}^{\phit(\ell,s,\Smax)}
	e^{-(D+\barmu)\,\ell\,u_1} \, 
	D\,(\ell+1)\,
	\int_{\phit(\ell+1,\phi(\ell,s,u_1),s_1)}^{\phit(\ell+1,\phi(\ell,s,u_1),\Smin)}
	e^{-(D+\barmu)\,(\ell+1)\,u_2}
\\
 \nonumber
& \quad \qquad
	\PP_{(\ell,\phi(\ell+1,\phi(\ell,s,u_1),u_2))}
		\left((X_t,S_t) \in \setunL \times [\Smin, \, \Smax], \, 
				\forall t\in[0,(T-u_1-u_2)\vee 0]\right)
\\
\label{expression.k_plus_un}
& \quad \qquad
\,\dif u_2\,\dif u_1\,.
\end{align}

Assumption $s<s_1$ implies that $u_1\geq \phit(\ell,s_1,s_2)$, moreover $u_1\geq \phit(\ell,s,s_2)$ implies that $u_2 \geq \phit(\ell+1,s_2,s_1)$.
Hence $T-u_1-u_2 \leq T-\gamma$.
In addition, $u_2\in [\phit(\ell+1,\phi(\ell,s,u_1),s_1),\, \phit(\ell+1,\phi(\ell,s,u_1),\Smin)]$ implies that
$\phi(\ell+1,\phi(\ell,s,u_1),u_2)\in[\Smin,\,s_1]$.
Then, in order to obtain \eqref{eq.p.fct.s1s2}, by recurrence, it is suffisant to prove that
\begin{multline}
\mu(\bar s_L) \,\ell\,
	\int_{\phit(\ell,s,s_2)}^{\phit(\ell,s,\Smax)}
	e^{-(D+\barmu)\,\ell\,u_1} \, 
	D\,(\ell+1)\,\\
	\times
	\int_{\phit(\ell+1,\phi(\ell,s,u_1),s_1)}^{\phit(\ell+1,\phi(\ell,s,u_1),\Smin)}
	e^{-(D+\barmu)\,(\ell+1)\,u_2}\,\dif u_2\,\dif u_1
 \geq C\,.
\label{proba.minoree.C}
\end{multline}

Remarking that, from Corollary~\ref{cor:monotonyphit}, we have
\[
\phit(\ell+1,\phi(\ell,s,u_1),\Smin)-\phit(\ell+1,\phi(\ell,s,u_1),s_1)=\phit(\ell+1,s_1,\Smin)
\]
we then obtain
\begin{align*}
\nonumber
& D\,(\ell+1)\,
	\int_{\phit(\ell+1,\phi(\ell,s,u_1),s_1)}^{\phit(\ell+1,\phi(\ell,s,u_1),\Smin)}
	e^{-(D+\barmu)\,(\ell+1)\,u_2}\,\dif u_2
\\
\nonumber
& =
	\frac{D}{D+\barmu}\,e^{-(D+\barmu)\,(\ell+1)\,\phit(\ell+1,\phi(\ell,s,u_1),s_1)}\,
	\left(1-e^{-(D+\barmu)\,(\ell+1)\,\phit(\ell+1,s_1,\Smin)}\right)
\\
& \geq
	\frac{D}{D+\barmu}\,e^{-(D+\barmu)\,(\ell+1)\,\phit(\ell+1,\Smax,s_1)}\,
	\left(1-e^{-(D+\barmu)\,(\ell+1)\,\phit(\ell+1,s_1,\Smin)}\right)\,.
\end{align*}
In the same way,
\begin{align*}
\nonumber
 &\mu(\bar s_L) \,\ell\,
	\int_{\phit(\ell,s,s_2)}^{\phit(\ell,s,\Smax)}
	e^{-(D+\barmu)\,\ell\,u_1}\,\dif u_1
\\
\nonumber
&\qquad =
	\frac{\mu(\bar s_L)}{D+\barmu}\,e^{-(D+\barmu)\,\ell\,\phit(\ell,s,s_2)}\,
	\left(1-e^{-(D+\barmu)\,\ell\,\phit(\ell,s_2,\Smax)}\right)
\\
&\qquad \geq
	\frac{\mu(\bar s_L)}{D+\barmu}\,e^{-(D+\barmu)\,\ell\,\phit(\ell,\Smin,s_2)}\,
	\left(1-e^{-(D+\barmu)\,\ell\,\phit(\ell,s_2,\Smax)}\right)\,.
\end{align*}
Hence \eqref{proba.minoree.C} holds.

\item If $x=\ell+1$ and $s_2\leq s\leq \Smax$:
Replacing both steps:
\begin{enumerate}
\item the first event is a division and occurs at time $T_1 \in [\phit(\ell,s,s_2)),\,  \phit(\ell,s,\Smax))$
\item the second event is a washout and occurs at time $T_2 \in [T_1+\phit(\ell+1,S_{T_1},s_1)),\,   T_1+\phit(\ell+1,S_{T_1},\Smin))$
\end{enumerate}
in the proof for $x=\ell$ and $s\leq s_1$ by
\begin{enumerate}
\item the first event is a washout and occurs at time $T_1 \in [\phit(\ell+1,s,s_1)),\,  \phit(\ell+1,s,\Smin))$
\item the second event is a division and occurs at time $T_2 \in [T_1+\phit(\ell,S_{T_1},s_2)),\,  T_1+\phit(\ell,S_{T_1},\Smax))$
\end{enumerate}
gives the same lower bound starting from $x=\ell+1$ and $s\geq s_2$:
\begin{align}
\label{eq.p.fct.s1s2bis}
\PP_{(\ell+1,s)}
		\left((X_t,S_t) \in \setunL \times [\Smin, \, \Smax], \, 
				\forall t\in[0,T]\right) \geq C^{\left\lfloor\frac{T}{\gamma}\right\rfloor+1}\,.
\end{align}

\item If $x\neq \ell+1$ and $s=\Smax$: As $(x,s)\in B(L,[\Smin,\Smax])$, therefore, $x>\ell+1$. Let $t_1 = |\Smax-s_2|/\max\{D\,\Sin,\, k\,\barmu\,L\}$, by Lemma~\ref{prop.minore.proba.evt}, 
\begin{align*}
\nonumber
& \PP_{(x,s)}
	\left((X_t,S_t) \in \setunL \times [\Smin, \, \Smax], \, 
				\forall t\in[0,T]\right)
\\
& \qquad
	\geq \int_0^{t_1} \int_{u_1}^{t_1}\dots \int_{u_{x-\ell-2}}^{t_1}
	\left(\prod_{k=\ell+2}^{x}D\,k\right)\,
	e^{-(D+\barmu)\,\left(x\,u_1+
	\sum_{i=1}^{x-\ell-1}(x-i)\,(u_{i+1}-u_{i})\right)}\,	
\\
	 &  \hspace{1.5cm}
	\PP_{(x,s)}
	\Big((X_t,S_t) \in \setunL\times [\Smin, \, \Smax], \, 
				\forall t\in[0,T]  \ \Big| \ \mathcal{E}_D(u_1,\dots,u_{x-\ell-1})\Big)\,
\\
	 &  \hspace{1.5cm}
	\dif u_{x-\ell-1}\dots\,\dif u_2\,\dif u_1\,.
\end{align*}
As $\bar s_{x-i}\leq \bar s_{\ell+1}<s_2$ for all $i\in \llbracket 1,x-\ell-1 \rrbracket$, we easily check from Lemma~\ref{lemme.control.phi_moins_1} that, on the event $\{(X_0,S_0)=(x,s)\}\cap\mathcal{E}_D(u_1,\dots,u_{x-\ell-1})$, the process $(X_t,S_t)_{0\leq t\leq t_1}$ stays in $\setunL \times [s_2,\Smax]$.
By the Markov property and \eqref{eq.p.fct.s1s2bis} we then obtain, for $s_0=\Psi(x,s,(u_i,x-i)_{1\leq i \leq x-\ell-2},u_{x-\ell-1})\in [s_2,\Smax]$ with $u_{x-\ell-1}\leq t_1$:
\begin{align*}
&\PP_{(x,s)}
	\Big((X_t,S_t) \in \setunL \times [\Smin, \, \Smax], \, 
				\forall t\in[0,T]   \ \Big| \ \mathcal{E}_D(u_1,\dots,u_{x-\ell-1})\Big)
\\
&\qquad
	= \PP_{(\ell+1,s_0)}
	\Big((X_t,S_t) \in \setunL \times [\Smin, \, \Smax], \, 
				\forall t\in[0,(T-u_{x-\ell-1})\vee 0]  \Big)
\\
&\qquad \geq
	C^{\left\lfloor\frac{T}{\gamma}\right\rfloor+1}\,,
\end{align*}
and therefore 
\begin{align*}
\nonumber
& \PP_{(x,s)}
	\left((X_t,S_t) \in \setunL \times [\Smin, \, \Smax], \, 
				\forall t\in[0,T]\right)
\\
 & \qquad \geq
	C^{\left\lfloor\frac{T}{\gamma}\right\rfloor+1}
	\,P_d(x,x-\ell-1, t_1)
\\
&
	\qquad\geq
	C^{\left\lfloor\frac{T}{\gamma}\right\rfloor+1}\, 
P_d(L,L-1, t_1)\,.
\end{align*}

\item If $x\neq \ell$ and $s=\Smin$: Remarking that $(x,s)\in B(L,[\Smin,\Smax])$ implies $x < \ell$, in the same way as the previous case and using \eqref{eq.p.fct.s1s2}, we obtain
\begin{align*}
\nonumber
& \PP_{(x,s)}
	\left((X_t,S_t) \in \setunL \times [\Smin, \, \Smax], \, 
				\forall t\in[0,T]\right)
\\
 & \qquad \geq
	C^{\left\lfloor\frac{T}{\gamma}\right\rfloor+1}\, \left(\frac{\mu(\bar s_L)}{\barmu}\right)^{\ell-x}
	\,P_b(x,\ell-x, t_2)
\\
&
	\qquad\geq
	C^{\left\lfloor\frac{T}{\gamma}\right\rfloor+1}\, \left(\frac{\mu(\bar s_L)}{\barmu}\right)^{L-1}\,
P_b(1,L-1, t_2)
\end{align*}
with $t_2 = |s_1-\Smin|/\max\{D\,\Sin,\, k\,\barmu\,L\}$. \qedhere
\end{itemize}
\end{proof}

\subsection{Proof of Lemma~\ref{lemma.Tys.depuis.S}}
\label{sec:T.depuis.S}

Lemma~\ref{lemma.Tys.depuis.S} is a corollary of the following lemma with $\delta_1=\frac{\varepsilon}{4}$ and $\delta_2=\frac{\varepsilon}{3}$.

\begin{lemma}
\label{reach.y_s_general}
Let $L\in \NN^*$, $\varepsilon>0$ and let $\delta_1, \delta_2$ such that $\varepsilon/2>\delta_2>\delta_1>0$. Then for all $(y,r)\in \setunL \times \ens{\bar s_L}{\bar s_1} \backslash \{(\ell,\bar s_\ell),\, \ell\in \setunL\}$ such that $0<\delta_1\leq \frac{\min\{r-\bar s_{L},\, \bar s_1-r\}}{\max\{D\,\Sin,\, k\,\barmu\,L\}}$ and for all $(x,s) \in \setunL \times \ens{\phis(y,r,\delta_2)}{\phis(y,r,\delta_1)}$
\[
\PP_{(x,s)}
	\Big(T_{y,r} \leq \varepsilon \Big)\geq  
	C^{\varepsilon, \delta_1, \delta_2}_{|\bar s_y -r|}
\]
with
\begin{align*}
 C^{\varepsilon, \delta_1, \delta_2}_{|\bar s_y -r|}
	&= e^{-(D+\barmu)\,L\,\frac{\varepsilon}{2}}
	\times
	\min\left\{P_d(L,L-1, t^\star) \, ; \, 
	\left(\frac{\mu(\bar s_L)}{\barmu}\right)^{L-1}\,P_b(1,L-1, t^\star)\right\}
\end{align*}
where
$
t^\star = 
	\frac{\min \left\{D\,|\bar s_y-r|\,\delta_1 \, , \, 
		D\,(D\,\delta_2+1)\,|\bar s_y-r|\,(\varepsilon/2-\delta_2)\right\}}
		{\max\{D\,\Sin,\,k\,\barmu\,L\}}
$.
\end{lemma}
\begin{proof}
The aim is to prove that, with positive probability, the process goes from $(x,s)$ to $\{y\}\times \ens{\phis(y,r,\varepsilon/2)}{r}$, in a time less than $\varepsilon/2$; and 
then starting from an initial condition in $\{y\}\times \ens{\phis(y,r,\varepsilon/2)}{r}$ it reaches $(y,r)$ in a time less than $\varepsilon/2$.
We then have three cases.

1. If $x=y$ then by definition of $\phis(y,r,.)$, for all $s_0 \in \ens{\phis(y,r, \varepsilon/2)}{r}$ if there is no jump during the time interval $[0,\varepsilon/2]$, then the process starting from $(y,s_0)$ reaches $(y,r)$ before the time $\varepsilon/2$,
then, from Lemma~\ref{prop.minore.proba.evt},
\begin{align}
\PP_{(y,s_0)}
	\Big(T_{y,r} \leq \varepsilon/2 \Big)
&\geq 
	\PP_{(y,s_0)}
	\Big(T_1>\varepsilon/2\Big)
\geq
\label{eq.reach.Tys.fromSandy}
	e^{-(D+\barmu)\,L\frac{\varepsilon}{2}}\,.
\end{align}
As
$\PP_{(y,s)}
	\Big(T_{y,r} \leq \varepsilon \Big) \geq \PP_{(y,s)}
	\Big(T_{y,r} \leq \varepsilon/2 \Big)$ and $s \in \ens{\phis(y,r, \varepsilon/2)}{r}$, then the result holds.

\medskip

2. If $x<y$ then from Lemma~\ref{prop.minore.proba.evt},
\begin{align}
\nonumber
& \PP_{(x,s)}
	\Big(T_{y,r} \leq \varepsilon\Big)
\\\nonumber
& \qquad
	\geq \int_0^{t^\star} \int_{u_1}^{t^\star}\dots \int_{u_{y-x-1}}^{t^\star}
	\left(\prod_{k=y+1}^{x}D\,k\right)\,
	e^{-(D+\barmu)\,\left(x\,u_1+
	\sum_{i=1}^{x-y-1}(x-i)\,(u_{i+1}-u_{i})\right)}\,	
\\
	 &  \qquad\qquad\qquad
	\PP_{(x,s)}
	\Big(T_{y,r} \leq \varepsilon \Big| \ \mathcal{E}_D(u_1,\dots,u_{x-y}) \Big)\,
	\dif u_{x-y}\dots\,\dif u_2\,\dif u_1\,.
 \label{eq.proba.atteinte.depuis.S}
\end{align}
In order to obtain the result, it is suffisant to prove that for all $u_1<\dots <u_{x-y}<t^\star$,
\begin{align}
\label{eq.sub.in.good.domain}
	\Psi(x,s,(u_i,x-i)_{1\leq i \leq x-y-1},u_{x-y}))\in \ens{\phis(y,r, \varepsilon/2)}{r}\,.
\end{align}
Indeed we easily check, using \eqref{eq.bound.travel.dist1} below and Lemma~\ref{lemme.control.phi_moins_s}, that $t^\star\leq \delta_1<\varepsilon/2$. By the Markov property and \eqref{eq.reach.Tys.fromSandy} we then obtain
\begin{align*}
&\PP_{(x,s)}
	\Big(T_{y,r} \leq \varepsilon \Big| \ \mathcal{E}_D(u_1,\dots,u_{x-y}) \Big)
\\
&\qquad\qquad\qquad =
\PP_{(x,s)}
	\Big( u_{x-y} \leq T_{y,r} \leq \varepsilon \Big| \ \mathcal{E}_D(u_1,\dots,u_{x-y}) \Big)
\\
&\qquad\qquad\qquad =
\PP_{(y,\Psi(x,s,(u_i,x-i)_{1\leq i \leq x-y-1},u_{x-y}))}
	\Big(T_{y,r} \leq \varepsilon-u_{x-y} \Big)
	\\
&\qquad\qquad\qquad \geq
\PP_{(y,\Psi(x,s,(u_i,x-i)_{1\leq i \leq x-y-1},u_{x-y}))}
	\Big(T_{y,r} \leq \varepsilon/2 \Big)
\\
&\qquad\qquad\qquad \geq
	e^{-(D+\barmu)\,L\frac{\varepsilon}{2}}\,.
\end{align*}
and
\begin{align*}
\PP_{(x,s)}
	\Big(T_{y,r} \leq \varepsilon\Big)
	&\geq
	e^{-(D+\barmu)\,L\frac{\varepsilon}{2}}\, P_d(x,x-y, t^\star)
	\geq
	e^{-(D+\barmu)\,L\,\frac{\varepsilon}{2}}\, P_d(L,L-1, t^\star)\,.
\end{align*}

Let us prove that \eqref{eq.sub.in.good.domain} holds. More generally, we will prove that for all $n\in \NN$, for all $u_1<\dots< u_{n+1}<t^\star$, for all $(x_i)_{1\leq i \leq n}$ with value in $\setunL$
\begin{align}
\label{eq.sub.in.good.domain.general}
\Psi(x,s,(u_i,x_i)_{1\leq i \leq n},u_{n+1})\in \ens{\phis(y,r, \varepsilon/2)}{r}\,.
\end{align}
By \eqref{eq:defPsi} and \eqref{eq:flow}
\begin{equation}
\label{eq.bound.travel.dist}
\left|\Psi(x,s,(u_i,x_i)_{1\leq i \leq n},u_{n+1})-s\right| \leq t^\star\,\max\{D\,\Sin,\, k\,\barmu\,L\}\,.
\end{equation}

First $\delta_1\leq \frac{\min\{r-\bar s_L,\, \bar s_1-r\}}{\max\{D\,\Sin,\, k\,\barmu\,L\}}$ ensures, from Lemma~\ref{lemme.control.phi_moins_s}-\ref{lemme.control.phi_moins_s.maj} and Remark~\ref{phis0.less.sin} that $\bar s_L\leq\phis(y,r,\delta_1)\leq\bar s_1$, then from Lemma~\ref{lemme.control.phi_moins_s}-\ref{lemme.control.phi_moins_s.min},
\begin{equation}
\label{eq.bound.travel.dist1}
	|r-\phis(y,r,\delta_1)| \geq D\, |\bar s_y-r|\delta_1,
\end{equation}
Second,
\begin{itemize}
\item if $\phis(y,r , \varepsilon/2)>0$, as $\phis$  inherits a flow property from $\phi$, we have $\phis(y,r , \varepsilon/2) = \phis(y,\phis(y,r , \delta_2)  , \varepsilon/2-\delta_2)$. 
Then from Lemma~\ref{lemme.control.phi_moins_s}-\ref{lemme.control.phi_moins_s.min}
\begin{align*}
|\phis(y, r, \varepsilon/2) - \phis(y,r, \delta_2)| 
	& \geq 
	D\,|\phis(y, r, \delta_2) - \bar s_y|\,\left(\frac{\varepsilon}{2}-\delta_2\right)
\\
	& = 
	D\,\left(|\phis(y,r, \delta_2) - r| + |r- \bar s_y|\right)\,\left(\frac{\varepsilon}{2}-\delta_2\right)
\\
	& \geq
	D\,(D\,\delta_2+1)|r- \bar s_y|\,\left(\frac{\varepsilon}{2}-\delta_2\right)\,,
\end{align*}
hence, by \eqref{eq.bound.travel.dist}, the definition of $t^\star$, \eqref{eq.bound.travel.dist1} and the previous inequality,
\begin{align*}
&\left|\Psi(x,s,(u_i,x_i)_{1\leq i \leq n},u_{n+1})-s\right| 
\\
	&\qquad \qquad \leq \min\{|r-\phis(y,r,\delta_1)|\,;\,|\phis(y, r, \varepsilon/2) - \phis(y,r, \delta_2)| \}
\end{align*}
with $s\in \ens{\phis(y,r,\delta_2)}{\phis(y,r,\delta_1)} \subset \ens{\phis(y,r,\varepsilon/2)}{r}$ and \eqref{eq.sub.in.good.domain.general} holds.

\item If $\phis(y,r , \varepsilon/2)=0$, hence by \eqref{eq.bound.travel.dist}, the definition of $t^\star$ and \eqref{eq.bound.travel.dist1},
\begin{align*}
\left|\Psi(x,s,(u_i,x_i)_{1\leq i \leq n},u_{n+1})-s\right| 
&\leq |r-\phis(y,r,\delta_1)|
\end{align*}
with $s\in \ens{0}{\phis(y,r,\delta_1)}$ then $\Psi(x,s,(u_i,x_i)_{1\leq i \leq n},u_{n+1})\in \ens{0}{r}$ and \eqref{eq.sub.in.good.domain.general} holds.
\end{itemize}

3. If $x<y$ then in the same way, reaching $y$ by $x-y$ successive division events, we have 
\begin{align*}
\PP_{(x,s)}
	\Big(T_{y,r} \leq \varepsilon\Big)
&
	\geq
	e^{-(D+\barmu)\,L\frac{\varepsilon}{2}}\, \left(\frac{\mu(\bar s_L)}{\barmu}\right)^{y-x}
	\,P_b(x,y-x, t^\star)
\\
&
	\geq
	e^{-(D+\barmu)\,L\,\frac{\varepsilon}{2}}\, \left(\frac{\mu(\bar s_L)}{\barmu}\right)^{L-1}
	\,P_b(1,L-1, t^\star)\,.\qedhere
\end{align*}
\end{proof}

\subsection{Proof of Lemma~\ref{lem:LV}}
\label{sec:proof.lem.LV}

On $\{0\}\times (0,\Sin)$, we have $\mathcal{L}\tilde V=0=V$. So let us prove the result on $\NN^* \times (0,\bar s_1)$.
For convenience,  we consider the natural extension of $V$ to $x=0$ given by $V(0,s)= \log(\rho)^{-1} e^{\alpha s} + s^{-1} + (1+\theta)/(\bar s_1-s)^{p}$ for all $s\in (0,\bar s_1)$. As for all $(x,s)\in \NN^* \times (0,\bar s_1)$
\[
	\mathcal{L}\tilde V(x,s) = \mathcal{L}V(x,s) - D\,V(0,s)\,\1_{x=1}\leq \mathcal{L}V(x,s)
\]
and as $\tilde V=V$ on $\NN^* \times (0,\bar s_1)$, it is sufficient to prove that there exists $\eta>D$ and $\zeta>0$ such that, on $\NN^* \times (0,\bar s_1)$,
\[
	\mathcal{L}V \leq -\eta V + \zeta \psi\,.
\]

We will prove that there exists $\eta>D$ such that $\mathcal{L} V + \eta V$ is bounded from above on $\NN^* \times (0,\bar s_1)$. As $\psi \geq 1$ on $\NN^* \times (0,\bar s_1)$ it therefore implies the result.
To that end, let define, for all $(x,s)\in \NN \times (0,\bar s_1)$
\[
V_0:(x,s)\mapsto \log(\rho)^{-1}\,\rho^x e^{\alpha s},
\qquad V_1:(x,s)\mapsto s^{-1},
\qquad V_2:(x,s)\mapsto (1+\1_{x\leq 1}\theta) (\bar{s}_1-s)^{-p} 
\] 
so that $V=V_0+V_1+V_2$. By the linearity of $\mathcal{L}$, we then have $\mathcal{L}V=\mathcal{L}V_0+\mathcal{L}V_1+\mathcal{L}V_2$ on $\NN^* \times (0,\bar s_1)$, with for $(x,s)\in \NN^* \times (0,\bar s_1)$
\begin{align*}
\mathcal{L} V_0(x,s)
&=\left[[D(\Sin - s)- k \mu(s) x ]  \alpha + (\rho-1)\left(\mu(s)-\frac{D}{\rho}\right) x\right]\,V_0(x,s) \\
\mathcal{L}V_1(x,s) & = -\frac{D(\Sin - s)- k \mu(s) x }{s}\,V_1(x,s)\\
\mathcal{L} V_2(x,s)
&=\left[p \frac{D(\Sin - s)- k \mu(s)x }{\bar{s}_1-s} -\mu(s)\1_{x= 1}\,\frac{\theta}{1+\theta}+\,2\,D\,\theta\,\1_{x= 2}\right]\,V_2(x,s)
\end{align*}
We will prove that there exist $\eta>D$ such that $\mathcal{L} V_0 + \mathcal{L} V_1+ \eta (V_0+V_1)$ and $\mathcal{L} V_2 + \eta V_2$ are bounded from above on $\NN^* \times (0,\bar s_1)$.

\medskip
Let $\eta \in \mathbb{R}$ and let $0<\varepsilon< D\frac{\rho-1}{\rho}$. As $\alpha\geq \frac{\rho-1}{k}$ we have
\begin{align*}
(\mathcal{L} V_0 + \mathcal{L} V_1 + \eta (V_0+V_1))(x,s)
& \leq A(x, s)+B(x, s)
\end{align*}
with 
\begin{align*}
A(x, s)&:=\left[D\,\Sin\,  \alpha - \left(D\frac{\rho-1}{\rho}-\varepsilon\right) x + \eta\right]\,V_0(x,s),\\
B(x,s)&:=\left[-\frac{D(\Sin - s)- k \mu(s) x }{s}+ \eta-\varepsilon\,x\,\frac{V_0(x,s)}{V_1(x,s)}\right]\,V_1(x,s)\,.
\end{align*}
We easily check that $A$ is bounded on every set on the form $\llbracket 1,L\rrbracket \times (0,\bar s_1)$ with $L\geq 1$, moreover $\sup_{s\in(0,\bar s_1)}A(x, s)$ tends towards $-\infty$ when $x \to \infty$. Then $A$ is bounded from above.
In addition, from the expression of $V_0$ and $V_1$, $\frac{k \mu(s)}{s}-\varepsilon\,\frac{V_0(x,s)}{V_1(x,s)}\leq 0$ if $x \geq C+2\,\log(1/s)/\log(\rho)$, with $C:=\log\left(k\,\mu(\bar s_1)\log(\rho)/\varepsilon\right)/\log(\rho)$. Therefore, setting  $\bar \mu_1' = \sup_{s\in[0,\bar s_1]}\mu'(s)$, we obtain
\begin{align*}
B(x,s)&\leq \left[-\frac{D(\Sin - s)}{s}+ \eta+k \,\frac{\mu(s)}{s}\,\left|C+\frac{2}{\log(\rho)}\,\log\left(\frac{1}{s}\right)\right| \right]\,\frac{1}{s}\\
&\leq \left[-\frac{D(\Sin - s)}{s}+ \eta+k \,\bar \mu_1'\,|C|+\frac{2\,k \,\bar \mu_1'}{\log(\rho)}\,\left|\log\left(\frac{1}{s}\right)\right| \right]\,\frac{1}{s}
\end{align*}
The right member does not depend on $x$, is bounded on every set on the form $(r,\bar s_1)$ with $0<r<\bar s_1$, and tends towards $-\infty$ when $s\to 0$. Hence $B$ is bounded from above and $\mathcal{L} V_0 + \mathcal{L} V_1 + \eta (V_0+V_1)$ is bounded from above for every $\eta\in \mathbb{R}$.

\medskip

We easily check that $\mathcal{L} V_2 + \eta V_2$ is bounded on every set on the form $\mathbb{N^*} \times (0,r]$, with $0<r<\bar s_1$.
Moreover, for $x\geq 2$ and $\bar s_2<s<\bar s_1$,  we have 
\[
D(\Sin - s)- k \mu(s) x \leq D(\Sin - s)- 2\,k \mu(s) < D(\Sin -\bar{s}_2)- 2\,k\,\mu(\bar{s}_2)=0
\]
then
\begin{align*}
\sup_{x\geq 2}\frac{\mathcal{L} V_2(x,s)}{V_2(x,s)}
&\leq p \frac{D(\Sin - s)- 2\,k\,\mu(s) }{\bar{s}_1-s} + 2\,D\,\theta
\end{align*}
tends to $-\infty$ when $s\to \bar{s}_1$ then $\mathcal{L} V_2 + \eta V_2$ is bounded from above on $\mathbb{N^*}\setminus\{1\} \times (0,\bar s_1)$ for all $\eta\in\mathbb{R}$. For $x =1$, \eqref{hyp:theta_p} leads to 
\begin{align*}
\lim_{s \to \bar{s}_1} \frac{\mathcal{L} V_2(1,s)}{V_2(1,s)}
&= \lim_{s \to \bar{s}_1} \left[\frac{p\, [D(\Sin - s)- k \mu(s)] }{\bar{s}_1-s} - \frac{\theta \mu(s)}{1+\theta}\right]\\
&= \lim_{s \to \bar{s}_1} \frac{p\, [D(\bar s_1 - s)+ k (\mu(\bar s_1)-\mu(s))}{\bar{s}_1-s} - \frac{\theta \mu(\bar s_1)}{1+\theta}\\
& = p[D+ k \,\mu'(\bar s_1)]- \frac{\theta \mu(\bar s_1)}{1+\theta}\\
& <-D
\end{align*}
It follows that $\mathcal{L} V_2 + \eta\,V_2$ is bounded from above for all $0<\eta<-\lim_{s \to \bar s_1} \mathcal{L} V_2(1,s)/V_2(1,s)$.
Therefore Lemma~\ref{lem:LV} holds and we can choose any $\eta\in (D, -\lim_{s \to \bar s_1} \mathcal{L} V_2(1,s)/V_2(1,s))$.

Note that relaxing the assumptions as in Remark~\ref{remark:mu.lipschitz}, the limit above does not necessary exist. However we can bound from above $\limsup_{s \to \bar s_1} \mathcal{L} V_2(1,s)/V_2(1,s)$ by $-D$ replacing $\mu'(\bar s_1)$ by $\klip$ in \eqref{hyp:theta_p}. In the same way, in the upper bound of $B$, $\bar \mu_1'$ can be replaced by a local Lipschitz constant of $\mu$ in the neighborhood of 0 when $s$ tends towards 0.

\subsection{Proof of Lemma~\ref{lem:bertrand}}
\label{sec:lem.bertrand}

As $s_1\mapsto \PP_{(y,r)}\left((X_{\tau},S_{\tau})\in \{x\}\times [s_0,s_1]\right)$ is increasing, we assume, without loss of generality, that $s_1\leq s_K$.
In the same way as the proof of Proposition~\ref{prop.partout.a.partout}, we prove that the probability $\PP_{(y,r)}\left((X_\tau,S_\tau)\in\{x\}\times [s_0,s_1]\right)$ is bounded from below by the probability that the process $(X_t,S_t)_t$
\begin{enumerate}
\item reaches $B(L,[\tilde s_0,\tilde s_1])$ before $\tau-\varepsilon$ (\textit{i.e.} $T_{L,[\tilde s_0,\tilde s_1]}\leq \tau-\varepsilon$);
\item stays in $\llbracket 1, L \rrbracket \times [\tilde s_0,\tilde s_1]$ during the time interval $[T_{L,[\tilde s_0,\tilde s_1]},\, \tau-\varepsilon]$;
\item reaches $\{x\}\times [s_0,s_1]$ in the time interval $[\tau-\varepsilon,\,\tau]$ and stays in this set until $\tau$;
\end{enumerate}
that is
\begin{align}
\nonumber
& \PP_{(y,r)}\left((X_\tau,S_\tau)\in\{x\}\times [s_0,s_1]\right)
\\
\nonumber
& \qquad \geq \PP_{(y,r)}\left(T_{L,[\tilde s_0,\tilde s_1]}\leq \tau-\varepsilon\right)\,
\\
\nonumber
& \qquad \qquad
		\times \PP_{(y,r)}
			\left((X_t,S_t)\in \llbracket 1, L \rrbracket \times [\tilde s_0,\tilde s_1], 
				\forall t\in [T_{L,[\tilde s_0,\tilde s_1]},\tau-\varepsilon]\ \Big| \
				 T_{L,[\tilde s_0,\tilde s_1]}\leq \tau-\varepsilon \right)
\\ \label{lemma.reach.small.substrate.decompo}
& \qquad \qquad
		\times \PP_{(y,r)}
		 \left( (X_\tau,S_\tau)\in\{x\}\times [s_0,s_1] \ \Big| \ E \right)\,,
\end{align}
where
\[
E:= \{T_{L,[\tilde s_0,\tilde s_1]}\leq \tau-\varepsilon\} \cap 
		\left\{(X_t,S_t)\in \llbracket 1, L \rrbracket \times [\tilde s_0,\tilde s_1], 
				\forall t\in [T_{L,[\tilde s_0,\tilde s_1]},\tau-\varepsilon]\right\}
\]
with $L$, $\tilde s_0$, $\tilde s_1$ and $\varepsilon$ well chosen so that we can bound from below the three probabilities in the right member of \eqref{lemma.reach.small.substrate.decompo}. More precisely, we will choose $L$ sufficiently large such that the substrate concentration $\frac{s_1+s_0}{2}$ can be reached from $S_K$ in a time less than $\tau$ with $L$ individuals; and $\tilde s_0$,  $\tilde s_1$ and $\varepsilon$ will be chosen such that $[\tilde s_0, \tilde s_1]\subset [s_0,s_1]$ is centered in $\frac{s_1+s_0}{2}$ and such that the process can not exit from $[s_0,s_1]$ is a time less that $\varepsilon$ with a bacterial population in $\setunL$.

\medskip

From Lemma~\ref{lem:defsbar}, there exists $L_{s_0} \geq x\wedge \max_{(y,r)\in K}y$ such that $\bar s_\ell <\frac{s_1+s_0}{2}$ for all $\ell \geq L_{s_0}$.
Moreover, for $\ell \geq L_{s_0}$, as $\bar s_\ell<\frac{s_1+s_0}{2}<S_K$, then $\phit\left(\ell,S_K, \frac{s_1+s_0}{2}\right)<+\infty$ and
\begin{align*}
\frac{s_1+s_0}{2} 
	&= S_K + \int_0^{\phit\left(\ell,S_K, \frac{s_1+s_0}{2}\right)} 
		\left[D(\Sin-\phi(\ell,S_K,u))-k\,\mu(\phi(\ell,S_K,u))\,\ell\right]\,\dif u
\\
	&\leq
	S_K + 
		\left[D\left(\Sin-\frac{s_1+s_0}{2}\right)-k\,\mu\left(\frac{s_1+s_0}{2}\right)\,\ell\right]\,\phit\left(\ell,S_K, \frac{s_1+s_0}{2}\right)
\end{align*}
then
\[
\phit\left(\ell,S_K, \frac{s_1+s_0}{2}\right)
	\leq \frac{S_K-\frac{s_1+s_0}{2}}{k\,\mu\left(\frac{s_1+s_0}{2}\right)\,\ell-D\left(\Sin-\frac{s_1+s_0}{2}\right)}\,.
\]
The right term in the previous inequality tends to 0 when $\ell \to \infty$, we can then choose $L \geq L_{s_0}$ such that $\phit\left(L,S_K, \frac{s_1+s_0}{2}\right) <\tau$.

Let set $0<\varepsilon<\min\left\{\tau-\phit(L,S_K, \frac{s_1+s_0}{2}),\,\frac{s_1-s_0}{2\,\max\{D\,\Sin,\, k\,\barmu\, L\}}\right\}$ and let us define $\tilde s_0 = s_0+\varepsilon\,\max\{D\,\Sin,\, k\,\barmu\, L\}$ and $\tilde s_1 = s_1-\varepsilon\,\max\{D\,\Sin,\, k\,\barmu\, L\}$, then $\tilde s_0<\tilde s_1$ and $[\tilde s_0, \tilde s_1] \subset [s_0,s_1]$.

From Corollary~\ref{cor:monotonyphit},  $\tau-\varepsilon>\phit(L,S_K, \frac{s_1+s_0}{2})>\phit(L,S_K, \tilde s_1)>\phit(L,r, \tilde s_1)$, for all $(y,r)\in K$. Then from Lemma~\ref{lemma.reach.T.Smin.Smax}-\ref{lemma.reach.T.Smin.Smax.s_sup_Smax} and Remark~\ref{rk.Pd.Pb.inc}, 
\begin{align}
\label{lemma.reach.small.substrate.proba1}
\PP_{(y,r)}\left(T_{L,[\tilde s_0,\tilde s_1]}\leq \tau-\varepsilon\right)
	&\geq
	e^{-(D+\barmu)\,(\tau-\varepsilon-\delta)\,L}\,
	\left(\frac{\mu(\bar s_L)}{\barmu}\right)^{L-1}\,P_b(1,L-1, \delta)=:C_1
\end{align}
with $\delta:=(\tau-\varepsilon-\phit\left(L, S_K,\tilde s_1\right))\,\frac{D\,\left|\bar s_L-s_K\right|}{D\,\left|\bar s_{L}-s_K\right|+\max\{D\,\Sin,\,k\,\barmu\,L\}}$.

Moreover, from Lemma~\ref{lemma.stay.in.S.general}, there exists $C_2>0$ such that for all $(z,s)\in B(L, [\tilde s_0,\tilde s_1])$, 
\[
\PP_{(z,s)}\left((X_t,S_t) \in \llbracket 1,L\rrbracket \times [\tilde s_0, \, \tilde s_1] ,\,
				\forall t\in[0,\tau-\varepsilon]\right)\geq C_2,\,
\]
therefore, by Markov Property
\begin{align}
\label{lemma.reach.small.substrate.proba2}
\PP_{(y,r)}
			\left((X_t,S_t)\in \llbracket 1,  L \rrbracket \times [\tilde s_0,\tilde s_1], 
				\forall t\in [T_{ L,[\tilde s_0,\tilde s_1]},\tau-\varepsilon] \ | \ T_{ L,[\tilde s_0,\tilde s_1]}\leq \tau-\varepsilon\right)
	\geq C_2\,.
\end{align}

In addition, on the event $\left\{X_u \in \llbracket 1, L \rrbracket,\, \forall u\in[0,\varepsilon]\right\}$,
\begin{align*}
\left|S_\varepsilon-S_0 \right| 
	&= \left|\int_0^\varepsilon \left(D\,(\Sin-S_u)-k\,\mu(S_u)\,X_u \right)\,\dif u\right|
	\leq \varepsilon\,\max\{D\,\Sin,\, k\,\barmu\,  L\}
\end{align*}
then, as $\tilde s_0-s_0=s_1-\tilde s_1=\varepsilon\,\max\{D\,\Sin,\, k\,\barmu\,  L\}$, for all $(z,s)\in \llbracket 1, L \rrbracket \times [\tilde s_0,\tilde s_1])$, 
\[
	\PP_{(z,s)}\left(S_\varepsilon \in [s_0,s_1] \, | \, 
		X_u \in \llbracket 1, L \rrbracket,\, \forall u\in[0,\varepsilon]\right) =1\,.
\]
Therefore,  bounding from below the probability by the probability that, in addition, there is no event if $z=x$, there are $z-x$ washouts is $z>x$ and there are $x-z$ divisions if $z<x$ in the time interval $[0,\varepsilon]$ and no more event, then
\begin{align*}
&\PP_{(z,s)}
		 \left( (X_{\varepsilon},S_{\varepsilon})\in\{x\}\times [s_0,s_1]\right)
\\
	& \qquad \geq 
	\PP_{(z,s)}
		 \left(T_1>\varepsilon \right)\,\1_{z=x}
\\
	& \qquad\quad 
	+ \PP_{(z,s)}
		 \left(\bigcap_{i=1}^{z-x}\{T_i\ \leq \varepsilon\} \cap \{ X_{T_i} = z-i \} \cap \{T_{z-x+1}>\varepsilon\}\right)\,\1_{z>x}
\\
	& \qquad\quad 
	+\PP_{(z,s)}
		 \left(\bigcap_{i=1}^{x-z}\{T_i\ \leq \varepsilon\} \cap \{ X_{T_i} = z+i \} \cap \{T_{x-z+1}>\varepsilon\}\right)\,\1_{z<x}\,.
\end{align*}
For all $u_1<\dots < u_{|z-x|}\leq\varepsilon$, from the Markov Property and Lemma~\ref{prop.minore.proba.evt}, if $z>x$
\begin{align*}
\PP_{(z,s)}\left(T_{|z-x|+1}>\varepsilon \ |\ \mathcal{E}_D(u_1,\,\dots,u_{|z-x|}\right)
 &=  \PP_{(x,\Psi(z,s,(u_i,z-i)_{1\leq i \leq |z-x|-1},u_{|z-x|}))}\left(T_1>\varepsilon-u_{|z-x|}\right)
\\ &			\geq e^{-(D+\barmu)\,x\,\varepsilon}
\end{align*}
and if $z<x$
\begin{align*}
\PP_{(z,s)}\left(T_{|z-x|+1}>\varepsilon \ |\ \mathcal{E}_B(u_1,\,\dots,u_{|z-x|}\right) 
	&= \PP_{(x,\Psi(z,s,(u_i,z+i)_{1\leq i \leq |z-x|-1},u_{|z-x|}))}\left(T_1>\varepsilon-u_{|z-x|}\right)
\\	&\geq e^{-(D+\barmu)\,x\,\varepsilon}
\end{align*}
then, still from Lemma~\ref{prop.minore.proba.evt},
\begin{align*}
\PP_{(z,s)}
		 \left( (X_{\varepsilon},S_{\varepsilon})\in\{x\}\times [s_0,s_1]\right)
& \geq 
	e^{-(D+\barmu)\,x\,\varepsilon}\,\1_{z=x}
\\
	& \quad 
	+ P_d( L, L-1,\varepsilon)\,e^{-(D+\barmu)\,x\,\varepsilon}\,\1_{z>x}
\\
	& \quad 
	+\left(\frac{\mu(\bar s_L)}{\barmu}\right)^{ L-1}\,P_b(1, L-1, \varepsilon)\,
	e^{-(D+\barmu)\,x\,\varepsilon}\,\1_{z<x}
\\
	&\geq C_3
\end{align*}
with $C_3:=e^{-(D+\barmu)\,x\,\varepsilon}\,
	\min\left\{P_d( L, L-1,\varepsilon)\,e^{-(D+\barmu)\,x\,\varepsilon},\,
	\left(\frac{\mu(\bar s_L)}{\barmu}\right)^{ L-1}\,P_b(1, L-1, \varepsilon)\right\}$.
Then by Markov Property,
\begin{align}
\label{lemma.reach.small.substrate.proba3}
\PP_{(y,r)}
		 \left( (X_\tau,S_\tau)\in\{x\}\times [s_0,s_1] \ | \ E \right)
	\geq C_3\,.
\end{align}

\medskip

Finally, from \eqref{lemma.reach.small.substrate.decompo}, \eqref{lemma.reach.small.substrate.proba1}, \eqref{lemma.reach.small.substrate.proba2} and \eqref{lemma.reach.small.substrate.proba3}
\[
\PP_{(y,r)}\left((X_\tau,S_\tau)\in\{x\}\times [s_0,s_1]\right)
\geq C_1\,C_2\,C_3=:\epsilon_0>0\,.
\]

\subsection{ Lemma~\ref{lem:lyap2} and its proof}

\begin{lemma}
\label{lem:lyap2} 
There exists $\varepsilon>0$ and $A,C, \beta>0$ such that for all 
$(x,s) \in \mathbb{N}^* \times [\bar{s_1}, +\infty)$
\[
\mathbb{E}_{(x,s)}\left[e^{(D+C)(T_\varepsilon\wedge \Text)} \right] \leq A  e^{\beta s} \qquad
\text{and} \qquad
\PP_{(x,s)}(T_\varepsilon<\infty)>0.
\]
with $T_\varepsilon=T_{ \mathbb{N}^*\times (0,\bar s_1-\varepsilon]}:=\inf\{t\geq 0, \, (X_t,S_t)\in \mathbb{N}^*\times (0,\bar s_1-\varepsilon]\}$ the hitting time of $\mathbb{N}^*\times (0,\bar s_1-\varepsilon]$.
\end{lemma}

\begin{proof}
Let $g$ be defined for $(x,s)\in \mathbb{N}\times \mathbb{R}_+$ by
\[
g(x,s) = \left( \1_{x\geq 2} + (1+\delta_1) \1_{x=1} +\delta_0 \1_{x=0} \right) e^{\beta s}\,,
\]
with $\delta_0$, $\delta_1$ and $\beta$ positive constants (fixed below). Then  $g\geq \min\{1,\delta_0\}$ and
\begin{align*}
\frac{\mathcal{L} g(x,s)}{g(x,s)} + D
&=
\begin{cases}
\left[D(S_{in}-s) - k \mu(s)\right] \beta - \mu(s)\, \frac{\delta_1}{1+\delta_1}  + D\, \frac{\delta_0}{1+\delta_1}&  \text{if } x = 1,\\
\left[D(S_{in}-s) -2\, k \mu(s)\right] \beta + D\,(1+\delta_1) & \text{if } x=2,\\
\left[D(S_{in}-s) - k \mu(s)\,x\right] \beta +D & \text{if } x\geq 3.
\end{cases}
\end{align*}
We can choose $\delta_0$, $\delta_1$, $\beta$ and $\varepsilon>0$ such that 
\begin{align}
\label{eq:lg.leq.minus_C_plus_D_g}
\mathcal{L} g(x,s) \leq -(C+ D) g(x,s), \qquad \forall (x,s)\in \mathbb{N}^*\times [\bar s_1-\varepsilon,+\infty)\,,
\end{align}
with $C>0$.
Inded, let $\delta_1>0$ and let $\bar \varepsilon\in (0,\bar s_1 - \bar s_2)$ be fixed. From Lemmas \ref{lem:defsbar} and \ref{lem:rapprochement}, we have $D\,(\Sin-\bar s_1+\bar \varepsilon)- 2\,k\,\mu(\bar s_1-\bar \varepsilon)<0$, then we can choose $\beta>0$ sufficiently large such that 
\[
C_1 :=\left[D(S_{in}-\bar s_1+\bar \varepsilon) -2\, k \mu(\bar s_1-\bar \varepsilon)\right] \beta + D\,(1+\delta_1) <0\,.
\]
Moreover, $\left[D(S_{in}-\bar s_1+\varepsilon) - k \mu(\bar s_1-\varepsilon)\right] \beta - \mu(\bar s_1-\varepsilon)\, \frac{\delta_1}{1+\delta_1} \longrightarrow_{\varepsilon\to 0}- \mu(\bar s_1)\, \frac{\delta_1}{1+\delta_1}<0$, then we can choose $\varepsilon \in (0, \bar \varepsilon)$ and $\delta_0>0$ sufficiently small such that 
\[
C_2 := \left[D(S_{in}-\bar s_1+\varepsilon) - k \mu(\bar s_1-\varepsilon)\right] \beta - \mu(\bar s_1-\varepsilon)\, \frac{\delta_1}{1+\delta_1}  + D\, \frac{\delta_0}{1+\delta_1}<0\,.
\]
Setting such $\beta$, $\delta_0$ and $\varepsilon$, then for all $x\geq 1$ and $s \geq \bar s_1-\varepsilon$ we have  
\begin{align*}
\frac{\mathcal{L} g(x,s)}{g(x,s)} + D
&\leq
\begin{cases}
C_2 &  \text{if } x = 1\\
C_1 & \text{if } x\geq 2\\
\end{cases}
\end{align*}
and \eqref{eq:lg.leq.minus_C_plus_D_g} holds with $C:=-(C_1\vee C_2)>0$.

For any initial condition $(x,s) \in \mathbb{N}^* \times [\bar s_1, +\infty)$, we have $(X_u,S_u)\in \mathbb{N}^*\times [\bar s_1-\varepsilon,+\infty)$ for all $u<T_\varepsilon\wedge \Text$, then 
by standard arguments using the Dynkin's formula and \eqref{eq:lg.leq.minus_C_plus_D_g} (see for instance \cite[Theorem 2.1]{MTIII} and its proof) $\left(g\left(X_{t\wedge T_\varepsilon\wedge \Text},S_{t\wedge T_\varepsilon\wedge \Text}\right)e^{(C+D)\left(t\wedge T_\varepsilon\wedge \Text\right)}\right)_t$ is a nonnegative super-martingale. Then, as by \eqref{eq.def.Text} $T_\varepsilon\wedge \Text$ is a.s. finite,  by classical arguments (stopping time theorem applied to truncated stopping times and Fatou's lemma)
\[
\min(1,\delta_0)\,\mathbb{E}_{(x,s)}\left[ e^{(C+D) (T_\varepsilon\wedge \Text)}  \right] \leq \mathbb{E}_{(x,s)}\left[ g(X_{T_\varepsilon\wedge \Text},S_{T_\varepsilon\wedge \Text}) e^{(C+D) (T_\varepsilon\wedge \Text)} \right] \leq g(x,s)
\]
which leads to the first part of the lemma.

\medskip

We can show that the upper bound of Lemma~\ref{lemme.control.phi_moins_1} holds even if $s_0\geq \bar s_1$. Then from Lemma~\ref{lem:defsbar}, for all $\ell \geq 2$ and $s\geq \bar s_1-\varepsilon>\bar s_2$, 
\[
	\phit(\ell,s,\bar s_1-\varepsilon) \leq \frac{s-\bar s_1+\varepsilon}{D\,(\bar s_1-\varepsilon-\bar s_\ell)}\leq \frac{s-\bar s_1+\varepsilon}{D\,(\bar s_1-\varepsilon-\bar s_2)}=:t_{\bar s_1-\varepsilon}\,.
\]
Then, if $x\geq 2$,  from Lemma~\ref{prop.minore.proba.evt}
\begin{align}
\PP_{(x,s)}(T_\varepsilon<\infty) 
		& 
	 \geq \PP_{(x,s)}(t_{\bar s_1-\varepsilon}<T_1) \geq e^{-(D+\mu(s))\,x\,t_{\bar s_1-\varepsilon}}>0\,.
\label{eq:proba.T.leq.Text.positive}
\end{align}
If $x=1$, then for all $\delta>0$, from Lemma~\ref{prop.minore.proba.evt} and the Markov property
\begin{align*}
\PP_{(1,s)}(T_\varepsilon<\infty) 
		& \geq \PP_{(1,s)}(\{T_1\leq \delta\} \cap \{X_{T_1}=2\} \cap \{T_\varepsilon<\infty\})
\\
	&\geq \int_0^{\delta}
	\mu\left(\phi(1,s,u)\right)\,
	e^{-(D+\mu(s))\,u}\,	
	\PP_{(1,s)}
	\Big(T_\varepsilon<\infty \ \Big| \{X_u=2\}\cap\{T_1=u) \Big)\,\dif u
\\
	&= \int_0^{\delta}
	\mu\left(\phi(1,s,u)\right)\,
	e^{-(D+\mu(s))\,u}\,	
	\PP_{(2,\phi(1,s,u))} (T_\varepsilon<\infty)\,\dif u\,.
\end{align*}
From Lemma~\ref{lem:rapprochement}, $\phi(1,s,u)>\bar s_1>\bar s_1-\varepsilon$, then by \eqref{eq:proba.T.leq.Text.positive}, $\PP_{(1,s)}(T_\varepsilon<\infty) >0$.
\end{proof}

\section{Theorems of \cite{BCGM} and \cite{CV20}}

We recall in this section the theorems of \cite{BCGM} and \cite{CV20} which establish the convergence towards a unique quasi-stationary distribution. 

\medskip

Let $(X_t)_{t\geq 0}$ be a càd-làg Markov process on the state space $\mathcal{X}\cup \{\partial\}$, where $\mathcal{X}$ is a measurable space and $\partial$ is an absorbing state.  
Let $V:\mathcal{X} \to (0,\infty)$ a measurable function. 
We assume that for any $t>0$, there exists $C_t>0$ such that $\mathbb{E}_x[V(X_t)\1_{X_t\notin \partial} ]\leq C_t\,V(x)$ for any $x\in \mathcal{X}$. We denote by $\mathcal{B}(V)$ the space of measurable functions $f:\mathcal{X}\to\mathbb{R}$ such that $\sup_{x\in \mathcal{X}} \frac{|f(x)|}{V(x)}<\infty$ and $\mathcal{B}_+(V)$ its positive cone. 
Let $(M_t)_{t\geq 0}$ the semigroup defined for any measurable function $f\in \mathcal{B}(V)$ and any $x\in\mathcal{X}$ by
\[
	M_t f(x) := \mathbb{E}_x[f(X_t)\,\1_{X_t\notin \partial}]
\]
and let define the dual action, for any $\xi\in \mathcal{P}(V)$ , with $\mathcal{P}(V)$ the set of probability measures that integrate $V$, by
\[
	\xi M_t f := \mathbb{E}_\xi[f(X_t)\,\1_{X_t\notin \partial}] =  \int_\mathcal{X} M_tf(x)\xi(\dif x).
\]

\begin{hypothesis}[Assumption~A of \cite{BCGM}]
\label{hyp:BCGM}
Let $\psi:\mathcal{X} \to (0,\infty)$ such that $\psi\leq V$. There exist $\tau>0$, $\beta>\alpha>0$, $\theta>0$, $(c,d)\in (0,1]^2$, $K\subset \mathcal{X}$ and $\nu$ a probability measure on $\mathcal{X}$ supported by $K$ such that $\sup_K V/\psi <\infty$ and
\begin{itemize}
\item[(A1)]  $M_\tau V\leq \alpha\,V + \theta\,\1_K\,\psi$,
\item[(A2)]  $M_\tau \psi\geq \beta\,\psi$,
\item[(A3)]  $\inf_{x\in K} \frac{M_\tau(f\,\psi)(x)}{M_\tau \psi(x)}\geq c\,\nu(f)$ \quad for all $f\in \mathcal{B}_+(V/\psi)$,
\item[(A4)] $\nu\left( \frac{M_{n\tau}\psi}{\psi}\right) \geq d \sup_{x\in K} \frac{M_{n\tau}\psi(x)}{\psi(x)}$ \quad for all positive integers $n$.
\end{itemize}
\end{hypothesis}

\begin{theorem}[Theorem~5.1.  of \cite{BCGM}]
\label{th:BCGM}
Assume that $(M_t)_{t\geq 0}$ satisfies Assumption~\ref{hyp:BCGM} with $\inf_\mathcal{X}V>0$. Then, there exist a unique quasi-stationary distribution $\pi$ such that $\pi\in\mathcal{P}(V)$, and $\lambda_0>0$, $h\in\mathcal{B}_+(V)$, $C,\omega>0$ such that for all $\xi\in\mathcal{P}(V)$ and $t\geq 0$,
\[
\lVert e^{\lambda_0 \,t}\,\PP_\xi(X_t\in \cdot)-\xi(h)\,\pi\rVert_{TV}\leq C\,\xi(V)\,e^{-\omega\,t}
\]
and
\[
\lVert\PP_\xi(X_t\in . | X_t\neq \partial)-\pi\rVert_{TV}\leq C\,\frac{\xi(V)}{\xi(h)}\,e^{-\omega\,t},
\]
with $\lVert.\rVert_{TV}$ the total variation norm on $\mathcal{X}$.
\end{theorem}

\begin{hypothesis}[Condition~(G) (including Remark~2.2) of \cite{CV20}]
\label{hyp:CV20}
There exist positive real constants $\theta_1, \theta_2, c_1, c_2, c_3$, an integer $n_1\geq 1$, a function $\psi:\mathcal{X} \to \mathbb{R}_+$ and a probability measure $\nu$ on a measurable subset $K$ of $\mathcal{X}$ such that
\begin{itemize}
\item[(G1)] (Local Dobrushin coefficient). For all $x \in K$ and all measurable $A\subset K$,
\[
	P_{n_1}(V\,\1_A)(x)\geq c_1\,\nu(A)\,V(x).
\]
\item[(G2)] (Global Lyapunov criterion). We have $\theta_1<\theta_2$ and
\[
	\inf_{x\in K} \frac{\psi(x)}{V(x)}>0, \qquad \sup_{x\in \mathcal{X}} \frac{\psi(x)}{V(x)}\leq 1,
\]
\[
P_1 V(x) \leq \theta_1 \, V(x) + c_2\,\1_K(x)\,V(x),\, \quad \forall x\in \mathcal{X}
\]
\[
P_1 \psi(x) \geq \theta_2 \, \psi(x),\, \quad \forall x\in \mathcal{X}.
\]
\item[(G3)] (Local Harnack inequality). We have 
\[
	\sup_{n\in \mathbb{Z}_+}\frac{\sup_{y\in K}P_n \psi(y)/\psi(y)}{\inf_{y\in K}P_n \psi(y)/\psi(y)} \leq c_3
\]

\item[(G4)] (Aperiodicity). For all $x\in K$, there exists $n_4(x)$ such that for all $n\geq n_4(x)$,
\[
	P_n(\1_K\,V)(x)>0.
\]
\end{itemize}
\end{hypothesis}

\begin{theorem}[Corollary 2.4. of \cite{CV20}]
\label{th:CV20}
Assume that there exists $t_0>0$ such that $(P_{n})_{n\in\mathbb{N}}:=(M_{nt_0})_{n\in\mathbb{N}}$ satisfies Assumption~\ref{hyp:CV20},  $\left(\frac{M_tV}{V}\right)_{t\in[0,t_0]}$ is upper bounded by a constant $\bar c>0$ and  $\left(\frac{M_t\psi}{\psi}\right)_{t\in[0,t_0]}$ is lower bounded by a constant $\underline{c}>0$. Then there exist a positive measure $\nu_P$ on $\mathcal{X}$ such that $\nu_P(V)=1$ and $\nu_P(\psi)>0$, and some constants $C''>0$ and $\gamma>0$ such that, for all measurable functions $f:\mathcal{X}\to \mathbb{R}$ satisfying $|f|\leq V$ and all positive measure $\xi$ on $\mathcal{X}$ such that $\xi(V)<\infty$ and $\xi(\psi)>0$,
\[
  \left| \frac{\xi M_t f}{\xi M_t V} -\nu_P(f)\right| \leq C''\,e^{-\gamma\,t} \frac{\xi(V)}{\xi(\psi)}, \quad \forall t\geq 0.
\]
In addition, there exists $\lambda_0\in \mathbb{R}$ such that $\nu_P M_t = e^{\lambda_0\,t}\,\nu_P$ for all $t\geq 0$, and $e^{-\lambda_0\,t}\,M_t V$ converges uniformly and exponentially toward $\eta_P$ in $\mathcal{B}(V)$ when $t\to \infty$. 
Moreover, there exist some constants $C'''>0$ and $\gamma'>0$ such that, for all measurable functions $f:\mathcal{X}\to \mathbb{R}$ satisfying $|f|\leq V$ and all positive measures $\xi$ on $\mathcal{X}$ such that $\xi(V)<+\infty$,
\begin{align}
\label{eq:cor2.4.CV20.eta}
	\left| e^{-\lambda_0\,t}\,\xi M_t f-\xi(\eta_P)\,\nu_P(f) \right| \leq C'''\, e^{-\gamma'\,t}\, \xi(V), \quad \forall t\geq 0\,.
\end{align}
\end{theorem}

\section*{Acknowledgements}
This work was partially supported by the Chaire ``Mod\'elisation Math\'ematique et Biodiversit\'e'' of VEOLIA Environment, \'Ecole Polytechnique, Mus\'eum National d'Histoire Naturelle and Fondation X and by the ANR project NOLO,  ANR-20-CE40-0015,  funded by the French Ministry of Research.

\bibliographystyle{alpha}

\newcommand{\etalchar}[1]{$^{#1}$}

\end{document}